\documentclass[12pt]{article}
\usepackage[T2A]{fontenc}
\usepackage[utf8]{inputenc}
\usepackage[english,russian]{babel}

\usepackage[textwidth=16.5cm, textheight=24.8cm, left=1.5cm, right=1.5cm]{geometry}

\usepackage{amsmath,amssymb,amsthm,url}
\usepackage{dsfont,textcomp,graphicx,overpic}
\usepackage[usenames,dvipsnames]{color}
\usepackage{hyperref}
\hypersetup{
   colorlinks   = true, 
   urlcolor     = blue, 
   linkcolor    = blue, 
   citecolor   = red 
}
\usepackage{tikz}
\usetikzlibrary{positioning, calc, intersections, through, arrows}
\usetikzlibrary{graphs}
\usetikzlibrary{graphs.standard}
\usetikzlibrary{patterns}
\graphicspath{{figs/}}

\newtheoremstyle{mydefinition}
  {\medskipamount}
  {\medskipamount}
  {\normalfont}
  {\parindent}
  {\bfseries}
  {.}
  { }
  {}

\newcommand{\diag}{\operatorname{diag}}

\def\t{\widetilde}
\def\R{{\mathbb R}} \def\Z{{\mathbb Z}}

\newcommand{\presonly}[1]{}

\newcommand{\algor}[1]{}
\newcommand{\invadraw}[1]{}
 
\newcommand{\talkonly}[1]{}
\newcommand{\talkno}[1]{#1}

\newcommand{\short}[1]{#1}
 
\renewcommand{\ge}{\geqslant}
\renewcommand{\geq}{\geqslant}
\renewcommand{\le}{\leqslant}

\renewcommand{\emptyset}{\varnothing}

\renewcommand{\Box}{\mathbin{\text{\scalebox{.84}{$\square$}}}}

\DeclareMathOperator{\id}{id}

\theoremstyle{plain}
\newtheorem{theorem}{Теорема}[subsection]
    \newtheorem{lemma}[theorem]{Лемма}
    
    \newtheorem{proposition}[theorem]{Утверждение}
    \newtheorem{conjecture}[theorem]{Гипотеза}

\theoremstyle{mydefinition}
\newtheorem{pr}[theorem]{Задача}
\newtheorem{remark}[theorem]{Замечание}
\newtheorem{example}[theorem]{Пример}

\begin{document}









\talkonly{\huge}
 
\title{Циклы в графах и в гиперграфах: \\ в направлении теории гомологий}

\author{А. Мирошников, О. Никитенко, А. Скопенков
\talkno{\thanks{\emph{А. Мирошников, А. Скопенков:} Московский физико-технический институт, mirosh.av@mail.ru.
\newline
\emph{О. Никитенко:} Алтайский технический университет, nov142857@gmail.com.
\newline
\emph{А. Скопенков:} Независимый московский университет, \url{https://users.mccme.ru/skopenko/}, skopenko@mccme.ru. 
\newline
При написании этого текста использованы материалы \cite{ADN+} Летней Конференции Турнира Городов и курсов МФТИ.
Мы благодарны Э. Алкину и А. Воропаеву за совместную работу над теми материалами, А. Воропаеву за редактирование решений к \S$2$ той версии, С. Дженжеру за совместную работу над предыдущими версиями этого текста, Э. Алкину, И. Богданову, Р. Карасеву, Д. Мусатову, А. Рябичеву и О. Стырту за полезные обсуждения, а также А. Рябичеву за подготовку некоторых рисунков.
}}
}

\date{}

\maketitle

\talkno{\tableofcontents}

\subsection*{Предисловие: неформальное описание темы и основных результатов}\label{s:intrnf}
\addcontentsline{toc}{subsection}{Предисловие: неформальное описание темы и основных результатов}

Этот обзор демонстрирует важные идеи \emph{теории гомологий} (раздела \emph{алгебраической топологии}) на простом языке теории графов. 
Он доступен математикам и программистам, не знакомым с топологией.
Изучение большей части обзора не требует предварительных знаний (за пределами школьной программы), но требует некоторой математической культуры. 
Его разумно изучать \emph{перед} чтением учебников по теории гомологий. 
Кому-то из читателей эти учебники и не понадобятся, будет достаточно изложенных здесь идей. 
А для тех, кто уже изучал теорию гомологий по учебникам, рассмотрение и применение ее частных случаев обычно оказывается интересным и нетривиальным.  

Несмотря на элементарность, обзор подводит читателя к переднему краю науки, см. статьи \cite{Va92, FH10, MS17, SS23, Dz25} и ссылки в них.

Сначала мы введем и обсудим основное понятие, естественно возникающее в теории графов, а потом опишем основные результаты.

Назовем \textbf{$1$-циклом} в графе множество $C$ ребер такое, что каждая вершина принадлежит четному числу ребер из $C$.
\talkno{(Более точное название: одномерный симплициальный цикл по модулю~$2$.)}
Сравните со следующей теоремой Эйлера: \emph{в связном графе существует цикл, проходящий по каждому ребру ровно один раз, тогда и только тогда, когда каждая вершина этого графа принадлежит четному числу ребер}. 
См. также \S\ref{s:congra}. 
Версией целочисленного обобщения этого понятия является понятие допустимого потока (feasible flow) \cite{Fl}.

Рассмотрим примеры.
Множество ребер цикла в графе, не проходящего ни по одному ребру дважды, является $1$-циклом, но не наоборот.
Пустое подмножество ребер также является $1$-циклом.
В графе-цикле два $1$-цикла, в~графе <<восьмерка>> "--- четыре, а~в~полном графе на 4 вершинах "--- восемь (последнее не так очевидно).

Понятие $1$-цикла и его обобщения имеют многочисленные приложения.
Однако привести здесь даже простейшие из них не представляется целесообразным. 
Это связано с тем, что приложения более сложны, чем само понятие $1$-цикла. 
Ведь $1$-циклы возникают как средство изучения более сложных объектов (<<препятствий>>, представляемых <<коциклами>>, ср. с \cite[\S9]{Sk}). 
О наиболее простых приложениях см., например, \cite[\S\S 4.11, 6, 9]{Sk20}, \cite[\S\S 1, 8, 9, 10]{Sk}. 
По этим книгам читатель сможет познакомиться не только с приложениями, но и с основами теории гомологий, обобщающей идеи, приведенные здесь и при обсуждении приложений.

\smallskip
Этот текст начинается со следующих базовых утверждений. 
 Для данного графа мы найдем

$\bullet$ количество всех 1-циклов в нем (утверждение \ref{pr:orinumber2}.a, к нему подводит утверждение~\ref{l:gencom1}.a); 
 
$\bullet$ <<небольшое>> количество $1$-циклов в нем таких, что любой $1$-цикл является \emph{суммой} нескольких из них (утверждения~\ref{l:gencom1}.cde, \ref{pr:knn1}.c, \ref{l:tkn} и лемма  \ref{lem:sumcyc});

$\bullet$ соотношения между <<простейшими>> 1-циклами (утверждения \ref{pr:rel1}.b и \ref{pr:knn2}.b).

\smallskip
Основные результаты, излагаемые в этом тексте "--- описание (и вычисление количества)
 
$\bullet$ симметричных $1$-циклов в графах с симметрией 
(утверждения \ref{prop:inv} и \ref{p:uni1s}); 




$\bullet$ $1$-циклов \emph{по модулю границ} и $2$-циклов в \emph{квадрате графа} 
(теоремы~\ref{stcycles2} и~\ref{t:kunneth} "---  это частные случаи знаменитой 
формулы Кюннета о <<гомологиях произведения>> 
\cite[11.8.3]{Sk20}, \cite[\S15.6]{FF89});  

$\bullet$ порождающих в множестве $2$-циклов во \emph{взрезанном квадрате графа} (теорема \ref{t:2cyc-sym-bases}).

Все выделенные выше курсивом понятия определены в тексте.
Последний из этих результатов связан с парадоксальным примером: различные поверхности допускают <<одинаковую>> склейку из одинакового числа квадратиков (замечание \ref{remark:2cycl-bij}.ab).

Имеются красивые задачи для исследования\talkno{ (\ref{stcyclesdel1}, \ref{stcyclesdel}.bc, \ref{c:stcyclesdel} и \ref{p:sawrong}.bcd)},
решение которых нам неизвестно; некоторые из них могут оказаться несложными. 
\talkonly{\newpage}

\talkno{

\subsection*{О стиле этого текста}
\addcontentsline{toc}{subsection}{О стиле этого текста}


Основные идеи представлены на <<олимпиадных>> примерах: на простейших 
частных случаях, свободных от технических деталей, и со сведением к необходимому минимуму научного языка.
За счет этого и текст становится доступным для начинающих, и удается быстро добраться до интересных сложных и важных результатов, методов и теоретических идей.
Таким образом, многие из приведенных идей работают для более общих случаев. 
Мы не тратим время читателя на несложные обобщения. 
Их легко придумать или найти в литературе. 
Трудно 
именно \emph{применить} общую теорию для ярких результатов, сформулированных вне этой теории, когда неизвестно заранее, какую теорию нужно применять (и вообще, можно ли хоть какую-то теорию применить). 

\emph{`Listeners are prepared to accept unstated (but hinted) generalizations much more than they are able, on the spur of the moment, to decode a precisely stated abstraction and to re-invent the special cases that motivated it in the first place.'}\footnote{Слушатели в гораздо большей степени готовы воспринимать не сформулированные (но упомянутые) обобщения, чем способны <<навскидку>> раскодировать явно сформулированную абстракцию и воспроизвести те конкретные примеры, которые изначально ее мотивировали.} \cite{Ha74}

\emph{`The preceding theorem generalizes to many situations. In fact, there may not actually be an optimal generalization in the sense that no matter what generalization is given, someone could produce a more general one. One of the advantages of being an understander rather than a theorem quoter is that one may be able to obtain approaches to a wide variety of theorems some of which may not even have been formulated yet.'}\footnote{Предшествующая теорема имеет многочисленные обобщения. В действительности, «наилучшего» обобщения может и не быть, в том смысле, что для каждого обобщения кто-то сможет найти еще более общее. Одним из преимуществ читателя, ориентированного на понимание сути, а не просто на цитирование теорем, является способность получать подходы к широкому спектру теорем, не все из которых вообще сформулированы.}
\cite{Bi83}

Более простой материал приводится, чтобы сделать естественным и доступным более сложный.
Попытка начинать с более простого (например, с частных случаев) повышает самостоятельность --- а, значит, глубину и надежность --- освоения материала. 
Проще самому доказать частный случай, самому продумать переход от частного к общему, чем  самому сразу доказать общий случай. 
Самостоятельно придуманное надежнее запоминается и легче модифицируется. 
Кроме того, обычно на частном случае проще отловить и исправить ошибки. 
Подробнее см. \cite[\S11.2]{ZSS}. 


Как правило, мы приводим \emph{формулировку} красивого или важного утверждения \emph{перед} последовательностью определений и результатов, составляющих его \emph{доказательство}.\algor{\footnote{Часто происходит обратное: формулировки красивых результатов и важных проблем, ради которых была придумана теория, приводятся только {\it после} продолжительного изучения этой теории (или не приводятся совсем). Это развивает представление о математике как науке, изучающей немотивированные понятия и теории. Такое представление принижает ценность математики.}}
В~таких случаях для доказательства утверждения требуется часть дальнейшего материала.
Об этом указано после формулировки утверждения.
Некоторые результаты приведены без доказательств, тогда даются ссылки вместо указаний.   
Если к утверждению (или задаче) не приведено ни доказательство, ни ссылка на него, то оно несложно.  
Основные результаты называются <<теоремами>>, менее важные результаты "--- <<утверждениями>>, важные вспомогательные результаты "--- <<леммами>>.
\invadraw{Еще менее важный или широко известный материал
приведен в приложениях.}
 
В тексте есть 
небольшое количество задач (полезно рассматривать и утверждения как задачи).
Изучение путем решения задач характерно для серьезного обучения математике, см. \cite[\S1.1]{HC19}, \cite[\S1.2]{ZSS} и данные там ссылки. 
Оно продолжает древнюю культурную традицию.\algor{\footnote{Например, послушники дзенских монастырей обучаются, размышляя над загадками, данными им наставниками.
Впрочем, эти загадки являются скорее наводящими на размышления парадоксами, а не задачами.
См. подробнее \cite{Su}.}}
\invadraw{In a lecture we often start with a question and only then give an answer, e.g.
`Может ли количество ... быть равным 7?'  instead of Assertion \ref{pla-kon}.a.}
Если условие задачи является формулировкой утверждения, то в~задаче требуется это утверждение доказать (и тогда в ссылках мы называем это утверждение утверждением, а не задачей).
Краткие указания
к задачам
приводятся сразу, более длинные в конце параграфа.
\emph{Загадкой} называется не сформулированный четко вопрос; здесь нужно придумать и~четкую формулировку, и~доказательство.
\algor{Если номер задачи помечен звездочкой,
то эта задача посложнее соседних.}
\invadraw{Имеются красивые задачи для исследования, полное решение которых мне неизвестно.}

Определения важных понятий даны \textbf{жирным шрифтом}, чтобы их было проще найти.
\algor{{\it Разделы и задачи, отмеченные звездочкой, а также замечания,} не используются в дальнейшем.}

{\it Разделы и задачи, отмеченные звездочкой,} не используются в дальнейшем.

}


\section{Одномерные циклы}\label{s:1cyc}

\subsection{Основные обозначения и определения}\label{s:nota}
 
В большей части параграфа \ref{s:1cyc} можно работать с графами на интуитивном уровне.
Строго говоря, \textbf{графом} называется пара $(V,E)$ из конечного множества~$V$ и семейства~$E$, состоящего из нескольких двухэлементных подмножеств множества~$V$.  
\textbf{Ребром} называется элемент семейства~$E$.

В этом тексте $K$ "--- граф.
Часто мы обозначаем ребро $\{a,b\}$ коротко через $ab$, а семейство ребер графа --- так же, как сам граф. 
Обозначим через

$\bullet$ $[n]$ множество $\{1, 2, \ldots, n\}$;

$\bullet$ $K_n$ полный граф на множестве $[n]$ вершин;   

$\bullet$ $K_{m,n}$ полный двудольный граф с долями  $[m]$ и $[n]'$ (мы обозначаем через $A'$ копию множества~$A$).  

\short{
\begin{figure}[ht]\centering
    \includegraphics[scale=1.5]{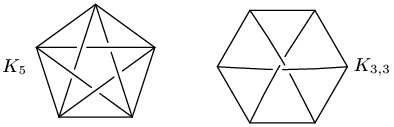} 
    \caption{\talkonly{\Large} Непланарные графы $K_5$ и $K_{3,3}$}
    \label{planar}
\end{figure}
}

Напомним, что \emph{простым циклом} $v_1v_2\ldots v_k$ 
в графе называется набор $\{v_1v_2, v_2v_3, \ldots, v_kv_1\}$ ребер, для которого вершины $v_1, \ldots, v_k$ попарно различны.

\short{
\begin{figure}[ht]\centering
    \includegraphics[width=10cm]{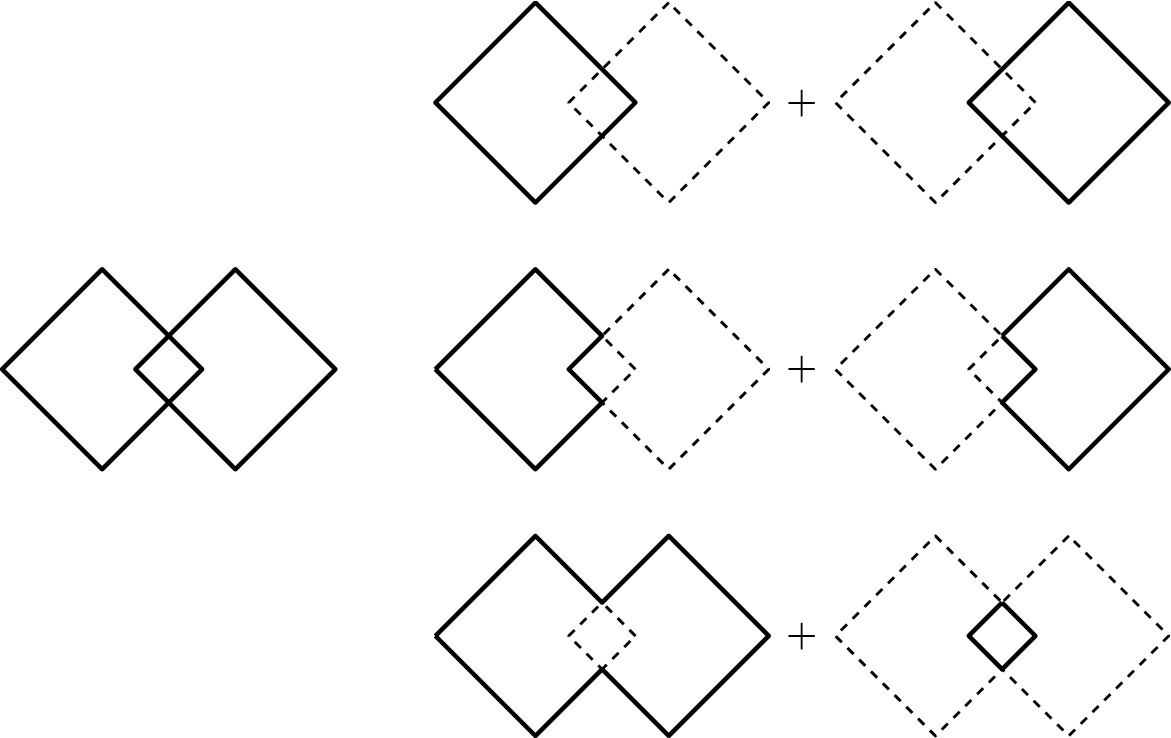}
    \caption{Слева: $1$-цикл. Справа: его представления в виде суммы простых 
    циклов}\label{f:cycles_sums}
\end{figure}

\begin{figure}[ht]\centering
    \includegraphics[width=2.5cm]{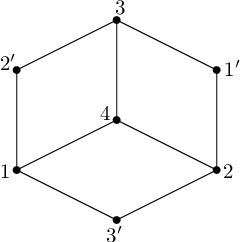}
    \caption{Цикл длины 6 является суммой трех циклов длины 4}\label{f:cycles64}
\end{figure}
}

\textbf{Суммой} (суммой по модулю $2$, или симметрической разностью, рис. \ref{f:cycles_sums} и \ref{f:cycles64}) множеств $A$ и $B$ называется
\[
    A + B := (A \cup B) \setminus (A \cap B).
\]

Пустой 1-цикл мы обозначаем нулем. 

Обозначим через $H_1(K)$ множество всех $1$-циклов в графе $K$, с операцией сложения. 

Мы иногда сокращаем $f(x)$ до $fx$ для образа элемента $x$ при отображении $f$.

\talkonly{\newpage}
\subsection{Как появляются одномерные циклы}\label{s:congra}

 
\begin{pr}\label{p:ouvert} 
    На плоскости нарисован без самопересечений граф (не обязательно связный), из каждой вершины которого выходит четное число ребер. 
    Тогда грани 
    можно раскрасить в два цвета {\it правильно}, т.~е. так, что при переходе через каждое ребро цвет грани меняется.  
(Определение грани можно найти, например, в \cite[\S 1.3]{Sk20}.)
\end{pr}

Это означает, что если плоский граф является $1$-циклом, то он равен сумме границ всех его граней одного цвета.
(Здесь \emph{границей} грани называется множество ребер, примыкающих к этой грани.)

\begin{pr}[Т. Голенищева-Кутузова, В. Клепцын и И. Ященко]\label{p:nonret} 
В некоторых клетках квадрата 20×20 стоит стрелочка в одном из четырех направлений. На границе квадрата все стрелочки смотрят вдоль границы по часовой стрелке. Кроме того, стрелочки в соседних (возможно, по диагонали) клетках не смотрят в противоположных направлениях. Докажите, что найдется клетка, в которой стрелочки нет.
\end{pr}

\begin{proof}[Доказательство (Э. Алкин)]
Назовем узел сетки \emph{внутренним}, если он содержится в четырех клетках, и \emph{граничным}, если содержится не более чем в двух клетках.
 
Пусть стрелочки стоят во всех клетках. Для каждой пары соседних (по общей стороне) клеток, в одной из которых стоит стрелка вверх, а другой --- стрелка вправо, раскрасим их общую сторону в красный цвет. В следующем абзаце докажем, что из каждого внутреннего узла выходит четное число красных ребер.
 
Рассмотрим произвольный внутренний узел и смежные с ним четыре клетки. Можно считать, что из этого узла выходит хотя бы одно красное ребро. Значит, в четырех смежных с узлом клетках стоят только стрелки вверх и стрелки вправо. Начав в любой клетке, совершим обход этих четырех клеток, каждый раз переходя в смежную по ребру клетку. Во время обхода мы пересечем красные ребра столько раз, сколько было изменений ориентаций стрелок. Поскольку после обхода мы вернулись в стартовую клетку, таких смен ориентаций было четное число.
 
Итак, из каждого внутреннего узла выходит четное число красных ребер. 
Из условия следует, что из одного граничного узла выходит одно красное ребро, а из остальных граничных узлов --- нисколько. 
Значит, сумма степеней всех вершин красного графа нечетна. 
Противоречие.
\end{proof}

\textbf{Замечание.}
Это доказательство чуть более простое, чем доказательства, опубликованные в брошюре ММО-2010 (задача 6 для 8 класса). 
В нем естественно появляется понятие 1-цикла (по модулю 2 и по модулю границы квадрата). 
Для знакомых с основами топологии добавим, что это доказательство похоже на доказательство леммы Шпернера \cite{Sh89} и на кусочно линейную версию, предположительно известную до 1963 года, доказательства Хирша 1963 года неретрагируемости диска на его границу \cite[п. 3.11 электронной версии]{Sk20}. 
 

\talkonly{\newpage}
\subsection{Одномерные циклы в графах}\label{s:cgra1}

\begin{proposition}\label{en:gencom1:sum1} 
    Сумма $1$-циклов является $1$-циклом.
\end{proposition}

\begin{proof}
    Возьмем $1$-циклы $C_1$ и $C_2$.
    Назовем \emph{степенью}~$d_C(v)$ вершины~$v$ в наборе~$C$ ребер количество ребер этого набора, которым принадлежит вершина~$v$.
    Тогда $d_{C_1+C_2}(v) = d_{C_1}(v) + d_{C_2}(v) - 2d_{C_1\cap C_2}(v)$ четно.
    Поэтому $C_1+C_2$ является $1$-циклом.
\end{proof}

\begin{pr}\label{l:gencom1} 
    (a) Количество $1$-циклов в $K_n$ равно $2^{\binom{n-1}{2}}$.

    (b) Если все ребра $1$-цикла в $K_n$ имеют общую вершину, то этот $1$-цикл пустой.
    
    (c) Любой $1$-цикл в $K_n$ является суммой нескольких циклов длины~$3$.

    (d) Любой $1$-цикл в $K_n$ является суммой нескольких из следующих циклов: цикл $123$ и циклы длины~$4$.

    (e) Цикл $123$ не является суммой никаких циклов длины~$4$ в $K_n$.
\end{pr}

П.~(b) доказывается напрямую. 

\begin{proof}[Доказательство пп.~(c,d)] 
    (c) Покажем, что произвольный $1$-цикл $C$ в $K_n$ равен сумме
    \[
        \widehat C := \sum\limits_{\substack{ij \in C\\ i,j < n}} ijn
    \]
    циклов $ijn$ по всем ребрам $ij\in C$, не содержащим вершину $n$.
    Каждое ребро в $C + \widehat C$ содержит вершину $n$. 
    Из этого и из п.~(b) следует, что $C + \widehat C = \emptyset$, т.~е. $C=\widehat C$.
    
    \talkno{(П.~(c) также следует из утверждения \ref{pr:orinumber1}.)}

    (d) Из п.~(c) следует, что достаточно доказать п.~(d) для циклов длины~$3$.
    Назовем цикл длины~$3$ \emph{представимым}, если его сумма с циклом $123$ представима в виде суммы циклов длины~$4$.
    Цикл $123$ представим.
    Сумма двух циклов длины~$3$, имеющих общее ребро, является циклом длины~$4$. 
    Поэтому если цикл $\alpha$ длины~$3$ представим, то представим и любой цикл длины~$3$, имеющий с $\alpha$ общее ребро.
    Следовательно, все циклы длины~$3$ представимы.
\end{proof}

\begin{proof}[Набросок доказательства п.~(e)]
    Для конечного множества $X$ обозначим через $|X|_2\in\Z_2$ четность количества элементов в множестве~$X$.
    Тогда для любых $1$-циклов $C,C'$ выполнено $|C+C'|_2 = |C|_2 + |C'|_2$.
\end{proof}
 
\begin{proof}[Доказательство утверждения \ref{l:gencom1}.a] 
    Достаточно доказать, что существует взаимно однозначное соответствие между $H_1(K_n)$ и множеством $2^{K_{n-1}}$ подмножеств ребер из $K_{n-1}$.
 
    Определим отображения
    \begin{align*}
        \varphi\colon H_1(K_n) \to 2^{K_{n-1}}\quad & \text{формулой} \quad \varphi C := C \cap K_{n-1}, \\
        \widehat \varphi\colon 2^{K_{n-1}} \to H_1(K_n)\quad & \text{формулой} \quad
        \widehat \varphi D := \sum\limits_{ij\in D} ijn.
    \end{align*}
    Поскольку $\varphi(ijn)=\{ij\}$, то $\varphi\widehat \varphi D = D$ для любого $D \subset K_{n-1}$.
    Обратно, для любого $C\in H_1(K_n)$ имеем $\widehat \varphi\varphi C = \widehat{C} = C$, где второе равенство доказано в доказательстве п.~(c).
    Таким образом, $\varphi$ и $\widehat\varphi$ являются взаимно однозначными соответствиями.
\end{proof}

\begin{proposition}[рис. \ref{f:cycles_sums}]\label{pr:simsum} 
    Любой $1$-цикл является суммой нескольких простых циклов. 
\end{proposition}

Это следует из леммы \ref{lem:sumcyc}.a. 

\begin{proof}[Набросок прямого доказательства]
Достаточно доказать, что любой $1$-цикл $C$ является суммой нескольких реберно непересекающихся простых циклов, содержащихся в $C$. 
Так как каждая вершина $1$-цикла принадлежит четному числу его ребер, то в $C$ содержится  простой цикл.
Удаляя из $C$ ребра этого простого цикла, получаем также $1$-цикл.
Продолжая, получим набор реберно непересекающихся простых циклов, объединение которых равно $C$. 
Поэтому их сумма равна $C$.
\end{proof}
 
Простой цикл называется \emph{бесхордовым}, если никакое ребро графа не соединяет две непоследовательные вершины этого цикла. 

\begin{proposition}\label{pr:orinumber1}
    Любой $1$-цикл является суммой нескольких бесхордовых циклов.
\end{proposition}

Это следует из утверждения \ref{pr:simsum}, поскольку любая <<хорда>> простого цикла дает его представление в виде суммы двух простых циклов меньшей длины.

Следующее утверждение \ref{pr:orinumber2} (и его доказательство) аналогично утверждению~\ref{l:gencom1} с заменой полного графа на произвольный граф. 

\begin{proposition}\label{pr:orinumber2}
(a) В связном графе с~$V$ вершинами и~$E$ ребрами количество $1$-циклов равно $2^{E-V+1}$.

(b) В дереве есть ровно один $1$-цикл, и он пустой.
\end{proposition}
 
П.~(b) доказывается напрямую. 
П.~(a) следует из нижеприведенной леммы~\ref{l:bij1}, использующей п.~(b). 

\short{
\begin{figure}[ht]\centering
    \includegraphics[width=12cm]{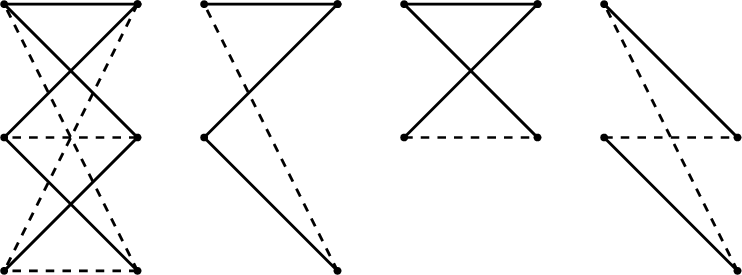}
    \caption{Слева: $K_{3,3}$, ребра вне остовного дерева выделены пунктиром. Два рисунка в центре: $\widehat \sigma$ и $\widehat \tau$ для соответствующих пунктирных ребер $\sigma$ и $\tau$. Справа: $\widehat \sigma + \widehat \tau$.}\label{f:k_33}
\end{figure}
}

В следующих двух леммах $K$ --- связный граф, а $T$ --- произвольное остовное дерево в нем.
Для ребра $\sigma \in K\setminus T$ обозначим через $\widehat\sigma = \widehat\sigma_T$ простой цикл в $K$, образованный ребром $\sigma$ и простым путем в $T$, соединяющим концы ребра $\sigma$ (рис. \ref{f:k_33}).

\begin{lemma}\label{l:bij1}
    Определим отображения
    \begin{align*}
        \varphi\colon H_1(K) \to 2^{K\setminus T }&\quad\text{формулой}\quad \varphi C := C\setminus T, \\
        \widehat \varphi\colon 2^{K\setminus T} \to H_1(K)&\quad\text{формулой}\quad 
        \widehat \varphi D := \sum\limits_{\sigma \in D} \widehat \sigma.
    \end{align*}
    Отображения $\varphi$ и $\widehat \varphi$ являются взаимно однозначными соответствиями между $H_1(K)$ и множеством $2^{K\setminus T}$
    подмножеств ребер из $K\setminus T$.
\end{lemma}

\begin{proof}
    Поскольку $\varphi \widehat \sigma=\{\sigma\}$, то $\varphi \widehat \varphi D = D$ для любого $D\subset K\setminus T$.
    Обратно, для любого $C \in H_1(K)$ имеем $\widehat \varphi\varphi C = C$ по лемме~\ref{lem:sumcyc}.b (ибо $\widehat \varphi\varphi C = \widehat C$). 
\end{proof}
 
\begin{lemma}\label{lem:sumcyc} 
    (a) Любой $1$-цикл в $K$ является суммой нескольких циклов $\widehat\sigma$.
    
    (b) Для любого $1$-цикла $C$    
        \[
            C = \widehat C:= \sum_{\sigma \in C \setminus T} \widehat \sigma.
        \]
    
    (c) Представление~из п.~(a) единственно.
\end{lemma}

П.~(a) следует из п.~(b). 
П.~(c) получается рассмотрением ребер вне $T$.

\begin{proof}[Доказательство п.~(b)]
    Так как $C$ и $\widehat C$ являются $1$-циклами, то $C':=C + \widehat C$ тоже является $1$-циклом.
    Поскольку $\widehat\sigma \setminus T = \{\sigma\}$ для любого ребра $\sigma\in K\setminus T$, то $\widehat C \setminus T = C \setminus T$.
    Значит, $C'\setminus T=0$. 
    Поэтому $C'\subset T$. 
    Тогда из утверждения~\ref{pr:orinumber2}.b следует, что $C'=0$, т.~е. $\widehat C = C$.
\end{proof}

Множество $B \subset H_1(K)$ называется \textbf{базисом} множества $A \subset H_1(K)$, если каждый $1$-цикл из $A$ имеет единственное представление в виде суммы нескольких $1$-циклов из $B$. 
Например, базисом в $H_1(K)$ является 

$\bullet$ для связного графа $K$ и произвольного остовного дерева в нем --- множество $1$-циклов $\widehat\sigma$ (ввиду лемм \ref{lem:sumcyc}.ac);  

$\bullet$ для плоского графа $K$ --- множество границ всех граней, кроме одной. 

См. другой пример в лемме~\ref{p:symgra}.b.
Базис можно определить и для более общих  случаев, см. раздел <<О стиле этого текста>> и замечание \ref{r:gene}.b. 
 
 \short{
\begin{figure}[ht]\centering
     \includegraphics[width=8cm]{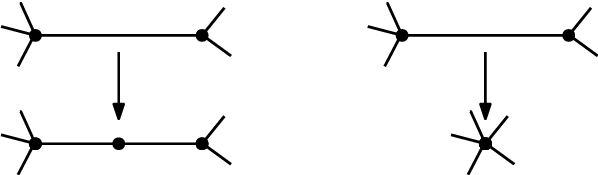} 
     \caption{Подразделение ребра (слева), стягивание ребра (справа)}\label{f:subdivision}
\end{figure}
}

\begin{remark}\label{r:gene}
    (a) Количество $1$-циклов в графе не меняется ни при \emph{подразделении} ребра, ни при \emph{стягивании} ребра (рис. \ref{f:subdivision}). 
    (Вторую из этих идей нужно формулировать аккуратно, ибо при стягивании ребра граф может перестать быть графом, а стать <<мультиграфом>>.)
    Это можно либо вывести из утверждения~\ref{pr:orinumber2}.a, либо доказать напрямую 
    (и тогда использовать для другого доказательства утверждения~\ref{pr:orinumber2}.a; впрочем, этот подход не будет более простым).

    (b) Операция суммы позволила коротко записать доказательство утверждения~\ref{pr:orinumber2}.a (хотя можно было обойтись без нее). 
    Рассмотрение не просто множеств, а множеств с операциями, часто оказывается полезным. 
    Но сейчас читателю вряд ли интересно выписывать использованные очевидные свойства этой операции (т.е. называть множество с этой операцией линейным пространством над $\Z_2$ или абелевой 2-группой).  
    Точно так же, как обсуждать структуру линейного пространства над $\Z_2$ на семействе всех подмножеств $n$-элементного множества тому, кто только что узнал, что в этом семействе $2^n$ подмножеств.  
    Однако, когда для получения ярких результатов будут нужны разные операции с разными свойствами, обсуждение этих свойств (и соответствующих математических терминов) уже будет интересно читателю.  
\end{remark}

Следующая задача подводит к своему многомерному обобщению --- к задаче \ref{pr:rook}.ef.

\begin{pr}\label{pr:knn1}
    (a) Количество $1$-циклов в $K_{n,n}$ равно $2^{(n-1)^2}$.
    
    (b) Если каждое ребро некоторого $1$-цикла в $K_{n,n}$ содержит хотя бы одну из вершин $n$ или $n'$, то этот $1$-цикл пустой.
        
    (c) Любой $1$-цикл в $K_{n,n}$ является суммой нескольких циклов длины~$4$.
\end{pr}

Обозначим через $\t{K_n}$ граф, полученный из $K_{n,n}$ удалением всех ребер $jj'$, $j\in[n]$ (соединяющих вершину $j$ с ее <<копией>> $j'$ в другой доле графа $K_{n,n}$). 
Например, $\t{K_3}$ "--- цикл $12'31'23'$ длины $6$. 
Граф $\t{K_n}$ естественно появляется в лемме~\ref{t:2cyc-bij}.

\begin{proposition}\label{l:tkn}
    Пусть $n\ge4$. 
    Тогда любой $1$-цикл в $\t{K_n}$ является суммой нескольких циклов длины $4$.
\end{proposition}

\begin{proof}
    Любой простой цикл в $\t{K_n}$ образован $l$-элементной циклической последовательностью вершин, такой что 
    
    $\bullet$ число~$l$ четное и больше~$2$; 
    
    $\bullet$ в последовательности чередуются вершины из разных долей;
    
    $\bullet$ внутри каждой отдельной доли вершины не повторяются; 
    
    $\bullet$ не существует двух последовательных вершин $m$ и $m'$.

    Предположим, что этот цикл бесхордовый.
    Тогда любые две непоследовательные вершины цикла не являются смежными, т.~е. либо находятся в одной доле, либо являются вершинами $m$ и $m'$.
    Для любой вершины цикла имеется $\frac{l}{2}-2$ не соседних с ней по циклу вершин из другой доли.
    Тогда $\frac{l}{2}-2\le1$, так что $l\in\{4,6\}$.
    
    В бесхордовом цикле длины~$6$ любые две противоположные вершины равны $m$ и $m'$.
    Следовательно, цикл равен $m_1m'_2m_3m'_1m_2m'_3$ для некоторых попарно различных $m_1,m_2,m_3$.
    Так как $n\ge4$, то существует $a\in[n]-\{m_1,m_2,m_3\}$.
    Тогда цикл равен $m_1m_2'm_3a' + m_2m_3'm_1a' + m_3m_1'm_2a'$ (рис. \ref{f:cycles64}).
    Теперь требуемый результат следует из утверждения~\ref{pr:orinumber1}.
\end{proof}

\begin{proof}[Другое доказательство утверждения \ref{l:tkn}]
    Достаточно доказать утверждение только для циклов $\widehat{ij'}$ из  доказательства леммы \ref{p:symgra}.b, поскольку они образуют базис в $H_1(\t{K_n})$.
    Для циклов $\widehat{3j'}$ и $\widehat{i2'}$ утверждение тривиально.
    Для $i \neq 3$ и $j \neq 2$ имеем 
    $$\widehat{ij'} = \begin{cases}
        12'ij' + 2'31'i, & i \neq 2;\\
        12'3j' + 31'2j', & j \neq 3;\\
        12'43' + 2'31'4 + 1'23'4 \text{ (рис. \ref{f:cycles64})}, & i = 2,~ j = 3. 
    \end{cases}$$
  
\end{proof}

\talkonly{\newpage}\subsection{Соотношения между одномерными циклами}\label{s:cgra2}

В этом разделе мы опишем соотношения между простейшими $1$-циклами (треугольниками в полном графе, а также циклами длины 4 в полном двудольном графе).
Этот материал обобщается в утверждениях \ref{pr:rel} и \ref{pr:rook2}.
    
\begin{proposition}\label{pr:rel1}    
    (a) Для любых вершин $a,b,c,d$ в $K_n$ выполнено $abc+abd+acd+bcd=0$.
    
    (b) Любое линейное соотношение между циклами длины~$3$ является суммой нескольких соотношений из утверждения (a).
\end{proposition}

\short{
\begin{figure}[ht]\centering
    \includegraphics[width=3cm]{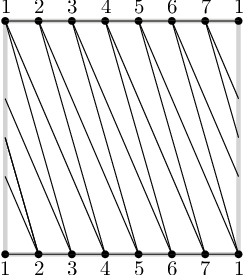}
    \caption{Триангуляция тора ($7$-вершинная)}\label{f:k7-on-torus}
\end{figure}
}

Например, 

$\bullet$ соотношение $123+124+134+235+245+345=0$ является суммой соотношений
    
     $123+124+134+234=0$ и $234+235+245+345=0$;

$\bullet$ сумма границ треугольников любой триангуляции сферы или тора (рис.~\ref{f:k7-on-torus}) равна $0$.
 
Строго говоря, \emph{линейное соотношение} "--- это множество циклов длины~$3$ в $K_n$ такое, что каждое ребро графа $K_n$ содержится в четном числе циклов из этого множества.    
Поэтому строгая формулировка п.~(b) такова (мы отождествляем циклы длины~$3$ в $K_n$ с $3$-элементными подмножествами множества $[n]$).

Пусть $R_1,\ldots,R_k\subset[n]$ "--- такие $3$-элементные подмножества, что любое $2$-элементное подмножество множества $[n]$ содержится в четном количестве из них. 
Тогда существуют такие $4$-элементные подмножества множества $[n]$, что любое $R_i$ содержится в нечетном количестве из них, а любое другое $3$-элементное подмножество множества $[n]$ содержится в четном количестве из них.

    
В этой строгой формулировке п.~(b) эквивалентен утверждению~\ref{l:gencom}.a. 

\begin{proposition}\label{pr:knn2}
    (a) Для любых попарно различных $a,b,c\in[n]$ и различных $u',v'\in[n]'$ выполнено $ au'bv' + bu'cv' + cu'av' =0$.  
        
    (b) Любое линейное соотношение между циклами длины~$4$ в $K_{n,n}$ является суммой нескольких соотношений из утверждения~(a) и аналогичных соотношений $a'ub'v + b'uc'v + c'ua'v = 0$.
\end{proposition}

Строго говоря, \emph{линейное соотношение} "--- это множество циклов длины~$4$ в $K_{n,n}$ такое, что каждое ребро графа $K_{n,n}$ содержится в четном количестве циклов из этого множества.
 
О наиболее близких одномерных обобщениях (на целочисленные 1-циклы и на 1-когомологии) см., например, \cite[\S4.11, \S10.5]{Sk20}, \cite[задача 2.2.6.de]{GDI}, \cite[\S9.2]{Sk}.

\talkonly{\newpage}
\talkno{
\subsection{Симметричные одномерные циклы}\label{s:sym1}

Укажем, что именно используется далее из этого параграфа.
Симметрия $t$ на графе $\t{K_n}$ используется только начиная с утверждения \ref{l:h2sym}.
Лемма \ref{pr:knbij} используется только в доказательстве леммы \ref{p:symgra}.a.
Лемма \ref{p:symgra} используется только в доказательстве леммы \ref{t:2cyc-sym-lemma}.
Утверждение \ref{prop:inv} используется только в доказательстве утверждения \ref{p:uni1s}.
  
Напомним, что граф $\t{K_n}$ получен из $K_{n,n}$ удалением всех ребер $jj'$, $j\in[n]$. 
Обозначим через $t\colon\t{K_n}\to\t{K_n}$ симметрию (инволюцию), переставляющую доли графа, то есть переставляющую $j$ и $j'$ для каждого $j\in[n]$.  

Через $tQ$ обозначим набор ребер, симметричный набору $Q$ ребер графа $\t{K_n}$.
В частности, $Q$ может быть одним ребром или $1$-циклом.
Например, $t\t{K_3}=\t{K_3}$ и $t(12'31'ij')=1'23'1i'j$.
Назовем набор $Q$ ребер \emph{$t$-симметричным}, если $tQ = Q$.
Например, $\t{K_3}$ является $t$-симметричным.

\begin{lemma}\label{pr:knbij} 
    Существует взаимно однозначное соответствие между $t$-симметричными $1$-циклами в $\t{K_n}$ и $1$-циклами в $K_n$.
\end{lemma}

Нужное соответствие получается из отображения $\t{K_n}\to K_n$, переводящего вершины $j$ и $j'$ графа $\t{K_n}$ в вершину $j$ графа $K_n$ для каждого $j \in [n]$.

Из утверждения \ref{l:gencom1}.a и леммы \ref{pr:knbij}  вытекает, что количество $t$-симметричных $1$-циклов в $\t{K_n}$ равно $2^{\binom{n-1}{2}}$.

\begin{lemma}\label{p:symgra}
    (a) Любой $t$-симметричный $1$-цикл в $\t{K_n}$ является суммой нескольких из следующих $1$-циклов: $\t{K_3}$ и $Q+tQ$ для циклов $Q$ длины $4$. 

    (b) Множество всех $1$-циклов в $\t{K_n}$ имеет базис, состоящий из ($t$-симметричного $1$-цикла) $\t{K_3}$ и пар простых циклов длины 4 или 6, взаимно симметричных относительно $t$.
\end{lemma}

\begin{proof}[Набросок доказательства]
(a) Возьмем взаимно однозначное соответствие из леммы \ref{pr:knbij}. 
Тогда п.~(a) следует из утверждения \ref{l:gencom1}.d и из того, что прообразом цикла длины 3 является $\t{K_3}$, а прообразом цикла длины 4 является $Q+tQ$ для некоторого цикла $Q$ длины 4.

    
(b)  Остовным деревом графа $\t{K_n}$ является подграф, состоящий из простого пути $12'31'$ и ребер $i1', 1j'$ для всех $i, j > 1$.
Тогда из лемм~\ref{lem:sumcyc}.ac следует, что множество циклов 
    $$\widehat{ij'} := 
    \begin{cases}     
    12'3j',& i=3,~ j \neq 2;\\
    2'31'i,& i \neq 3,~ j=2; \\
    12'31'ij',& \text{иначе}
    \end{cases}$$ 
для ребер $ij'$ из $\t{K_n}$ таких, что $i,j>1$ и $(i,j)\ne(3,2)$, является базисом в $H_1(\t{K_n})$.

Поскольку $t\widehat{ij'} = \widehat{23'}+\widehat{ji'}$, то циклы
$$\t{K_3}=\widehat{23'},\quad \widehat{ij'},\quad t\widehat{ij'},\quad\text{где}\quad i>j>1 \quad\text{и}\quad (i,j)\ne(3,2),$$
образуют требуемый базис.
\end{proof}
 

Пусть для графа $K$ задана симметрия (инволюция) $t\colon K\to K$, т.е. перестановка множества вершин, для которой 

$\bullet$ вершины, соединенные ребром, переходят в вершины соединенные ребром, и 

$\bullet$ $t(tv)=v$ для любой вершины $v$. 

Примеры: 

$\bullet$  симметрия $t$ графа $\t{K_n}$, определенная в начале этого параграфа; 

$\bullet$ \emph{антиподальная} симметрия цикла четной длины, отображающая каждую вершину в диаметрально противоположную.

Вершина $v$ называется \emph{неподвижной} (при симметрии $t$), если $v = tv$.
Через $tQ$ обозначим набор ребер, симметричный набору $Q$ ребер графа $K$.
В частности, $Q$ может быть одним ребром или $1$-циклом.
Назовем набор $Q$ ребер \emph{симметричным}, если $tQ = Q$.
 
\begin{proposition}\label{prop:inv} 
Пусть для симметрии $t$ связного графа $K$ с $V$ вершинами и $E$ ребрами нет неподвижных вершин. 
Обозначим через $I$ количество симметричных ребер. 
Тогда количество симметричных $1$-циклов равно 
$\begin{cases}2^{(E-V+2)/2},&\text{если }I=0,\\ 2^{(E-V+I)/2},&\text{если } I>0.\end{cases}$ 
\end{proposition}

\begin{proof}[Набросок доказательства] {\it Случай $I = 0$.} Говоря нестрого, назовем \emph{факторграфом} графа $K$ по симметрии $t$ граф $K/t$, полученный из $K$ <<склеиванием>> симметричных вершин и симметричных ребер.
Строго говоря, вершинами графа $K/t$ являются неупорядоченные пары взаимно симметричных вершин графа $K$.
Ребро в $K/t$ между двумя такими парами проводится, если есть ребро в $K$ между некоторыми вершинами из обеих пар. 
Например, факторграфом графа $\t{K_n}$ по симметрии $t$ является граф $K_n$ (ср. с леммой \ref{pr:knbij}).

Пусть сначала 

($*$) ни одна вершина графа $K$ не соединена с двумя взаимно симметричными вершинами.

Тогда граф $K/t$ имеет $V/2$ вершин и $E/2$ ребер.
По утверждению \ref{pr:orinumber2}.a в $K/t$ ровно $2^{(E - V)/2 + 1}$ $1$-циклов.
Существует <<естественное>> взаимно однозначное соответствие между симметричными $1$-циклами в $K$ и $1$-циклами в $K/t$.
Значит, в $K$ столько же симметричных $1$-циклов.

Сведем общий случай к рассмотренному.
Обозначим через $K'$ граф, полученный из $K$ подразделением каждого ребра.
Обозначим через $V'$ и $E'$ количество вершин и ребер в $K'$ соответственно.
Поскольку $V' = V + E$ и $E' = 2E$, то $E' - V' = E - V$.
Определим симметрию $t'$ на графе $K'$ так: $t'$ совпадает с $t$ на вершинах графа $K$ и переводит вершину на ребре $\sigma$ графа $K$ в вершину на ребре $t\sigma$ графа $K$.  
Подразделение ребра не меняет количество $1$-циклов.
Подразделение двух взаимно симметричных ребер не меняет количество симметричных $1$-циклов.
Поэтому количества симметричных $1$-циклов в $K$ и $K'$ совпадают.
Граф $K'$ удовлетворяет условию ($*$). 
Следовательно, в нем ровно $2^{(E' - V')/2 + 1} = 2^{(E - V)/2 + 1}$ симметричных $1$-циклов.

{\it Случай $I>0$.}
В $K$ существует симметричное остовное дерево $T$, содержащее ровно одно симметричное ребро.
Докажем это в следующем абзаце.

Выберем произвольное симметричное ребро $\tau$.
Построим $T$ итеративно.
Обозначим $T_0 = \{\tau\}$.
Если дерево $T_i$ является остовным, то положим $T = T_i$.
Иначе, поскольку $K$ связен, существует вершина $v$ вне $T_i$, которая соединена ребром $\sigma$ с некоторой вершиной дерева $T_i$.
Так как $v \neq tv$, то вершины $v$ и $tv$ являются листьями в $T_{i+1} = T_i \cup \{\sigma, t\sigma\}$, поэтому $T_{i+1}$ является деревом.
Дерево $T$ не содержит симметричное ребро $\sigma$, отличное от $\tau$, поскольку иначе ребра $\sigma$, $\tau$ и два взаимно симметричных пути от вершин ребра $\tau$ до вершин ребра $\sigma$ образуют $1$-цикл в $T$.

Возьмем взаимно однозначное соответствие из леммы \ref{l:bij1}.
При этом соответствии симметричные $1$-циклы в $K$  переходят в симметричные наборы ребер вне $T$ (ввиду следующего абзаца).

Ввиду симметричности дерева $T$ верно $\widehat {t\sigma} = t\widehat{\sigma}$ для любого ребра $\sigma \in K \setminus T$.
Возьмем симметричный $1$-цикл $C$ в $K$.
Поскольку $C$ и $T$ симметричны, $\sigma \in C \setminus T$ тогда и только тогда, когда $\ t\sigma \in C \setminus T$.

Поскольку в $T$ ровно одно симметричное ребро, то вне $T$ ровно $I - 1$ симметричных ребер.
Симметричные наборы ребер вне $T$ разбиваются на $I - 1$ одноэлементных наборов и 
\linebreak
$(E - V + 1 - (I - 1)) / 2$ пар взаимно симметричных ребер.
Таким образом, количество симметричных наборов ребер вне $T$ равно $2^{(I - 1) + (E - V + 1 - (I - 1)) / 2)} = 2^{(E - V + I)/2}$.
\end{proof}

Более продвинутый материал об инволюциях см., например, в \cite[\S7]{Sk20}.


\begin{proof}[Набросок сведения случая $I = 0$ утверждения \ref{prop:inv} к случаю $I > 0$]
Возьмем произвольную вершину $v$ графа $K$. 
Обозначим через $K'$ граф, полученный из графа $K$ добавлением симметричного ребра $\{v,tv\}$, с инволюцией $t$. 
Множества симметричных $1$-циклов в $K'$ и в $K$ совпадают.
Для доказательства этого утверждения достаточно проверить, что в $K'$ нет симметричных $1$-циклов, содержащих симметричное ребро.
Докажем это в следующем абзаце.

В любом связном $1$-цикле $C$ есть эйлеров цикл --- цикл, проходящий по всем ребрам $1$-цикла $C$ ровно по одному разу.
Разобьем вершины симметричного $1$-цикла $C$ на два взаимно симметричных множества $U$ и $tU$.
При обходе по эйлеровому циклу в $C$ вершина меняет принадлежность множеству $U$ или $tU$ четное число раз.  
Значит, количество ребер $1$-цикла $C$ между $U$ и $tU$ четно.
При этом несимметричные ребра между $U$ и $tU$ разбиваются на взаимно симметричные пары.
Поэтому количество симметричных ребер между $U$ и $tU$ четно.
Значит, связный симметричный $1$-цикл не может содержать ровно одно симметричное ребро. 
\end{proof}

}

\subsection{Геометрическое отступление: декартово произведение графов *}\label{s:prod}

В этом параграфе мы приведем конструкции <<двумерных>> объектов. 
Формально, эти конструкции не используются далее, ибо излагаются на языке <<одномерном>> (\S\ref{s:1squ}-\S\ref{s:1delsqu}) или <<комбинаторном>> (\S\ref{s:cydepr}-\S\ref{s:2symsqu}). 
Но без знакомства с этими конструкциями такой язык может показаться недостаточно естественным. 

Возьмем в трехмерном пространстве $n$ прямоугольников $XYB_kA_k,~ k = 1, \ldots , n$, любые два из которых пересекаются только по отрезку $XY$. 
\emph{Книжкой с $n$ листами} называется объединение этих прямоугольников (рис. \ref{f-3pagesbook} для $n = 3$).

\short{
\begin{figure}[ht]
\centerline{
\includegraphics[width=3cm]{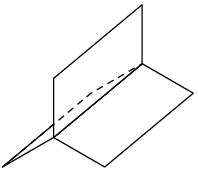}
}
\caption{Книжка с тремя листами} 
\label{f-3pagesbook}
\end{figure}
}

\begin{example}\label{p:realbook}
Любой граф можно нарисовать без самопересечений на книжке с некоторым количеством листов, зависящим от графа. 
Более строго, для любого~$n$ существуют целое число~$k$, а также $n$~точек и $n(n-1)/2$~таких несамопересекающихся ломаных на книжке с $k$~листами, что 
 
$\bullet$ каждая пара точек соединена ломаной, и 
 
$\bullet$ никакая ломаная не пересекает внутренность другой ломаной.
\end{example} 

\begin{proof}[Построение]
Построим нужные точки и ломаные на книжке с $n$ листами.
Обозначим  через $XY$ отрезок пересечения листов.
Поместим вершину $V_i$ на $i$-й лист вне отрезка $XY$, $i \in [n]$.
На $XY$ выберем $n(n-1)/2$ различных точек $A_{\{i,j\}}$, $i, j \in [n]$.
Ребро $e_{\{i, j\}}$ изобразим в виде двузвенной ломаной $V_iA_{\{i,j\}}V_j$.
Никакое звено ломаной не пересекает внутренность никакого звена другой ломаной, так как либо эти звенья лежат на различных листах, либо имеют общую вершину.
\end{proof}

Сначала мы дадим интуитивное определение цилиндра, а затем строгое.

Возьмем вершины $a$ и $b$ в графе $K$, соединенные ребром. Возьмем прямоугольник-ленточку $aa'b'b$, соответствующую этому ребру.
Склеим концевые отрезки ленточек, соответствующие одной и~той же вершине, так, чтобы штрихованные буквы склеивались со штрихованными.
Полученная двумерная фигура называется \emph{цилиндром} над графом~$K$.

Напомним, что $\R^d$ "--- это $d$-мерное евклидово пространство (для $d=2$ и $d=3$ это обычные плоскость и пространство, которые изучаются на уроках геометрии). 
\emph{Цилиндром} над подмножеством $U\subset\R^d$ называется 
$$U\times K_2:=\{(x,t)\in\R^{d+1}\ :\ x\in U,\ t\in[\,0,1\,]\}.$$

Например, цилиндры над $K_2$, $K_3$ и $K_5$ показаны на рис.~\ref{f-mnre}; цилиндр над $K_{3, 1}$ показан на рис. \ref{f-3pagesbook}; цилиндр над $K_{k,1}$ <<выглядит>> как книжка с $k$ листами. 

\short{
\begin{figure}[ht]
    \centerline{\includegraphics[width=3cm]{real-83.mps}\qquad
    \includegraphics[width=3.5cm]{real-23.mps}\qquad
    \includegraphics[width=5cm]{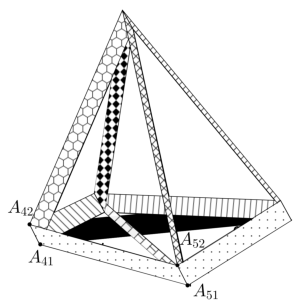}
    }
    \caption{Цилиндры над $K_2$, $K_3$, $K_5$} 
    \label{f-mnre}
\end{figure}
}

\begin{example}\label{p:realcyl} 
    Цилиндр над любым графом <<реализуем без самопересечений>> в $\R^3$.
\end{example}

\begin{proof}[Набросок построения]
    Возьмем точки $A_{11},\ldots,A_{1n}\in\R^3$, никакие четыре из которых не лежат в одной плоскости. Возьмем вектор $v$, не параллельный никакой плоскости, проходящей через любые три из этих точек. Для каждого $p\in[n]$ возьмем точку $A_{2p}$, для которой $\vec{A_{1p}A_{2p}}=v$. Если $v$ достаточно мал, то точки $A_{jp}$, $j\in\{1,2\}$, $p\in[n]$,"--- искомые.
\end{proof}

Для подмножеств $U,V\subset\R^d$ их 
\emph{произведением} называется
$$U\times V:=\{(x,y)\in\R^{2d}\ :\ x\in U,\ y\in V\}.$$
В частности, \emph{квадрат} подмножества $U\subset\R^d$ "--- это $U^2=U\times U$.  
Например, квадрат $K_2^2$ "--- это обычный квадрат на плоскости, а квадрат $K_3^2$ (называемый \emph{тором}) показан на рис.~\ref{f:k5i}, в центре.

\short{
\begin{figure}[ht]
\centerline{\includegraphics[width=5cm]{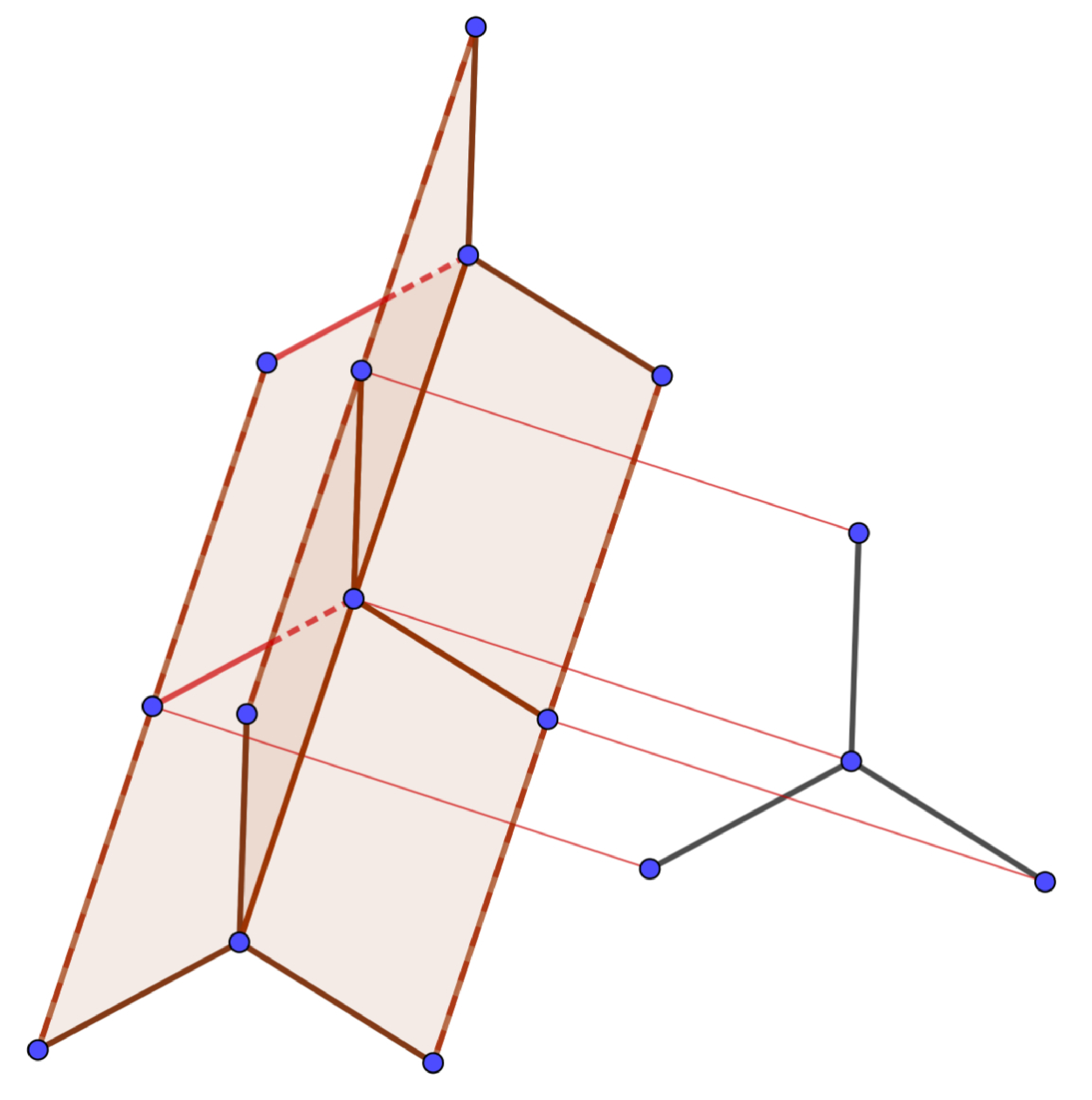}}
\caption{Изображение квадрата $K^2_{3,1}$ с самопересечениями в трехмерном пространстве}
\label{f:k3,1^2,1}
\end{figure}
}

 \begin{pr}[загадка]\label{p:cyl} 
 Нарисуйте в $\R^3$ произведение $K_3\times K_{3,1}$ без самопересечений.
 \end{pr}

\begin{remark} \label{r:body}
(a) Иногда термин <<граф>> используется для понятия <<тело графа>>, определяемого следующим образом.
Пусть некоторое подмножество в $\R^3$ находится во взаимно однозначном соответствии с множеством вершин графа $K$, причем ни один отрезок, соответствующий какому-либо ребру графа $K$, не пересекает внутренность никакого другого такого отрезка. 
\textit{Телом} графа $K$ является объединение этого подмножества и всех таких отрезков. 

(b) См. обобщения, например, в \cite[\S6]{Sk}. 
\end{remark}

\talkonly{\newpage}
\subsection{\talkno{Одномерные циклы}\talkonly{$1$-циклы} в произведении графов}\label{s:1squ}

Основной результат этого пункта "--- теорема Кюннета \ref{stcycles2} о классификации $1$-циклов с точностью до прибавления \emph{границ} в \emph{произведении} графов. 
Остальной материал этого пункта подводит к ее формулировке и доказательству. 
Сама теорема Кюннета \ref{stcycles2} используется далее только в доказательстве теоремы \ref{p:stcycles}. 

Назовем \textit{конфигурацией} пару жуков (красного и синего), расположенных в вершинах красного и синего графов.  
Назовем две конфигурации \textit{смежными}, если одна из них может быть получена из другой перемещением одного из жуков вдоль какого-либо (одного) ребра.

Иными словами, рассмотрим множество (конфигурационное пространство) упорядоченных пар $(x,y)$ точек $x$ и $y$ графов $K$ и $L$ (точнее, их тел, см. замечание \ref{r:body}.a), хотя бы одна из которых является вершиной.
Это множество является объединением конечного числа отрезков, т.~е. телом некоторого графа.
Приведем прямое комбинаторное определение этого графа.

Вершинами \textbf{$\Box$-произведения} $K\Box L$ графов $K$ и $L$ являются упорядоченные пары $(a,b)$ вершин $a$ графа $K$ и $b$ графа $L$, обозначаемые $a \Box b$.
Если вершины $b$ и $c$ графа $L$ соединены ребром, то вершины $a \Box b$ и $a \Box c$ графа $K\Box L$ соединены ребром, обозначаемым $a \Box bc$. 
Если вершины $b$ и $c$ графа $K$ соединены ребром, то вершины $b \Box a$ и $c \Box a$ графа $K\Box L$ соединены ребром, обозначаемым $bc \Box a$.  
Других ребер в $K\Box L$ нет.
Обозначим $K^{\square2}=K\Box K$. 
 
\talkonly{\newpage} 
\begin{figure}[!hbt]
    \begin{minipage}[t]{0.33\textwidth}
	\centering
	\includegraphics[width=2cm]{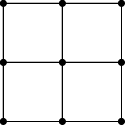}
    \end{minipage}
    \begin{minipage}[t]{0.33\textwidth}
	\centering
	\includegraphics[width=5cm]{real-33.mps}
    \end{minipage}
    \begin{minipage}[t]{0.33\textwidth}
	\centering
        \includegraphics[width=4cm]{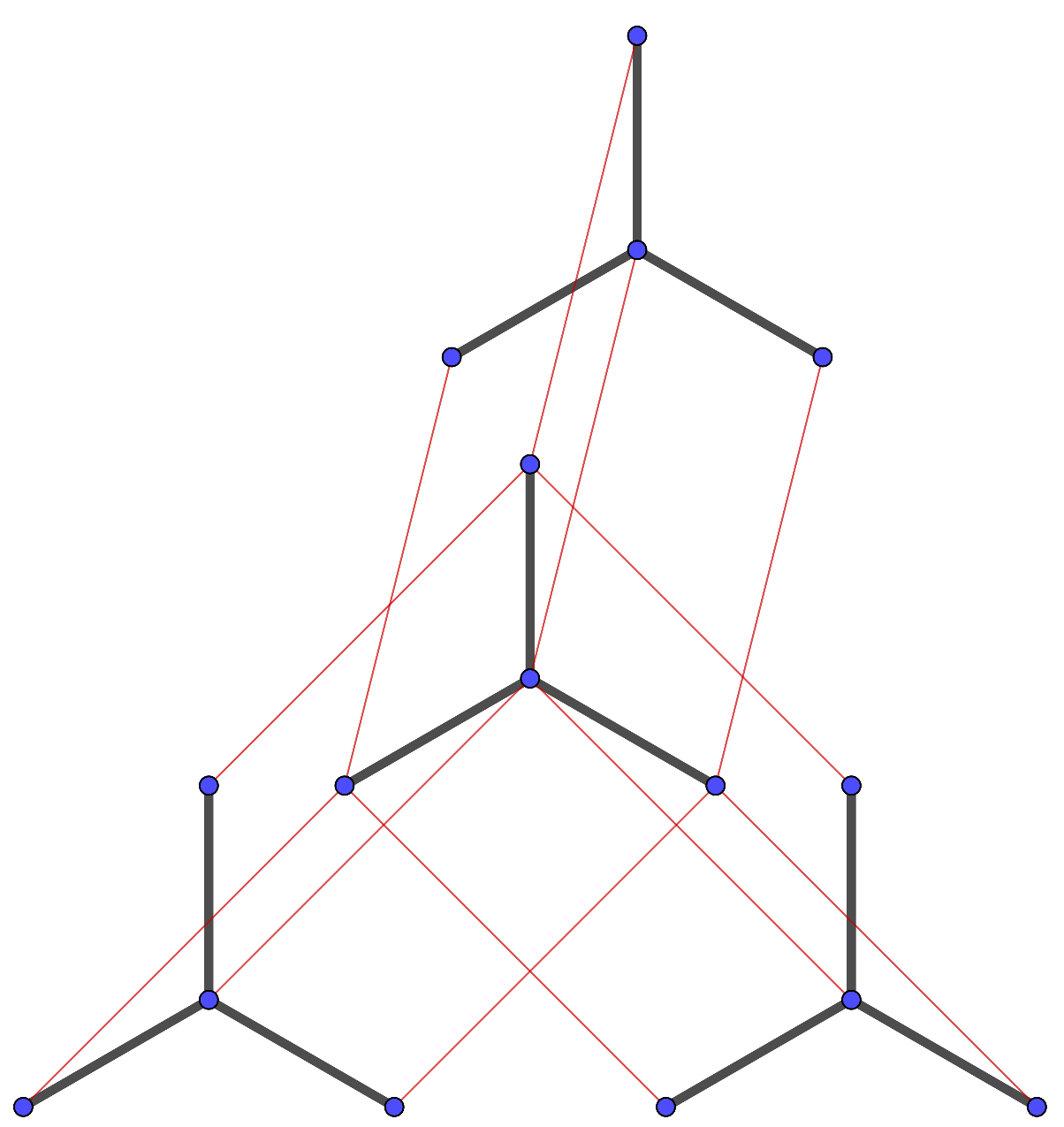}
    \end{minipage}\hfill
    \caption{Слева: $K^{\square 2}_{2,1}$, т.~е. сеточный граф $3\times 3$. 
    В центре: $K_3^{\square2}$ на торе $K_3^2$
    (граничные циклы затянуты серыми четырехугольниками).
    Справа: $K^{\square 2}_{3,1}$}
        \label{f:grid-3-3}
        \label{f:k5i}
        \label{f:k3,1^2,2}
\end{figure}

Например, 

$\bullet$ если $K=K_2$ --- отрезок, то $K^{\square 2}$ "--- цикл длины $4$;

$\bullet$ если $K=K_{2,1}$ --- путь на трех вершинах, то $K^{\square 2}$ "--- сеточный граф $3\times 3$ на рис~\ref{f:grid-3-3} слева; 

$\bullet$ если $K=K_3$ --- цикл на трех вершинах, то $K^{\square 2}$ "--- граф на рис. \ref{f:k5i} в центре; его можно получить из графа $K_{2,1}^{\square2}$ добавлением ребер  между соответствующими вершинами 1-й и 3-й строки, а также 1-го и 3-го столбцов 
(сравните с получением графа $K_3$ из пути~$123$ добавлением ребра $13$);

$\bullet$ если $K=K_{3,1}$ --- <<триод>>, то $K^{\square 2}$ "--- граф на рис.~\ref{f:k3,1^2,2} справа;

$\bullet$  вершины графа $K_4^{\square 2}$ "--- узлы сетки~$4\times 4$; 
две вершины соединены ребром тогда и только тогда, когда они находятся в одной строке или в одном столбце.

\talkno{Дополнительную информацию о произведении $K\Box L$ см. в \cite{CPG}.}


\talkonly{\newpage} 

Простой цикл в графе $K\Box L$ будем обозначать перечислением его вершин через запятую.
(Это обозначение отличается от того, что было дано во введении.)

\emph{Границей} называется простой цикл 
$$ab\Box uv := a\Box u,\ b \Box u,\ b \Box v,\ a \Box v$$
для ребер $ab$ и $uv$ в $K$ и $L$ (произведение $ab\times uv$ является прямоугольником,\talkno{ см. \S\ref{s:prod},} а $ab\Box uv=\partial(ab\times uv)$ является его границей). 

Во многих применениях полезно отождествить 1-циклы, отличающиеся на сумму нескольких границ, см., например, \cite[лемма 2.5 и \S7]{ABM+}. 

Оказывается, любой $1$-цикл в $K_{2,1}^{\square 2}$ и в $K_{3,1}^{\square 2}$ является суммой нескольких границ (попробуйте доказать!).  

\begin{proposition}\label{p:treek2} Если $K$ и $L$ "--- деревья, то любой $1$-цикл в $K\Box L$  является суммой нескольких границ.  
\end{proposition}

Это доказывается по индукции с использованием удаления висячей вершины.


\talkonly{\newpage} 
\begin{figure}[!hbt]\centering
    \begin{minipage}[t]{0.4\textwidth}
	\centering
    \includegraphics[width=6cm]{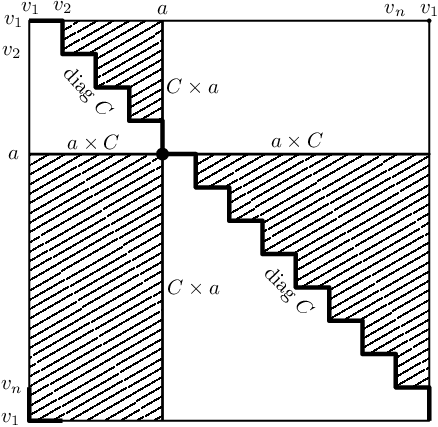}
    \end{minipage}
    \begin{minipage}[t]{0.3\textwidth}
	\centering
	\includegraphics[width=3.5cm]{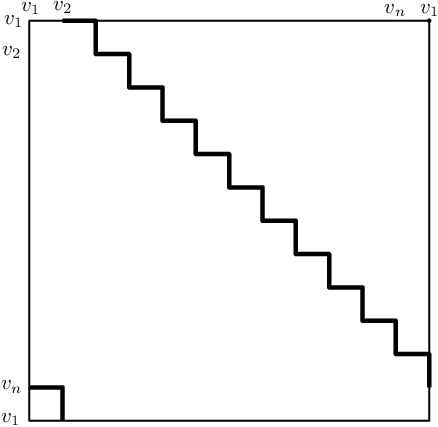}
    \end{minipage}
    \begin{minipage}[t]{0.3\textwidth}
	\centering
        \includegraphics[width=3.5cm]{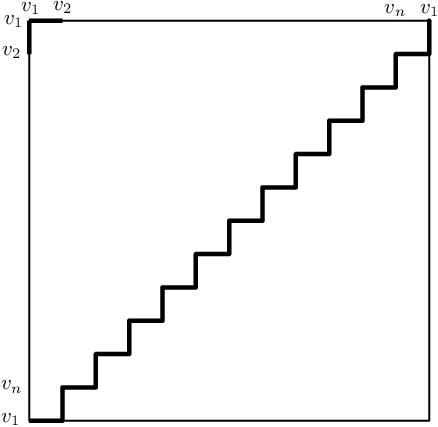}
    \end{minipage}\hfill
    \caption{Слева: диагональный цикл, симметризованный цикл, их разность "--- сумма заштрихованных границ (наглядное обоснование утверждения \ref{p:diag}.a для $a\in C$). 
    В центре: околодиагональный цикл. Справа: антидиагональный цикл}
\label{f:diag}
\end{figure}

Вот примеры простых циклов в $K_3^{\square 2}$: \emph{диагональный}, \emph{околодиагональный},  \emph{антидиагональный} циклы
$$1\Box1,\ 1\Box2,\ 2\Box2,\ 2\Box3,\ 3\Box3,\ 3\Box1;\qquad 1\Box2,\ 1\Box3,\ 2\Box3,\ 2\Box1,\ 3\Box1,\ 3\Box2;$$ $$1\Box1,\ 2\Box1,\ 2\Box3,\ 3\Box3,\ 3\Box2,\ 1\Box2.$$ 

Вот примеры простых 
циклов в $K^{\square 2}$ для вершины $a$ и непустого простого цикла $C=v_1\ldots v_k$ в $K$ (рис. \ref{f:diag}):   

$\bullet$ \textbf{левый и правый} циклы 
$a\Box C:=a \Box v_1,\ \ldots,\ a\Box v_k \quad\text{и}\quad C\Box a:=v_1\Box a,\ \ldots,\ v_k \Box a$;  

$\bullet$ \textbf{симметризованный} цикл $a\Box C+C\Box a$; 

$\bullet$ \textbf{диагональный} цикл
$\diag C:=v_1\Box v_1,\ v_1\Box v_2,\ v_2\Box v_2,\ \ldots ,\ v_k\Box v_k,\ v_k\Box v_1;$  

$\bullet$ \textbf{околодиагональный} цикл
$v_1\Box v_2,\ v_1\Box v_3,\ v_2\Box v_3,\ \ldots,\ v_k\Box v_1,\ v_k\Box v_2;$

$\bullet$ \textbf{антидиагональный} цикл
$v_1\Box v_1,\ v_2\Box v_1,\ v_2\Box v_k,\ \ldots ,\ v_k\Box v_2,\ v_1\Box v_2.$ 

\begin{pr}\label{p:diag} (a) Диагональный, \quad (b) околодиагональный, \quad (c) антидиагональный

цикл является суммой нескольких границ и симметризованного цикла.   
\end{pr}

\talkonly{\newpage} 
\begin{pr}\label{p:nobo}
    (a) Никакой ненулевой левый цикл не является суммой нескольких границ. \quad 
    
    (b) Никакой ненулевой диагональный цикл не является суммой нескольких границ. 
    
    (с) Никакой ненулевой левый цикл не является суммой нескольких диагональных циклов и границ. 
\end{pr}

Это вытекает из (очевидной) леммы \ref{p:probou}. 
 
\emph{Левой проекцией $C_y$ по модулю $2$} для $1$-цикла $C$ в $K\Box L$ называется множество всех ребер~$\sigma$ в~$L$, таких что имеется нечетное количество вершин~$a$ в~$K$, для которых $a \Box \sigma \in C$. 
\emph{Правая проекция $C_x$ по модулю $2$} определяется аналогично.

\begin{lemma}\label{p:probou} 
    Левая (и правая) проекция по модулю $2$ любой границы (и, следовательно, любой суммы границ) пуста.
\end{lemma}

\talkonly{\newpage} 
Количество $1$-циклов \emph{по модулю границ} (или \emph{с точностью до прибавления границ}) "--- это максимальное количество $1$-циклов в наборе $1$-циклов, ни один из которых не является суммой каких-либо других $1$-циклов этого набора и нескольких границ.

Два $1$-цикла $C,~C'$ в $K\Box L$ называются \emph{гомологичными} (или сравнимыми по модулю границ), если $C+C'$ является суммой нескольких границ. 
Обозначение: $C\sim C'$. 
Количество $1$-циклов по модулю границ есть количество классов гомологичности $1$-циклов. 
  
\begin{pr}\label{p:k21} Найдите количество классов гомологичности $1$-циклов в

(a)  $K_3^{\square 2}$; \quad (b)  $K_{2,2}^{\square 2}$; \quad (c)  $K_{2,3}^{\square 2}$; \quad (d) $K_4^{\square 2}$. 
\end{pr}

\talkno{Для некоторых пунктов задачи \ref{p:k21} полезна теорема Кюннета \ref{stcycles2}.b.}
 
\talkonly{\newpage} 
\begin{pr}\label{p:stcycles2} Пусть $T$ "--- дерево. 

(a') Любой $1$-цикл в $K\Box T$ является суммой правого цикла и нескольких границ. 

(a) Для любых $1$-цикла $C$ в $K\Box T$ и вершины $a$ в $T$ 

\quad(a1) существует единственный $1$-цикл $C_K$ в~$K$, такой что $C \sim C_K \Box a$;

\quad(a2) $C\sim C_x\Box a$.

(b) Количество классов гомологичности $1$-циклов в $K\Box T$ равно количеству $1$-циклов в $K$. 
\end{pr}

Это доказывается индукцией по количеству вершин в $T$ с использованием удаления висячей вершины. 
Кроме того, п.~(b) и существование в (a1) следуют из (a2); единственность в (a1) следует из леммы \ref{p:probou}.
На примере решения этой задачи читатель увидит (без технических деталей) основную идею доказательства теоремы Кюннета \ref{stcycles2}.  
 
\talkonly{\newpage} 
\begin{theorem}[Кюннет]\label{stcycles2}
Пусть графы $K$ и $L$ связны. 

(a') Любой $1$-цикл в $K\Box L$ является суммой правого цикла, левого цикла и нескольких границ. 

(a) Для любых $1$-цикла $C$ в $K\Box L$ и вершин $a,b$ в $K,L$ 

\quad(a1) существуют единственные $1$-циклы $C_K$ и $C_L$ в~$K$ и $L$ такие, что $C\sim C_K\Box b+a\Box C_L$;
 
\quad(a2) $C\sim C_x\Box b + a\Box C_y$.  
 
(b) Количество классов гомологичности $1$-циклов в $K\Box L$ равно произведению количества $1$-циклов в $K$ на количество $1$-циклов в $L$.
\end{theorem}

\begin{proof}[Набросок доказательства]
    По следующей лемме \ref{l:zerpro}, если $Z_x=Z_y=0$, то $Z\sim0$.
    Применяя ее к
    $Z = C + C_x\Box b + a\Box C_y$, получаем (a2). 
    
    П.~(b) и существование в (a1) следуют из (a2); единственность в (a1) следует из леммы \ref{p:probou}.
\end{proof}

\begin{lemma}\label{l:zerpro} 
    Если графы $K$ и $L$ связны, $Z$ является $1$-циклом в $K\Box L$ и $Z_x=0$, то $Z\sim a\Box Z_y$ для любой вершины $a$ в $K$.
\end{lemma}

\talkonly{\newpage
Если графы $K$ и $L$ связны, $Z$ является $1$-циклом в $K\Box L$ и $Z_x=0$, то $Z\sim a\times Z_y$ для любой вершины $a$ в $K$.} 

\begin{proof}[Набросок доказательства]
    Докажем лемму индукцией по числу пар $\{\sigma \Box v, \sigma \Box u\}$ ребер в $Z$, проецируемых слева на одно и то же ребро в $K$.
    Назовем эту индукцию \emph{главной}.
    В следующем абзаце докажем \emph{базу главной индукции}, когда все ребра в $Z$ имеют вид $v \Box \tau$. 
    Эту базу докажем индукцией по количеству вершин~$b$ в~$K$ таких, что $Z\cap (b\Box L)\ne\varnothing$.
    Назовем такие вершины \emph{интересными}, а эту индукцию --- \emph{вложенной}.
 
    \emph{База вложенной индукции}: если интересных вершин нет, то $Z$ пуст; если интересная вершина $a$ одна, то $Z = a \Box Z_y$ и лемма доказана.
    Для \emph{шага вложенной индукции} возьмем две интересные вершины $b, c$, отличные от $a$, и путь $b = v_1, \ldots, v_k = c$ в $K$.
    Заменим $Z$ на 
    $$Z' := Z + v_1v_2 \Box Z_y + \ldots + v_{k-1}v_k \Box Z_y.$$
    Для $Z'$ верно предположение вложенной индукции, поэтому 
    $Z \sim Z' \sim a \Box Z'_y = a \Box Z_y$.
    
    Для \emph{шага главной индукции} возьмем пару $\{\sigma \Box v, \sigma \Box u\}$ и путь $v=v_1\ldots v_k=u$ в $L$.
    Заменим $Z$ на 
    $$Z' := Z+\sigma\Box v_1v_2+\ldots+\sigma\Box v_{k-1}v_k.$$
    Для $Z'$ верно предположение главной индукции, поэтому $Z \sim Z' \sim a \Box Z'_y = a \Box Z_y$.
\end{proof}   

\begin{proof}[Набросок доказательства утверждения~\ref{p:diag}]
    Начнем с графа $K_3^{\square 2}$. 
    Можно считать, что $C=12\ldots k$. 

    (a) См. идею на рис. \ref{f:diag}.     
    Например: 
    $$\diag(123) = 1\Box K_3+K_3\Box 1 + 12\Box 23+12\Box 31 + 23\Box 31.$$

    А вот общая формула:
    $$\diag C = 1\Box C+C\Box 1 + \sum\limits_{\substack{i,j\in [k]\\ i<j}} i(i+1)\Box j(j+1),\talkno{\quad\text{где}\quad k+1:=1.}$$
\talkonly{где $k+1:=1$.}    

    (b) Околодиагональный цикл в $K_3^{\square 2}$ равен 
     $1\Box K_3+K_3\Box 1+12\Box 31+31\Box 12.$
    
    (c) Антидиагональный цикл в $K_3^{\square 2}$ равен 
    $1\Box K_3+K_3\Box 1+23\Box 31+31\Box 23+31\Box 31.$
\end{proof}

\talkno{\textit{Ответы к задаче \ref{p:k21}:} (a) $2^2$; (b) $2^2$; (c) $2^4$; (d) $2^6$.}

\talkonly{\newpage}
\subsection{Симметричные одномерные циклы в квадрате графа *}\label{s:1symsqu}

Рассмотрим симметрию (инволюцию) графа $K^{\square 2}$, переставляющую сомножители (т.~е. переставляющую точки $(x,y)$ и $(y,x)$). 
Рассмотрим соответствующую симметрию на $1$-циклах. 
\talkno{О простейшем применении симметричных $1$-циклов см. \cite[\S1.6]{Sk18}, \cite[\S1.6]{Sk}.}

\begin{theorem}\label{p:stcycles}  
    (a) Любой симметричный $1$-цикл в $K^{\square 2}$ является суммой нескольких симметризованных циклов $C\Box a + a\Box C$ и нескольких границ.
    \talkno{(Следовательно, он равен сумме нескольких диагональных циклов и нескольких границ; обратите внимание, что диагональный цикл <<симметричен по модулю границ>>.)}
    
    (b) Для связного графа $K$ формула 
    $$C\mapsto [C\Box a + a\Box C]$$ 
    определяет взаимно однозначное соответствие между $1$-циклами в $K$ и симметричными $1$-циклами в $K^{\square 2}$ по модулю границ (здесь $[x]$ обозначает класс гомологичности $1$-цикла $x$).
    
    Обратное соответствие задается формулой 
    $$[Z]\mapsto Z_x=Z_y.$$ 
\end{theorem}

\begin{proof}[Набросок доказательства теоремы \ref{p:stcycles}]
    Вот <<связный>> аналог п.~(а): любой симметричный $1$-цикл $Z$ в $K^{\square 2}$ для связного графа $K$ является суммой симметризованного цикла $a \Box C + C \Box a$ и нескольких границ.
    П.(b) следует из этого утверждения.
    П.(a) получается из него применением к каждой компоненте связности графа $K$.
    Докажем этот аналог в следующем абзаце.
    
    Положим $b = a$ в теореме Кюннета $\ref{stcycles2}$.a. 
    Получим, что $Z \sim Z_x \Box a + a \Box Z_y$.
    Поскольку $Z$ симметричен, то $Z_x = Z_y$.
\end{proof}

Очевидно, что в теореме \ref{p:stcycles}.a сумма границ симметрична.
Лемма \ref{l:symbo} дает более сильный результат, утверждающий, что эта сумма является суммой $1$-циклов особого вида.

\emph{Диагональной границей} называется 1-цикл $\sigma \Box \sigma$ для ребра $\sigma$ графа $K$.
\emph{Симметризованной границей} называется 1-цикл $\sigma \Box \tau + \tau \Box \sigma$ для пары $\{\sigma, \tau\}$ различных ребер графа $K$. 
 
\talkonly{\newpage}
 
\begin{lemma}\label{l:symbo} 
    Если сумма границ в $K^{\square2}$ симметрична, то она является суммой диагональных и симметризованных границ.
\end{lemma}

\begin{proof}[Набросок доказательства]
    Достаточно доказать лемму для связного графа $K$.
    Теперь достаточно доказать, что если сумма попарно различных границ симметрична, то либо сумма равна нулю, либо найдется диагональная граница, либо найдутся два взаимно симметричных слагаемых.
    Рассмотрим аналогичное утверждение, в котором $K^{\square 2}$ заменен на $K\Box T\bigcup\limits_{T\Box T} T\Box K$, где $T$ является деревом.
    Оно доказывается индукцией по числу вершин дерева $T$ с использованием удаления листа.
    Для того, чтобы свести лемму к этому аналогичному утверждению, применяется равенство $\sum\limits_{i,j} \sigma_i\Box\tau_j = 0$ для любых двух простых циклов в $K$, имеющих последовательные ребра $\sigma_1\ldots\sigma_k$ и $\tau_1\ldots\tau_\ell$ (см. утверждение~\ref{t:torus}.a и свойство \ref{p:ce2cy}(iii)).
    Применяя это равенство, заменим данную сумму границ на равную ей сумму, не содержащую слагаемых $\sigma\Box\tau$, соответствующих ребрам $\sigma, \tau$ вне остовного дерева $T$ графа $K$.
\end{proof}

\begin{proof}[Полное, немного другое доказательство (Е. Дженжер)]
    Достаточно доказать лемму для связного графа.
    Обозначим через $S$ симметричную сумму границ, а через $T$ "--- 
    остовное дерево в $K$.
    Назовем \emph{зеркальной границей} произвольную сумму диагональных и симметризованных границ.
    Докажем индукцией по числу $E$ ребер в $S\setminus (T\Box T)$, что $S$ является зеркальной границей.
    Назовем эту индукцию \emph{главной}.
    
    Докажем \emph{базу главной индукции} при $E = 0$ индукцией по числу $n$ вершин в дереве $T$.
    Назовем эту индукцию \emph{вложенной}.
    В этом случае $S\subset T\Box T$.

    \emph{База вложенной индукции} при $n = 1$ очевидна.
    В следующих четырех абзацах докажем \emph{переход вложенной индукции} от $n$ к $n+1$.
     
    Возьмем висячую вершину $v$ и обозначим через $u$ единственного соседа вершины $v$ в дереве $T$.
     
    Если $(uv, v)\in S$, то положим 
    $$S'=\sum\limits_{\substack{e\in K\\ e\neq uv}} (e\Box uv+uv\Box e) + uv\Box uv.$$
    
    Иначе (если $(uv, v)\notin S$) положим $$S'=\sum\limits_{\substack{e\in K\\ e\neq uv}} (e\Box uv+uv\Box e).$$
     
    Тогда в $S+S'$ нет ребер $(xy, v)$ и $(v, xy)$, где $x, y$ "--- вершины дерева $T$.
    Значит, $S+S'\subset (T-v)\Box (T-v)$. 
    По предположению вложенной индукции, $S+S'$ является зеркальной границей. 
    Тогда $S=(S+S')+S'$ является зеркальной границей.
    
    Докажем \emph{переход главной индукции} от $E$ к $E+1$. 
    Возьмем ребро $(ab, x)$ в $S\setminus(T\Box T)$. 
    
    Так как $S$ является суммой границ и $S_x=\emptyset$, то найдется вершина $y\neq x$ в $K$, для которой в $S$ есть ребро $(ab, y)$. 
    Возьмем произвольный путь $x \ldots y$ в $K$ между вершинами $x$ и $y$. 
    Положим 
    $$S' := ab \Box (x\ldots y) + (x\ldots y) \Box ab.$$
    Тогда 1-цикл $S'$ является зеркальной границей.
    Дополнение $S'\setminus(T\Box T)$ содержит только ребра $(ab, x)$, $ (x, ab)$, $ (ab, y)$, $ (y, ab)$.
    Так как они принадлежат $S$, то $S+S'$ содержит $E-4$ ребра не из $T\Box T$.
    Значит, $S+S'$ является зеркальной границей по предположению главной индукции.
    Тогда для $S=(S+S')+S'$ утверждение леммы также верно.
\end{proof}

\begin{proposition}\label{p:uni1s}
    Для связного графа $K$ с $V$ вершинами и $E$ ребрами количество симметричных $1$-циклов в $K^{\square 2}$ равно $2^{VE - \binom{V}{2}}$.
\end{proposition}

\begin{proof}[Набросок доказательства]
    Заменим <<диагональные>> вершины $a \Box a$ графа $K^{\square 2 }$ на ребра (рис. \ref{f:kcirc2}).  
    Строго говоря, по графу $K^{\square 2}$ построим граф $K^{\circ 2}$ следующим образом: заменим каждую вершину $a \Box a$ на две вершины $a \Box \circ$ и $\circ \Box a$, соединенные ребром (символу $\circ$ не придается отдельного смысла).
    Каждую вершину $a \Box b$, которая соединена с $a \Box a$ в $K^{\square 2}$, соединим с $a \Box \circ$ в $K^{\circ 2}$.
    Каждую вершину $b \Box a$, которая соединена с $a \Box a$ в $K^{\square 2}$, соединим с $\circ \Box a$ в $K^{\circ 2}$.
    Симметрия графа $K^{\circ 2}$ переставляет местами вершины 
    $a \Box b$ и $b \Box a$, где $a$ или $b$ могут совпадать с $\circ$.

\begin{figure}[!hbt]\centering
    \begin{minipage}[t]{0.33\textwidth}
	\centering
	\includegraphics[width=2cm]{grid-3-3.eps}
    \end{minipage}
    \begin{minipage}[t]{0.33\textwidth}
	\centering
	\includegraphics[width=2cm]{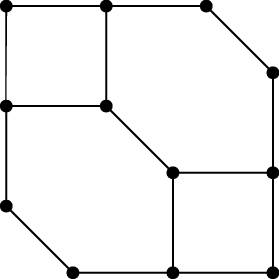}
    \end{minipage}
    \caption{Слева: $K^{\square 2}_{2,1}$, т.~е. сеточный граф $3\times 3$. 
    Справа: $K^{\circ 2}_{2,1}$}
        \label{f:kcirc2}
\end{figure}

    В графе $K^{\circ 2}$ ровно $V^2+V$ вершин и $2VE+V$ ребер.
    Из этих ребер ровно $V$ симметричны себе. 
    Тогда по утверждению \ref{prop:inv} в $K^{\circ 2}$ ровно $2^{VE - \binom{V}{2}}$ симметричных $1$-циклов.
    Поэтому достаточно доказать, что существует взаимно однозначное соответствие между симметричными $1$-циклами в $K^{\circ 2}$ и симметричными $1$-циклами в $K^{\square 2}$.
    
    Для этого определим \emph{стягивание <<диагональных>> ребер},
    т.е. отображение $f\colon V(K^{\circ 2})\to V(K^{\square 2})$, сопоставив вершине 
    $a \Box b$ вершину $a \Box a$, если $b = \circ$; $b \Box b$, если $ a = \circ$ и $a \Box b$ иначе.
    Отображение $f$ задает взаимно однозначное соответствие между <<недиагональными>> ребрами графа $K^{\circ 2}$ и ребрами графа $K^{\square 2}$.
    Следовательно, оно задает взаимно однозначное соответствие $\widehat f$ между наборами <<недиагональных>> ребер графа $K^{\circ 2}$ и наборами ребер графа $K^{\square 2}$.
    
    Возьмем симметричный $1$-цикл $C$ в графе $K^{\circ 2}$ и вершину $a \Box b$ в $K^{\square 2}$.
    Если $a \neq b$, то $a \Box b$ принадлежит тем ребрам множества $\widehat fC$, прообразы которых смежны с вершиной $a \Box b$ в $K^{\circ 2}$.
    Иначе $a \Box a$ принадлежит тем ребрам множества $\widehat fC$, прообразы которых смежны с вершинами $a \Box \circ$ и $\circ \Box a$ в $K^{\circ 2}$.
    Так как $C$ является $1$-циклом, количества указанных в обоих случаях ребер четны. 
    Поэтому $\widehat fC$ является $1$-циклом в $K^{\square 2}$.
    
    Поскольку $t\widehat f\sigma = \widehat ft\sigma$ для любого <<недиагонального>> ребра $\sigma$ в $K^{\circ 2}$, то $\widehat fC$ симметричен.
    
    Осталось заметить, что любой $1$-цикл в $K^{\circ 2}$ однозначно определяется своим набором <<недиагональных>> ребер.
\end{proof}

\talkonly{\newpage}
\subsection{Одномерные циклы во взрезанном квадрате графа *}\label{s:1delsqu}

Следующий граф $K^{\square \underline2}$ \talkno{(как и комбинаторный взрезанный квадрат $K^{\underline2}$, определенный в \S\ref{s:2delsqu})} является теоретико-графовым (топологическим) аналогом множества размещений.
Рассмотрим двух жуков на $K$, как описано в начале \S\ref{s:1squ}, но теперь им запрещено находиться в одной вершине. 
Иными словами, рассмотрим множество (конфигурационное пространство) упорядоченных пар $(x,y)$ точек графа $K$ (точнее, его тела, см. замечание \ref{r:body}.a), хотя бы одна из которых является вершиной, а другая не лежит внутри ребра с концом в этой вершине. 
Это множество является объединением конечного числа отрезков, т.~е. телом некоторого графа.
Вот прямое комбинаторное определение этого графа. 

Вершинами графа $K^{\square \underline2}$ являются упорядоченные пары различных вершин графа~$K$.
Вершины графа $K^{\square \underline2}$ соединены ребром в $K^{\square \underline2}$, если они соединены ребром в $K^{\square 2}$. 
Обозначение с нижним подчеркиванием для графа $K^{\square \underline 2}$ (как и для $K^{\underline 2}$ в \S\ref{s:2delsqu}) мотивировано обозначением для количества размещений (нижней степени, убывающего факториала). 
Мы просим читателя отличать $K^{\square \underline 2}$ от $K^{\square 2}$ (и далее $K^{\underline 2}$ от $K^2$). 

\begin{figure}[ht]
    \centerline{\includegraphics[width=5cm]{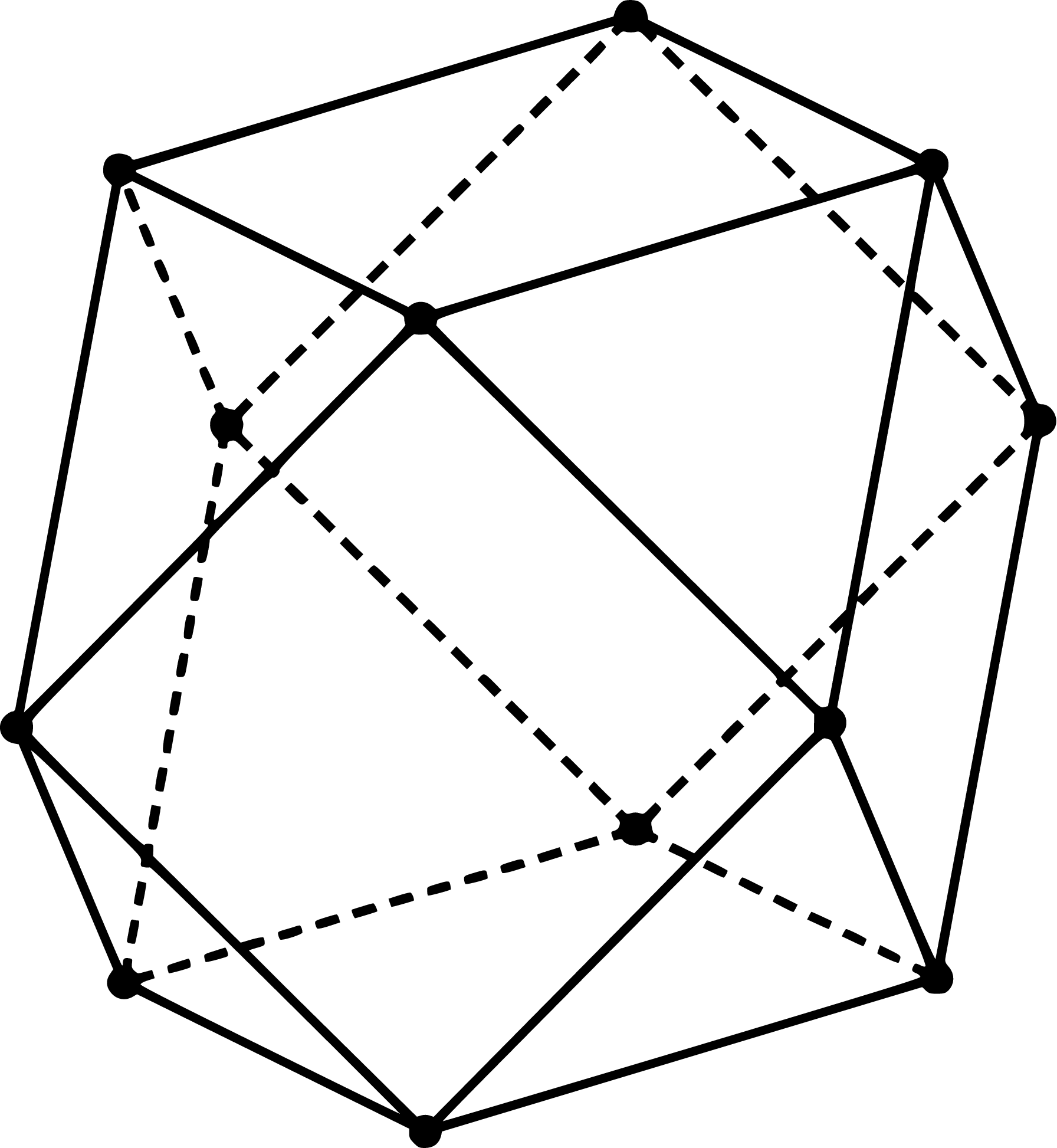}
    \qquad\includegraphics[width=8cm]{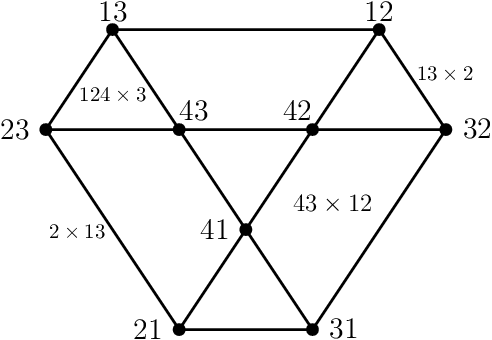}}
    \caption{Слева: кубооктаэдр; объединение ребер есть $K_4^{\square\underline2}$; четырехугольные грани образуют $K_4^{\underline2}$\talkno{ (см. \S\ref{s:cydepr})}. 
    Справа: то же с некоторыми пояснениями (не показана невидимая часть, проекция которой получается из изображенной проекции поворотом на $\pi/3$)}
    \label{f:del3}
\end{figure}

Например, 

$\bullet$ $K_{2,1}^{\square\underline2}$ "--- объединение двух непересекающихся копий графа $K_{2,1}$; 
 
$\bullet$ $K_3^{\square\underline2}$ "---  цикл на $6$ вершинах; 

$\bullet$ $K_{3,1}^{\square\underline2}$ "--- цикл на $12$ вершинах;  

$\bullet$ $K_4^{\square\underline2}$ "--- объединение ребер кубооктаэдра (рис.~\ref{f:del3}). 



Среди циклов, определенных перед утверждением \ref{p:diag}, некоторые лежат в $K^{\square \underline 2} \subset K^{\square 2}$. 
Это, например, околодиагональные циклы и симметризованные циклы $a\Box C+C\Box a$ для $a\not\in C$.
А вот симметризованный цикл $a \Box C + C \Box a$ для $a\in C$ не лежит в $K^{\square \underline 2}$.
 
\talkonly{\newpage}
\begin{pr}[загадка]\label{p:1delpro} 
Найдите количество $1$-циклов в $K^{\square \underline2}$ для связного графа $K$.
(\emph{Указание:} примените утверждение~\ref{pr:orinumber2}.a и его обобщение для несвязных графов.)
\end{pr}

\begin{pr}\label{p:knn21} 
Найдите количество $1$-циклов в $K^{\square \underline2}$ по модулю границ (содержащихся) в $K^{\square \underline2}$ (т.~е. по модулю границ, соответствующих парам несмежных ребер), для $K=$

(a) $K_3$; \quad (b) $K_{2,2}$; \quad (c)  $K_{2,3}$; \quad (d) $K_4$; \quad (e) $K_{3,3}$; \quad (f) $K_5$.   

(По поводу обобщения на произвольные графы см. \cite{FH10}.)
\end{pr}


Далее изучается выразимость одних $1$-циклов в $K^{\square\underline2}$ через другие. 

Например, следующая гипотеза является версией утверждения~\ref{p:diag}.b с заменой $K^{\square 2}$ на $K^{\square \underline 2}$.

\begin{conjecture}\label{stcyclesdel1} В $K^{\square\underline2}$ никакой околодиагональный цикл не является суммой границ и симметризованных циклов. 
\end{conjecture}


\talkonly{\newpage}

\emph{Триодическим циклом} называется цикл 
$$1 \Box 3,\ 1\Box 1',\ 1\Box 2,\ 1'\Box 2,\ 3\Box 2,\ 3\Box 1',\ \ldots\quad\text{в}\quad K_{3,1}^{\square 2},$$ 
где точками обозначена часть, симметричная выписанной части (т.~е. полученная заменой $x\Box y$ на $y\Box x$). 
\emph{Триодическим циклом} также называется аналогичный цикл, соответствующий $K_{3,1}$-подграфу графа $K$.  

\begin{pr}\label{stcyclesdel}
(a) Триодический цикл не является суммой никаких границ в $K^{\square\underline2}$.

(b) Любой ли симметризованный цикл в $K^{\square \underline2}$  является суммой околодиагональных циклов, триодических циклов и границ?  

(c) Любой ли околодиагональный цикл в $K_4^{\square \underline2}$  является суммой некоторых симметризованных циклов, триодических циклов и границ?  

(d) Выразите в $K_4^{\square \underline2}$ цикл $4\Box K_3+K_3\Box 4$ через околодиагональные и триодические циклы по модулю границ (см. ответ в \cite[замечание 7.3.b]{ABM+}).
\end{pr}

\begin{conjecture}\label{c:stcyclesdel} В $K^{\square\underline2}$ 

(a) любой $1$-цикл является суммой левых циклов, правых циклов, околодиагональных циклов, триодических циклов и границ;
 
(b) любой симметричный $1$-цикл является суммой околодиагональных циклов (переопределенных в симметричном виде), триодических циклов и границ (см. доказательство в \cite{Dz25});

(c) любой симметричный $1$-цикл является суммой <<симметричных околодиагональных циклов>>,
триодических циклов и симметризованных границ.  
\end{conjecture}

\begin{proof}[Указания и ответы к задаче~\ref{p:knn21}]
(a-d) Если множество границ в 
$K^{\square\underline2}$
для графа $K$ из пп.~(a-d)
имеет нулевую сумму, то это множество пусто. 
    \talkno{См. утверждения \ref{p:ce2cy} и \ref{t:dpcycle0}.abc.}
    
    (e), (f) Сумма всех границ в $K^{\square\underline2}$ равна нулю. 
    Это единственный непустой набор границ в $K^{\square\underline2}$, сумма которых равна нулю. 
    \talkno{См. утверждения \ref{p:ce2cy},  \ref{t:dpcycle} и \ref{t:dpcycle2}.}


\emph{Ответы:} (a) $2^1$; (b) $2^1$; (c) $2^5$; (d) $2^7$; (e) $2^8$; (f) $2^{12}$.
\end{proof}

\talkonly{\end{document}}


\section{Двумерные циклы}

В этом параграфе мы приведем <<двумерные>> обобщения результатов из \S\ref{s:1cyc}, а также их применения. 

\subsection{Двумерные циклы в гиперграфах *}\label{s:chygra1}

Назовем \emph{$2$-циклом} множество $C$, состоящее из $3$-элементных подмножеств (называемых \emph{гранями}) множества $[n]$, такое, что каждое $2$-элементное подмножество множества $[n]$ содержится в четном количестве подмножеств, являющихся элементами множества $C$.  

Например, пустое множество является $2$-циклом.

Приведем еще один пример $2$-цикла.
Для $4$-элементного подмножества $A\subset[n]$ назовем \textit{тетраэдром} $T_A$ множество всех $3$-элементных подмножеств множества $A$. 
Другими словами, для попарно различных $a,b,c,d\in [n]$ определим \textit{тетраэдр} 
$$T_{\{a,b,c,d\}}:=\bigl\{\{a,b,c\},\{a,b,d\},\{a,c,d\},\{b,c,d\}\bigr\}.$$ 
Очевидно, что любой тетраэдр является $2$-циклом.

Понятие 2-цикла уже неявно появлялось как соотношение между 1-циклами (утверждения \ref{pr:rel1} и \ref{pr:knn2}) и при подсчете количества 1-циклов по модулю границ (задача \ref{p:knn21}). 
Другие воплощения этого понятия появятся далее в виде ладейных циклов (задача \ref{pr:rook} и замечание \ref{p:cychyp}) и клеточных 2-циклов в произведении графов (\S\ref{s:cydepr}). 

Сравните следующие утверждения \ref{prop:2-cyc}, \ref{l:gencom} и \ref{pr:rel} с аналогичными утверждениями \ref{en:gencom1:sum1}, \ref{l:gencom1} и \ref{pr:rel1} соответственно.

\emph{Границей} грани $\{a, b, c\}$  называется cемейство подмножеств $\{ab, bc, ca\}$. 

\begin{proposition}\label{prop:2-cyc}
    (a) Сумма $2$-циклов является $2$-циклом.

    (b) Множество граней является $2$-циклом тогда и только тогда, когда сумма границ этих граней нулевая.
\end{proposition}
 
\begin{proposition}\label{l:gencom}
    (a) Любой $2$-цикл является суммой нескольких тетраэдров.
    
    (b) Количество $2$-циклов для $[n]$ равно $2^{\binom{n-1}{3}}$.
\end{proposition}

\begin{proof}[Набросок доказательства]
    Аналогично доказательствам соответствующих пунктов утверждения~\ref{l:gencom1}.
    
    (a) Достаточно показать, что произвольный $2$-цикл $C$ равен сумме 
    $$\widehat C:= \sum\limits_{\substack{\{ i, j, k\} \in C \\ i,j,k<n}} T_{\{i,j,k,n\}}$$
    тетраэдров $T_{\{i,j,k,n\}}$ по всем $3$-элементным подмножествам $\{i, j, k\} \in C$, не содержащим $n$.
    Каждая грань в $ C + \widehat C $ содержит $n$.
    Заметим, что если все грани $2$-цикла имеют общую вершину, то этот $2$-цикл пустой.
    Поэтому $C + \widehat C = 0$, т.~е. $C=\widehat C$.

    (b) Доказывается построением биекции между множеством $2$-циклов в $[n]$ и множеством подмножеств двумерных граней в $[n-1]$, аналогично утверждению~\ref{pr:orinumber2}.a. 
(Эта биекция дает единственность разложения на суммы тетраэдров вида $T_{\{i,j,k,n\}}$ в п.~(a).)
\end{proof}


\begin{pr}[загадка]\label{p:high}
    Придумайте и докажите многомерные аналоги утверждений \ref{prop:2-cyc}-\ref{l:gencom}.
\end{pr}

Частный случай $l=2$ следующей конструкции соответствует утверждениям~\ref{pr:knn1} и \ref{pr:knn2}.
 
\begin{pr}\label{pr:rook}
    Назовем~\emph{рядом} подмножество множества~$[n]^\ell$ векторов длины $\ell$, получаемое фиксацией всех координат, кроме одной. 
    Определим \emph{ладейный цикл} как подмножество множества~$[n]^\ell$, содержащее четное количество векторов в каждом ряду.
    Например, пустое множество является ладейным циклом.
    
    Назовем \emph{параллелепипедом} подмножество $P_1\times\ldots\times P_{\ell}\subset[n]^\ell$, где $P_i$ есть $2$-элементные подмножества множества $[n]$.
    Очевидно, что любой параллелепипед является ладейным циклом.
    
    (a) Сумма ладейных циклов является ладейным циклом.
    
    (b) Если в ладейном цикле нет элементов из $[n-1]^\ell$, то этот ладейный цикл пустой. 
    
    (c) Любой ладейный цикл является суммой нескольких параллелепипедов. 
    
    (d) Сколько имеется ладейных циклов в $[n]^\ell$?
\end{pr}

\begin{proof}[Доказательство п.~(с).]
    Для $a\in[n-1]^\ell$ обозначим через $P(a):=\{n,a_1\}\times\ldots\times\{n,a_\ell\}$ параллелепипед с противоположными вершинами $a$ и $(n,\ldots,n)$. 
    Достаточно доказать, что любой ладейный цикл $C\subset[n]^\ell$ равен сумме $\widehat C$ параллелепипедов $P(a)$ по всем $a\in C\cap[n-1]^\ell$. 
    Сумма $C+\widehat C$ является ладейным циклом.
    Так как $P(a)\cap[n-1]^\ell=\{a\}$, то $(C+\widehat C)\cap[n-1]^\ell=\emptyset$.
    По п.~(b) имеем $C+\widehat C = \emptyset$, т.~е. $C=\widehat C$.
\end{proof}

\begin{proof}[Указание к п.~(d) для $\ell = 2$.]
    Ребро~$ab'\in K_{n,n}$ соответствует паре $(a,b)\in [n]^2$.
    Вершина~$a\in[n]$ соответствует ряду $x=a$,
    вершина~$b'\in[n]'$ соответствует ряду $y=b$.
    Ребро, содержащее вершину, соответствует элементу множества $[n]^2$, принадлежащему ряду.
    Тогда $1$-циклы в~$K_{n,n}$ соответствуют ладейным циклам в $[n]^2$.
    В частности, циклы длины~$4$ соответствуют параллелепипедам (которые при $\ell=2$ можно назвать параллелограммами).
\end{proof}
 
Назовем \emph{$2$-гиперграфом} (двумерным гиперграфом, или $3$-однородным гиперграфом) пару $(V,F)$ из конечного множества~$V$ и семейства~$F$, состоящего из нескольких $3$-элементных подмножеств множества $V$. 
Назовем \emph{гранью} элемент семейства $F$. 
Назовем \emph{ребром} $2$-элементное подмножество, содержащееся в некоторой грани. 
Количества вершин, ребер и граней обозначаются $V$, $E$ и $F$ (причем $V$ и $F$ --- так же, как множества). 

Назовем \emph{$2$-циклом} (симплициальным, по модулю $2$) множество $C$ граней такое, что каждое ребро содержится в четном количестве граней из $C$.  

\begin{remark}\label{p:cychyp}
    Существуют $2$-гиперграф и взаимно однозначное соответствие между множеством его граней и $[n]^3$, которое дает взаимно однозначное соответствие между $2$-циклами и ладейными циклами (см.~задачу \ref{pr:rook}).
    
    Построим такой $2$-гиперграф. 
    В качестве его вершин возьмем $[n]\times[3]$.
    Грани образуются тройками вершин из различных строк.   
    Взаимно однозначное соответствие переводит грань $\{(a,1),(b,2),(c,3)\}$ в вершину $(a,b,c) \in [n]^3$.
    Например, октаэдр в $2$-гиперграфе переходит в параллелепипед в $[n]^3$.
\end{remark}

Назовем $2$-гиперграф \emph{связным}, если связен граф, являющийся объединением его ребер.

\begin{example}[ср.~утверждения~\ref{pr:orinumber2}.a и \ref{p:homol}] \label{p:example} 
    Существуют два связных $2$-гиперграфа, имеющие одинаковое количество вершин, ребер и граней, но разное количество $2$-циклов.
\end{example}

\begin{proof}[Набросок доказательства]
    Возьмем центрально-симметричную триангуляцию $T$ квадрата, настолько мелкую, что никакой треугольник, примыкающий к границе квадрата, не пересекает симметричный ему треугольник. 
    Обозначим через $T'$ гиперграф, полученный из гиперграфа $T$ склейкой центрально-симметричных вершин на границе квадрата. При этом  склеиваются центрально-симметричные ребра на границе квадрата. 
    ($T'$ "--- триангуляция \emph{проективной плоскости}.) 
    В гиперграфе $T$ нет ненулевых $2$-циклов. 
    В гиперграфе $T'$ множество всех граней образует ненулевой $2$-цикл. 
    
    У гиперграфов $T$ и $T'$ одинаковы количества $F$ граней и \emph{эйлеровы характеристики} $V-E+F$. 
    
    (1) К гиперграфу можно добавить грань, имеющую с ним ровно одну общую вершину. 
    
    При этом количество вершин увеличится на 2, ребер "--- на 3, граней "--- на 1.  
    
    (2) К гиперграфу можно добавить грань, имеющую с ним ровно одно общее ребро. 
    
    При этом количество вершин увеличится на 1, ребер "--- на 2, граней "--- на 1.  
    
    Эти операции не меняют количества $2$-циклов. 
    
    Применим операции (1) к тому из гиперграфов $T,T'$, в котором меньше ребер, и такое же число операций (2) к другому. 
    Добьемся равенства количеств ребер (сохраняя равенство количеств граней и эйлеровых характеристик). 
    Получим нужные гиперграфы. 
\end{proof}

\subsection{Соотношения между циклами в гиперграфах *}\label{s:cgra3}

В этом разделе мы затронем два независимых сюжета: утверждения \ref{P:parl}-\ref{p:homol} и \ref{pr:rel}-\ref{pr:rook2}.

\begin{pr} \label{P:parl}
 В~парламенте из $n$~человек имеется несколько (попарно различных по составу) комиссий по три~человека в~каждой. 
Известно, что каждый человек находится в~некоторой комиссии.  
Известно также, что если два человека находятся в~некоторой комиссии, то множество из этих двух человек содержится ровно в~двух комиссиях.
Такие две комиссии называются \emph{смежными}.
Еще известно, что любые две комиссии соединяются цепочкой комиссий, в которой соседние комиссии смежны.
Докажите, что число комиссий не меньше $2n-4$.
\end{pr}

Вот переформулировка и решение на языке гиперграфов. 
Произвольный $2$-гиперграф называется \emph{гранесвязным}, если любые две грани можно соединить цепочкой граней, в которой любые две соседние грани имеют общее ребро. 
 
\begin{proposition}[экстремальное свойство эйлеровой характеристики поверхностей]\label{p:eulex} 
Дан гранесвязный $2$-гиперграф без изолированных вершин, имеющий $V$ вершин и $F$ граней. 

(a) Если к каждому ребру примыкает ровно две грани, то $F\ge2V-4$. 
(Так как в границе каждой грани ровно три ребра, то $2E=3F$, где $E$ "--- количество ребер.
Ввиду этого равенства, неравенство $F\ge2V-4$ равносильно неравенству $V-E+F\le2$.)
     
(b) Если к каждому ребру примыкает не более двух граней, а к некоторому "--- ровно одна грань, то $V-E+F\le1$, где  $E$ --- количество ребер.
\end{proposition}

\emph{Одномерным остовом} 2-гиперграфа $(V, F)$ называется граф $(V, E)$, где $E$ --- множество ребер.

\begin{proof}[Набросок доказательства утверждения \ref{p:eulex}]\footnote{Это рассуждение можно считать алгебраической формализацией геометрической идеи, изложенной в \cite[доказательство утверждения 2.4.1d']{Sk20}, прямая реализация которой непроста \cite[\S5.1, 5.8]{Sk20}.}
    П.~(a) доказывается аналогично п.~(b) или сводится к нему. 
    Докажем п.~(b). 
    В его предположениях сумма любого непустого набора границ граней непуста. 
    Так как гиперграф гранесвязен, то его одномерный остов связен.
    Поэтому ввиду отсутствия изолированных вершин и утверждения \ref{pr:orinumber2}.a выполнено $2^F\le 2^{E-V+1}$. 
\end{proof}

Назовем \emph{$1$-циклом} в 2-гиперграфе 1-цикл в его одномерном остове.
Назовем два $1$-цикла в 2-гиперграфе \emph{гомологичными}, если их сумма является суммой границ некоторых граней (ср. с определением перед задачей \ref{p:k21}). 

\begin{pr}\label{pr:homol} В полном 2-гиперграфе любые два 1-цикла гомологичны.
\end{pr}

\begin{proposition}\label{p:homol}
    (a) Количество классов гомологичности $1$-циклов и количество $2$-циклов в любом гиперграфе является степенью числа 2.
    
    (b) Дан $2$-гиперграф c $V$ вершинами, $E$ ребрами и $F$ гранями.
    Обозначим через $b_0$ количество компонент связности, через $2^{b_1}$ количество классов гомологичности $1$-циклов, через $2^{b_2}$ количество $2$-циклов.
    Тогда $b_0-b_1+b_2=V-E+F$.
\end{proposition}
    
\begin{proof}[Набросок доказательства]
    Число $2^{b_1}$ равно отношению количества $2^{E-V+b_0}$ всех $1$-циклов (обобщение утверждения \ref{pr:orinumber2}.a) к количеству $2^{F-b_2}$ тех $1$-циклов, которые гомологичны нулю (т.~е. равны сумме границ некоторых граней).\footnote{Утверждение \ref{p:homol} достаточно доказать для связных $2$-гипеграфов. 
    Однако приведенное доказательство (без такой редукции) интересно тем, что обобщается на многомерный случай \cite[Утверждение~10.6.8]{Sk20}.}
\end{proof}

\begin{pr}\label{pr:rel}
    (a) Для любого $5$-элементного подмножества $A\subset[n]$ выполнено $\sum\limits_{j\in A}T_{A-\{j\}}=0$.  

    (b) Любое линейное соотношение между тетраэдрами является суммой нескольких соотношений из~п.~(a).

    Строгая формулировка п.~(b) аналогична строгой формулировке утверждения~\ref{pr:rel1}.b и получается из последней заменой чисел $2,3,4$ на числа $3,4,5$, соответственно (мы отождествляем тетраэдры в $[n]$ с $4$-элементными подмножествами $[n]$).
    П.~(b) является многомерным аналогом утверждения~\ref{l:gencom}.b; доказательство аналогично.
\end{pr}

 \begin{proposition} \label{pr:rook2}
    (a) Для любых попарно различных $a,b,c\in[n]$ и параллелепипеда $P\subset[n]^{\ell-1}$ выполнено $P\times\{a,b\}+P\times\{b,c\}+P\times\{c,a\}=0$.  
    
    (b) Любое линейное соотношение между параллелепипедами в $[n]^\ell$ является суммой нескольких соотношений из п. (a) и соотношений, получаемых из п. (a) перестановками координат.
    \emph{(Строгая формулировка аналогична строгим формулировкам утверждений~\ref{pr:knn2}.b и~\ref{pr:rel}.b.)}
\end{proposition}

\begin{proof}[Доказательство п.~(b).]
    Любой параллелепипед~$P$ с некоторым $P_i\subset[n-1]$ является суммой нескольких заданных соотношений и двух параллелепипедов, полученных из~$P$ заменой на~$n$ одного из двух элементов в~$ P_i$.
    Значит, в каждом из этих двух параллелепипедов количество пар $P_j$, не содержащих $n$, меньше, чем в~$P$.
    Следовательно, любое соотношение между параллелепипедами является суммой нескольких заданных соотношений и соотношения $P(a_1)+\ldots+P(a_s)=0$ для некоторых попарно различных $a_1,\ldots,a_s\in[n-1]^ {\ell}$.
    В последнем соотношении нет слагаемых, поскольку
    \[\emptyset=(P(a_1)+\ldots+P(a_s))\cap[n-1]^{\ell}=\{a_1,\ldots,a_s\}.
    \]
    Здесь второе равенство выполняется в силу $P(a)\cap[n-1]^{\ell}=\{a\}$. 
\end{proof}

\subsection{Двумерные циклы в произведении графов}\label{s:cydepr}

\emph{(Комбинаторным) произведением} графов $K$ и $L$ называется множество 
\[
    K\times L := \{(\sigma,\tau)\ :\ \text{$\sigma,\tau$ являются ребрами графов $K,L$ соответственно}\}
\]
упорядоченных пар $(\sigma, \tau)$ ребер $\sigma, \tau$ графов $K,L$ соответственно. 

Это определение не является общепринятым (обычно так обозначают декартово произведение графов,  см. \S\ref{s:prod}). 
Такие комбинаторные произведения можно формально свести к гиперграфам, определенным в \S\ref{s:chygra1}.  
Но читателю незачем тратить время на формализацию связи между аналогичными объектами, пока такая формализация не потребуется для какого-нибудь яркого результата.

В этом тексте случай $K\ne L$ используется только для определения тора и в утверждении \ref{p:tree2}.b. 

Пары из $K \times L$ будем называть \emph{клетками}.
Клетки $(\alpha,\beta)$ и $(\sigma, \tau)$ в $K \times L$ называются \emph{смежными}, если 

$\bullet$ либо $\alpha=\sigma$ и $\beta,\tau$ имеют общую вершину,

$\bullet$ либо $\beta=\tau$ и $\alpha,\sigma$ имеют общую вершину\footnote{\label{fn:box}Т.~е. если эти клетки имеют <<общее ребро>> в $K \Box L$. 
Тогда если склеить смежные клетки по их <<общим ребрам>> в $K \Box L$, то получится произведение, обсуждаемое в~\S\ref{s:prod}.}.

\emph{Квадратом} графа $K$ называется $K^2 := K \times K$. 
\textbf{Клеточным $2$-циклом} (в $K^2$) называется подмножество $C \subset K^2$ такое, что для каждой вершины $a$ и ребра $\beta$ графа $K$

$\bullet$ имеется четное количество ребер $\alpha$ в $K$ таких, что $\alpha\ni a$ и $(\alpha,\beta)\in C$, а также 

$\bullet$ имеется четное количество ребер $\alpha$ в $K$ таких, что $\alpha\ni a$ и $(\beta,\alpha)\in C$.

Иными словами, $C\subset K^2$ является клеточным $2$-циклом, если к любому ребру <<примыкает>> четное число клеток из $C$.

Сравните следующие свойство \ref{p:ce2cy}.(iii) и утверждение \ref{p:2cyclesum} с аналогичными утверждениями \ref{prop:2-cyc}, \ref{pr:rook}.a. 
 
\begin{proposition}\label{p:ce2cy} 
    Следующие условия равносильны для подмножества $C \subset K^2$:
    
    \quad(i) $C$ является клеточным $2$-циклом; 
    
    \quad(ii) для каждого ребра $\sigma$ в $K$ оба множества (<<вертикальное сечение>> и <<горизонтальное сечение>>) 
        \[
            C_{\sigma,\cdot}:=\bigl\{\tau\ :\ (\sigma,\tau)\in C\bigr\}\quad\text{и}\quad 
            C_{\cdot,\sigma}:=\bigl\{\tau\ :\ (\tau,\sigma)\in C\bigr\}
        \]
        являются $1$-циклами в $K$; 
    
    \quad(iii) сумма границ $\sigma\Box\tau$ по всем $(\sigma,\tau)\in C$ равна нулю.  
\end{proposition}

\begin{proof}[Доказательство]  
Для вершины $v$ графа $K$ обозначим $\delta v := \{\tau\ :\ \tau \ni v\}$. 
Условие (i) равносильно условию 

(i') $|C \cap (\delta v \times \sigma)|$ и $|C \cap (\sigma \times \delta v)|$ четны для любых ребра $\sigma$ и вершины $v$ графа $K$. 

\smallskip
\emph{Доказательство равносильности (ii) $\Leftrightarrow (i')$}. 
Условие (ii) равносильно условию 

(ii') $|C_{\cdot, \sigma} \cap \delta v|$ и $|C_{\sigma, \cdot} \cap \delta v|$ четны для любых ребра $\sigma$ и вершины $v$ графа $K$.

Последнее условие равносильно условию (i') ввиду равенств
$$|C \cap (\delta v \times \sigma)| = |C_{\cdot, \sigma} \cap \delta v|\quad\text{и}\quad|C \cap (\sigma \times \delta v)| = |C_{\sigma, \cdot} \cap \delta v|$$
для любых ребра $\sigma$ и вершины $v$ графа $K$.

\smallskip
\emph{Доказательство равносильности (iii) $\Leftrightarrow (i')$}. Обозначим через $S$ сумму границ $\sigma \Box \tau$ по всем $(\sigma, \tau) \in C$. Условие (iii) равносильно условию (i') ввиду равносильностей
$$v \Box \sigma \not \in S \Leftrightarrow |C \cap (\delta v \times \sigma)| \text{ четно} \quad\text{и}\quad 
\sigma \Box v \not \in S \Leftrightarrow |C \cap (\sigma \times \delta v)| \text{ четно}$$
для любых ребра $\sigma$ и вершины $v$ графа $K$.
\end{proof}

Далее под клеточным $2$-циклом будем понимать подмножество квадрата $K^2$, удовлетворяющее свойству \ref{p:ce2cy}(ii).

\begin{proposition}\label{p:2cyclesum} 
    Сумма по модулю $2$ клеточных $2$-циклов является клеточным $2$-циклом.
\end{proposition}

\emph{Тором} в $K^2$ называется произведение простых циклов в $K$. 
 
\begin{proposition}\label{t:torus} 
    (a) Тор является клеточным $2$-циклом.
    
    (b)  В торе имеется ровно один непустой клеточный $2$-цикл.
\end{proposition} 

\begin{proof}[Доказательство.] 
    Возьмем простые циклы $Z_1, Z_2$ в $K$ и тор $Z := Z_1 \times Z_2 \subset K^2$.  
    
    (a) Для любого ребра $\sigma$ графа $K$ множество $Z_{\sigma,\cdot}$ либо равно $Z_2$, либо пусто; множество $Z_{\cdot, \sigma}$ либо равно $Z_1$, либо пусто.
    
    (b) Возьмем непустой клеточный $2$-цикл $C$ в торе $Z$.
    Тогда найдется клетка $(\sigma, \tau) \in C$.
    Выберем произвольную клетку $(\alpha, \beta) \in Z$.
    Так как сечение $C_{\cdot, \tau}$ непусто, то $C_{\cdot, \tau} = Z_1$. 
    В частности, $(\alpha, \tau) \in C$.
    Так как сечение $C_{\alpha, \cdot}$ непусто, то $C_{\alpha, \cdot} = Z_2$.
    В частности, $(\alpha, \beta) \in C$. 
    Следовательно, $C = Z$.
\end{proof}

\begin{proposition}[ср. утверждение \ref{pr:orinumber2}.b]\label{p:tree2} 
    Для дерева $T\subset K$ любой клеточный $2$-цикл, содержащийся в  
    
    (a) $T^2$; \quad (b) $\overline T := T\times K\cup K\times T$,
    
    является пустым. 
\end{proposition} 

\begin{proof}
    (a) Пусть $C$ "--- клеточный $2$-цикл в $T^2$. 
    Поскольку $T$ "--- дерево, из утверждения~\ref{pr:orinumber2}.b следует, что $C_{\sigma,\cdot}=0$ для любого ребра $\sigma$ дерева $T$. 
    Поэтому $C=0$. 
    
    (b) Пусть $C$ "--- клеточный $2$-цикл в $\overline T$.
    Поскольку $T$ "--- дерево, из утверждения~\ref{pr:orinumber2}.b следует, что $C_{\sigma,\cdot}=0$ для любого ребра $\sigma$ графа $K \setminus T$. 
    Поэтому $C\subset T\times K$. 
    Поскольку $T$ "--- дерево, из утверждения~\ref{pr:orinumber2}.b следует, что $C_{\cdot,\sigma}=0$ для любого ребра $\sigma$ графа $K$. 
    Поэтому $C=0$.
\end{proof}

Нетривиальные примеры клеточных $2$-циклов приведены в утверждении \ref{t:dpcycle}. 

\begin{proposition}\label{t:kunn} 
    (a) Любой клеточный $2$-цикл в квадрате графа является суммой нескольких торов. 
         
    (b) Для связного графа $K$ с $V$ вершинами и $E$ ребрами количество клеточных $2$-циклов в $K^2$ равно $2^{(E-V+1)^2}$. 
\end{proposition}

Формулировки этого утверждения и следующей леммы, а также их доказательства, аналогичны формулировкам и доказательствам утверждений \ref{pr:simsum}, \ref{pr:orinumber2}.a и леммы \ref{lem:sumcyc}.

Обозначим через $T$ остовное дерево в связном графе $K$.
Для ребра $\sigma \in K\setminus T$ обозначим через $\widehat\sigma = \widehat\sigma_T$ простой цикл в $K$, образованый ребром $\sigma$ и простым путем в $T$, соединяющим концы ребра $\sigma$.
  
\begin{lemma}\label{lem:sum2cyc}
    Для любых остовного дерева $T$ в связном графе $K$ и клеточного $2$-цикла $C$ в $K^2$
    \[
        C = \widehat C := \displaystyle\sum_{(\sigma,\tau)\in C\setminus\overline T} \widehat\sigma\times\widehat\tau.
    \]
\end{lemma}

Обозначим через $H_2(K^2)$ множество всех клеточных $2$-циклов в $K^2$ с операцией сложения. 

\begin{proof}[Доказательство утверждения~\ref{t:kunn}]
    П.~(а) следует из леммы \ref{lem:sum2cyc}, примененной к каждой компоненте связности графа.
    
    (b) Ввиду того, что $K^2 \setminus \overline T = (K \setminus T)^2$, достаточно доказать, что существует взаимно однозначное соответствие между $H_2(K^2)$ и множеством $2^{K^2\setminus \overline T}$
    подмножеств ребер из $K^2\setminus \overline T$.
    Определим отображения
    $$\varphi\colon H_2(K^2) \to 2^{K^2\setminus\overline T} \quad\text{формулой}\quad \varphi C := C\setminus\overline T,\quad$$
    $$\widehat\varphi \colon 2^{K^2\setminus\overline T} \to H_2(K^2) \quad\text{формулой}\quad \widehat\varphi D := \sum\limits_{(\sigma,\tau) \in D} \widehat\sigma\times\widehat\tau.$$
    Поскольку $\varphi(\widehat \sigma \times \widehat \tau) = \{(\sigma, \tau)\}$, то $\varphi\widehat\varphi D = D$ для любого $D \subset K^2\setminus\overline T$.
    Обратно, для любого $C \in H_2(K^2)$ имеем $\widehat \varphi \varphi C = \widehat{C} = C$, где второе равенство есть лемма \ref{lem:sum2cyc}.
    Таким образом, $\varphi$ и $\widehat \varphi$ являются взаимно однозначными соответствиями.
\end{proof}

Базис множества клеточных $2$-циклов определяется аналогично случаю $1$-циклов, см. \S\ref{s:cgra1}.

\begin{theorem}[Кюннет]\label{t:kunneth} 
    Для 1-циклов $C_1, \ldots, C_q$ в $H_1(K)$ обозначим 
    $$B = B(C_1, \ldots, C_q): = \{C_i \times C_j\ :\ i, j \in [q] \}.$$

    (a) Существует базис $C_1,\ldots,C_q$ в $H_1(K)$, такой что $B(C_1, \ldots, C_q)$ "--- базис в $H_2(K^2)$. 
    
    (b) Если $C_1,\ldots,C_q$ "--- базис в $H_1(K)$, то $B(C_1, \ldots, C_q)$ "--- базис в $H_2(K^2)$.
    
    (Т.~е. $H_2(K^2)\cong H_1(K)\otimes H_1(K)$.)
\end{theorem} 

\begin{proof}[Набросок доказательства]
    (a) В качестве базиса $C_1, \ldots, C_q$ можно взять множество всех $1$-циклов $\widehat\sigma$, где $\sigma \in K \setminus T$.
    Из леммы \ref{lem:sum2cyc} следует, что $B$ составляет систему порождающих в $H_2(K^2)$. Можно проверить, что $B$ --- базис.
    
    (b) Любой клеточный $2$-цикл в $K^2$ есть сумма некоторых клеточных $2$-циклов из $B$,  и такое представление единственно.
    Мы доказываем это в следующих двух абзацах соответственно.
    
    Каждый из простых циклов $Z_1, Z_2$ в $K$ является суммой некоторых $1$-циклов из $C_1, \ldots, C_q$. 
    Поэтому каждый тор $Z_1 \times Z_2 \subset K^2$ является суммой некоторых клеточных $2$-циклов из $B$. 
    По утверждению~\ref{t:kunn}.a то же справедливо для произвольного клеточного $2$-цикла.

    Единственность достаточно доказать для связного графа $K$.
    Поскольку $C_1,\ldots, C_q$ "--- базис в $H_1(K)$, то $2^q = |H_1(K)|$. По утверждению \ref{pr:orinumber2}.a имеем $q = E - V + 1$.
    Следовательно, $|B| = q^2 = (E - V + 1)^2$.
    По утверждению~\ref{t:kunn}.b получаем $|H_2(K^2)| = 2^{(E - V + 1)^2}$.
    Это число равно числу линейных комбинаций клеточных $2$-циклов из $B$. 
    Поэтому представление единственно. 
\end{proof}


\subsection{Двумерные циклы во взрезанном квадрате графа}\label{s:2delsqu}

Говоря нестрого (ср. с \S\ref{s:prod}), \textit{взрезанный квадрат} графа $K$ "--- множество упорядоченных пар $(x,y)$ точек графа $K$ таких, что $x,y$ не принадлежат соседним ребрам.
Как и квадрат графа, это множество можно представить в виде объединения прямоугольников, см. сноску \ref{fn:box}.
Строго говоря, \textbf{комбинаторным взрезанным квадратом} графа $K$ называется
\[
K^{\underline2} := \{(\sigma,\tau)\ :\ \sigma,\tau\text{ "--- несмежные ребра графа $K$}\}.
\]

Например, произведение любых вершинно-непересекающихся циклов в $K$ является клеточным $2$-циклом в $K^{\underline2}$ (по утверждению~\ref{t:torus}.a).

\begin{remark}\label{remark:2delsqu} 
    (a) Комбинаторный взрезанный квадрат и его обобщения имеют многочисленные  приложения. 
    Например, для распознавания реализуемости гиперграфов в многомерных евклидовых пространствах или для различения их реализаций, см. обзоры \cite{Sk06, Sk18}.

    (b) Подмножество $K^{\underline2}\subset K^2$ меньше, чем $K^2$, поэтому иногда его легче нарисовать. 
    Например, $K_2^{\underline2}=K_3^{\underline2}=K_{n,1}^{\underline2}=\emptyset$. 

    (c) Склеим смежные клетки вдоль их <<общего ребра>>.
    Тогда комбинаторные взрезанные квадраты пути на 5 вершинах, цикла на $5$ вершинах, графа $K_4$, графа $K_{3,3}$, графа $K_5$ <<выглядят как>> несвязное объединение двух дисков, кольцо, кубооктаэдр без треугольных граней (рис.~\ref{f:del3}), сфера с четырьмя ручками, сфера с шестью ручками соответственно. 
    Действительно, для $K=K_{3,3}$ и $K=K_5$ можно доказать, что $K^{\underline 2}$ есть \emph{замкнутая связная ориентируемая двумерная поверхность} (ср. с утверждениями \ref{p:eulex} и \ref{t:dpcycle}), и применить теорему о классификации двумерных поверхностей (см.~подробности в~\cite[текст после~3.4.1]{Sa91}). 
\end{remark}

Из этого раздела далее используются только определение комбинаторного взрезанного квадрата и лемма \ref{t:2cyc-bij}.b.

\begin{pr}\label{t:dpcycle0} 
    Любой клеточный $2$-цикл в $K^{\underline2}$ пуст для 
    
    (a) простого цикла $K$; \quad (b) $K=K_{n,2}$; \quad (c) $K=K_4$;  
    
    (d) \emph{колеса} $K$, т.~е. для графа с множеством $\{0\}\cup[n]$ вершин, а также ребрами $\{n,1\}$, $\{0,j\}$ и $\{j,j+1\}$ для $j\in[n-1]$.
\end{pr} 

\begin{proof}[Набросок доказательства п.~(a-c)] 
    Пусть $C$ "--- клеточный $2$-цикл в $K^{\underline 2}$.
    Для любого ребра $\sigma$ графа $K$ сечение $C_{\sigma, \cdot}$ является подмножеством
    
    (a) графа-пути $K - \sigma$; \quad
    
    (b) графа $K_{n, 1} \cong K - v$ для некоторой вершины $v \in \sigma$;
    
    (c) множества $\{\tau\}$, где $\tau$ "--- единственное несмежное ребро с $\sigma$.
    
    Следовательно, $C_{\sigma, \cdot}$ "--- подмножество дерева.
    Тогда по утверждению \ref{pr:orinumber2}.b получаем $C_{\sigma, \cdot} = 0$.
    Тогда~$C~=~0$.

    \emph{Набросок доказательства п.~(d).} Пусть $C$ "--- клеточный $2$-цикл в $K^{\underline 2}$. 
    
    Так как и $C_{0j, \cdot}$, и $C_{\cdot, 0j}$ "--- подмножества графа-пути $K - 0 - j$ для любого $j \in [n]$, то $C_{0j, \cdot} = C_{\cdot, 0j} = 0$ для любого $j \in [n]$.
    Тогда $C\subset(K - 0)^{\underline 2}$ для графа-цикла $K - 0$, что сводит доказательство к п.~(a).
    
    (Это решение обобщает п.~(c), так как $K_4$ является колесом при $n = 3$.)
\end{proof}

\begin{proposition}\label{t:dpcycle} 
    (a) Подмножество $K_{3,3}^{\underline2}$ в $K_{3,3}^2$ является клеточным $2$-циклом.
    
    (b) Подмножество $K_{5}^{\underline2}$ в $K_{5}^2$ является клеточным $2$-циклом.
\end{proposition}

\begin{proof}[Доказательство] 
    Положим $C := K^{\underline 2}_{3, 3}$ ($C := K^{\underline 2}_5$). 
    Для произвольного ребра $\sigma$ графа $K_{3, 3}$ ($K_5$) каждое из множеств $C_{\sigma, \cdot}$ и $C_{\cdot, \sigma}$ есть цикл длины~4 (длины~3).
\end{proof}

\begin{pr}\label{t:dpcycle2} 
    (a) В $K_{3,3}^{\underline2}$ имеется ровно один непустой клеточный $2$-цикл.
    
    (b) В $K_{5}^{\underline2}$ имеется ровно один непустой клеточный $2$-цикл.
\end{pr} 

\begin{proof}[Набросок доказательства п.~(a)] (Ср. с определением гранесвязности перед утверждением \ref{p:eulex}.)
    Пусть $C$ "--- непустой клеточный $2$-цикл в $K^{\underline 2}_{3, 3}$.
    Тогда найдется пара несмежных ребер $\sigma, \tau$ такая, что $(\sigma, \tau) \in C$.
    Следовательно, $\tau \in C_{\sigma, \cdot}$.
    Выберем произвольную клетку $(\alpha, \beta) \in K^{\underline 2}_{3, 3}$.
    Для любых двух ребер $\sigma, \alpha$ в 
    $K_{3,3}$ существует ребро $\gamma$, такое что $\gamma$ не смежно ни с $\sigma$, ни с $\alpha$.
    Любое сечение клеточного $2$-цикла $C$ либо пусто, либо является циклом длины $4$. Так как сечение $C_{\sigma, \cdot}$ непусто, то $\gamma \in C_{\sigma, \cdot}$, а потому $(\sigma, \gamma) \in C$.
    Так как сечение $C_{\cdot, \gamma}$ непусто, то $\alpha \in C_{\cdot, \gamma}$, а потому $(\alpha, \gamma) \in C$.
    Так как сечение $C_{\alpha, \cdot}$ непусто, то $\beta \in C_{\alpha, \cdot}$, а потому $(\alpha, \beta) \in C$.
\end{proof}
 


Представления клеточного $2$-цикла $K_{3,3}^{\underline2} \subset K_{3,3}^2$ ($K_5^{\underline2} \subset K_5^2$) в виде суммы нескольких торов в $K_{3,3}^2$ (в $K_5^2$) можно получить из утверждения \ref{t:kunn}.a.
Поскольку в графе $K_{3,3}$ ($K_5$) нет двух непустых вершинно-непересекающихся циклов, то этот клеточный $2$-цикл не является суммой произведений вершинно-непересекающихся циклов.

Граф $\t{K_n}$ определен перед утверждением \ref{l:tkn}. 

\begin{lemma}\label{t:2cyc-bij} 
    (a) Существует взаимно однозначное соответствие между $K_{n,n}^{\underline2}$ и $\t{K_n}^2$, сохраняющее смежность клеток.
    
    (b) Существует взаимно однозначное соответствие $H_2(K_{n,n}^{\underline2}) \rightarrow H_2(\t{K_n}^2)$.
\end{lemma}

\begin{proof}[Набросок доказательства]
    (a) Определим отображение $f\colon K_{n, n}^2 \rightarrow K_{n,n}^2$  формулой
    $$f(\sigma_1\sigma_2',\tau_1\tau_2') := (\sigma_1\tau_1',\sigma_2\tau_2').$$
    (Напомним, что для $a,b\in[n]$ мы обозначаем через $ab'$ ребро, соединяющее вершины $a$ и $b'$.)
    Отображение $f$ является взаимно однозначным соответствием, так как $f^2 = \id$.
    Докажем в следующем абзаце, что $f$ сохраняет смежность клеток.
    
    Рассмотрим пару смежных клеток $(\sigma, \alpha)$ и $(\tau, \beta)$.
    По условию смежности, они имеют общее ребро.
    Без ограничения общности, будем считать, что $\alpha = \beta$.
    По условию смежности, ребра $\sigma$ и $\tau$ имеют общую вершину.
    Обозначим через $\sigma_1, \sigma_2$ и $\tau_1, \tau_2$ концы ребер $\sigma$ и $\tau$ соответственно.
    Тогда либо $\sigma_1 = \tau_1$, либо $\sigma_2 = \tau_2$.
    В обоих случаях клетки $f(\sigma, \alpha)$ и $f(\tau, \beta)$ смежны.
(Например, в случае $\sigma_1 = \tau_1$ клетки $f(\sigma,\alpha)$ и $f(\tau,\alpha)$ имеют общее ребро $\sigma_1\alpha_1' = \tau_1\alpha_1'$.)
     
    Следующие условия эквивалентны для ребер $\sigma, \tau \in K_{n,n}$ :
    
    $\bullet~ (\sigma, \tau) \in K_{n, n}^{\underline{2}}$;
    
    $\bullet~ \sigma$ и $\tau$ несмежны;
    
    $\bullet~ \sigma_1 \neq \tau_1$ и $\sigma_2 \neq \tau_2$;
    
    $\bullet~ f(\sigma, \tau)
    \in \t{K_n}^2$.
    
    Поэтому сужение отображения $f$ на $K_{n, n}^{\underline{2}}$ является взаимно однозначным соответствием между $K_{n, n}^{\underline{2}}$ и $\t{K_n}^2$.
    
    (b) Взаимно однозначное соответствие $g$ из п.~(a) переводит клеточные $2$-циклы в клеточные $2$-циклы, поскольку сохраняет смежность.
    Поэтому оно индуцирует отображение $\widehat g\colon H_2(K_{n,n}^{\underline2})\to H_2(\t{K_n}^2)$.
    Отображение $\widehat g$ является взаимно однозначным соответствием, поскольку $g$ является таковым.
\end{proof}

\begin{remark} \label{remark:2cycl-bij}
    (a) Пусть дан набор квадратов с ориентированными сторонами. 
    По разбиению этих сторон на пары построим поверхность склейкой сторон в каждой паре.
    Будем говорить, что два квадрата \emph{смежны при разбиении}, если какие-то их стороны принадлежат одной паре. 
    Оказывается, что имеются два таких набора из одинакового числа квадратов и два их разбиения, при которых 
    
    $\bullet$ построенные по этим разбиениям поверхности <<различны>> (т.е. не гомеоморфны), но 
    
    $\bullet$ квадраты \emph{одинаково смежны} (т.е. можно так занумеровать квадраты первого набора и квадраты второго набора, что квадраты $k$ и $l$ в одном наборе смежны тогда и только тогда, когда квадраты $k$ и $l$ в другом наборе смежны).  
    
    (b) Такие разбиения строятся, например, при помощи леммы \ref{t:2cyc-bij}.a.
    По ней <<тела>> клеточных $2$-циклов $K_{n,n}^{\underline2}$ и $\t{K_n}^{2}$ (ср. с замечанием \ref{r:body}.a) получены склейкой одинакового числа одинаково смежных квадратов.
    (Эти тела получены из $K_{n,n}^{\square \underline2}$ и $\t{K_n}^{\square 2}$ <<заклейкой>> каждой границы квадратом.)
    В следующих двух абзацах мы объясним, почему эти тела различны для $n = 3$.
    
    Геометрически, тела клеточных $2$-циклов $\t{K_3}^{2}$ и $K_{3,3}^{\underline2}$ являются тором и сферой с четырьмя ручками соответственно (см. замечание \ref{remark:2delsqu}.c).
    
    Обосновать предыдущее предложение не так просто, но эти тела можно различить алгебраически. 
    Достаточно посмотреть на инвариант тел "--- количество $1$-циклов по модулю границ в $\t{K_3}^{\square 2}$ и в $K_{3,3}^{\square \underline2}$.
    В $\t{K_3}^{\square 2}$ имеется четыре $1$-цикла по модулю границ (воспользуйтесь теоремой Кюннета~\ref{stcycles2}.b), в $K_{3,3}^{\square \underline2}$ их $256$ (см. ответ на задачу  \ref{p:knn21}.e').
    
    (c) По лемме \ref{t:2cyc-bij}.b и утверждению \ref{t:kunn}.b в $K_{n,n}^{\underline2}$ ровно $2^{(n^2-3n+1)^2}$ клеточных $2$-циклов.
    (Ср. с утверждениями \ref{t:kunnsym}.b, \ref{t:2cyc-sym-lemma}.b и загадкой \ref{p:sawrong}.d.)
\end{remark}


\subsection{Симметричные двумерные циклы в квадрате графа *}\label{s:2symsqu}
 
Рассмотрим симметрию (инволюцию) на $K^2$, переставляющую компоненты (т.~е. переставляющую точки  $(x,y)$ и $(y,x)$).  
Рассмотрим соответствующее отображение на клеточных $2$-циклах.

\begin{proposition}\label{p:tsym} 
    (a) Для любых простых циклов $Q,R$ в $K$ \emph{симметризованный тор} $Q\times R+R\times Q$ является симметричным клеточным $2$-циклом.
    
    (b) Сумма по модулю $2$ симметричных клеточных $2$-циклов является симметричным клеточным $2$-циклом.
    
    (c) Существуют граф и симметричный клеточный $2$-цикл в квадрате этого графа, не являющийся суммой нескольких симметризованных торов.
\end{proposition}

\begin{proof}[Набросок доказательства] 
    (a) Симметризованный тор является клеточным $2$-циклом как сумма двух клеточных 2-циклов (см. утверждения \ref{p:2cyclesum} и \ref{t:torus}).
    Причем он симметричен, поскольку указанная сумма симметрична.

    (c) Неформально говоря, следующий пример основан на том, что в симметризованном торе клетки вида $(\sigma, \sigma)$ сокращаются.

    \emph{Пример:} граф $K_3$ и клеточный $2$-цикл $K_3^2$.\\
    \emph{Доказательство.}
    Для любого ребра $\sigma \in K_3$ имеем $(\sigma, \sigma) \in K_3^2$.
    Обозначим через $C$ произвольную сумму нескольких симметризованных торов в $K^2$.
    Для любого ребра $\sigma \in K$ имеем $(\sigma, \sigma) \not \in C$.
    Поэтому $K_3^2$ не является суммой симметризованных торов.
\end{proof}

\begin{proposition}[ср. утверждение \ref{t:kunn}] \label{t:kunnsym} 
    (a) Любой симметричный клеточный $2$-цикл в квадрате графа является суммой нескольких симметризованных торов и квадратов простых циклов.
    
    (b) Для связного графа $K$ с $V$ вершинами и $E$ ребрами обозначим $q = E - V + 1$. 
    Тогда количество симметричных клеточных $2$-циклов в $K^2$ равно $2^{q(q+1)/2}$.
\end{proposition} 

\begin{lemma}\label{lem:sumsym2cyc}
    Для любых остовного дерева $T$ в связном графе $K$ и симметричного клеточного $2$-цикла $C$ в $K^2$
    $$C = \sum\limits_{\substack{\{\sigma, \tau\}\\ (\sigma, \tau) \in C \setminus \overline T\\ \sigma \neq \tau}}(\widehat \sigma \times \widehat \tau + \widehat \tau \times \widehat \sigma) + \sum\limits_{\substack{\sigma\\ (\sigma, \sigma) \in C \setminus \overline T}} \widehat \sigma \times \widehat \sigma,$$
    где $\widehat \sigma$ определено после утверждений \ref{pr:orinumber2} и \ref{t:kunn}.
\end{lemma}

Лемма \ref{lem:sumsym2cyc} следует из леммы \ref{lem:sum2cyc}, симметричности клеточного $2$-цикла $C$ и симметричности множества $\overline{T}$.

\begin{proof}[Набросок доказательства утверждения~\ref{t:kunnsym}] 
    П.~(а) следует из леммы $\ref{lem:sumsym2cyc}$, примененной к каждой компоненте связности графа.
    
    (b) Возьмем взаимно однозначное соответствие из доказательства утверждения \ref{t:kunn}.b.
    Ввиду леммы \ref{lem:sumsym2cyc} симметричные клеточные $2$-циклы в $K^2$ переходят при этом соответствии в симметричные наборы клеток вне $\overline{T}$.
    Симметричные наборы клеток вне $\overline{T}$ разбиваются на $q$ одноэлементных наборов и $\binom{q}{2}$ пар взаимно симметричных клеток.
    Таким образом, количество симметричных наборов клеток вне $\overline{T}$ равно $2^{q + \binom{q}{2}}$.
    (Ср. с доказательством утверждения \ref{prop:inv}.)
\end{proof}

\begin{proposition} \label{prop:symkunneth} 
    Для 1-циклов $C_1, \ldots, C_q$ в $H_1(K)$ обозначим 
    $$B' = B'(C_1, \ldots, C_q): = \{C_i \times C_j + C_j \times C_i,~ C_i \times C_i\ :\ i\neq j,~ i, j \in [q] \}.$$
    
    (a) Существует базис $C_1, \ldots, C_q$ в $H_1(K)$, такой что $B'(C_1, \ldots, C_q)$ --- базис в множестве симметричных клеточных $2$-циклов в $K^2$.
    
    (b) Если $C_1,\ldots,C_q$ "--- базис в $H_1(K)$, то 
    $B'(C_1, \ldots, C_q)$ "--- базис в множестве симметричных клеточных $2$-циклов в $K^2$.
\end{proposition}

\begin{proof}[Набросок доказательства]
    (a) В качестве базиса $C_1, \ldots, C_q$ можно взять множество всех $1$-циклов $\widehat\sigma$, где $\sigma \in K \setminus T$.
    Из леммы \ref{lem:sumsym2cyc} следует, что $B'$ составляет систему порождающих в множестве симметричных клеточных $2$-циклов.
    Можно проверить, что $B'$ --- базис.
    
    (b) Обозначим через $B$ базис в $H_2(K^2)$, полученный из теоремы Кюннета \ref{t:kunneth}.b по базису $C_1, \ldots, C_q$.
    Возьмем симметричный клеточный $2$-цикл $C$ и рассмотрим его представление в виде суммы некоторых элементов из $B$.
    Клеточный $2$-цикл $C$ симметричный.
    Поэтому если $C_i \times C_j$ встречается в этом представлении, то в нем встречается и $C_j \times C_i$.
    Значит, $C$ представим в виде суммы некоторых клеточных $2$-циклов из $B'$.
    
    Единственность представления в виде суммы некоторых элементов из $B'$ следует из единственности представления в виде суммы некоторых элементов из $B$.
\end{proof}

\begin{pr}\label{pr:del} 
    Любой симметричный клеточный $2$-цикл, содержащийся в $K^{\underline2}$, является суммой нескольких симметризованных торов  (лежащих в $K^2 \supset K^{\underline2}$).
\end{pr}

\begin{proof}
    Поскольку $K^{\underline2} \subset K^2$, то по лемме \ref{lem:sumsym2cyc} любой симметричный клеточный $2$-цикл в $K^{\underline2}$ является суммой нескольких симметризованных торов и нескольких квадратов простых циклов.
    Так как в $K^{\underline2}$ нет клеток $(\sigma, \sigma)$ ни для какого $\sigma \in K$, то в указанной сумме имеются только симметризованные торы.
\end{proof}

Симметрия $t\colon \t{K_n}\to\t{K_n}$ определена перед леммой~\ref{pr:knbij}. 
Симметрия $(t\times t)\colon \t{K_n}^2\to\t{K_n}^2$ определяется формулой $(t\times t)(x,y) := (tx,ty)$. 

\begin{proposition}\label{l:h2sym} 
    (a) Взаимно однозначное соответствие из леммы~\ref{t:2cyc-bij}.b отображает клеточные $2$-циклы в $K_{n,n}^{\underline2}$, переставляемые симметрией, в клеточные $2$-циклы в $\t{K_n}^2$, переставляемые симметрией~$(t \times t)$.
    
    (b) Существует взаимно однозначное соответствие между
    симметричными клеточными $2$-циклами $K_{n,n}^{\underline2}$ и $(t \times t)$-симметричными клеточными $2$-циклами в $\t{K_n}^2$. 
(Значит, количества таких клеточных $2$-циклов совпадают, ср. с леммой \ref{t:2cyc-sym-lemma}.b.)
\end{proposition}

\begin{proof}[Доказательство] 
    (a) Обозначим симметрию на $K^2$ через $s$, а взаимно однозначное соответствие из леммы~\ref{t:2cyc-bij}.b через $f$.
    Достаточно доказать, что для любой клетки $\alpha \in K_{n,n}^{\underline2}$ выполнено $f(s\alpha) = (t \times t)f(\alpha)$.
    Это верно, поскольку
    $$f(s(\sigma_1\sigma_2', \tau_1\tau_2')) = f(\tau_1\tau_2', \sigma_1\sigma_2') = (\tau_1\sigma_1',\tau_2\sigma_2') = (t \times t)(\sigma_1\tau_1',\sigma_2\tau_2') = (t \times t)f(\sigma_1\sigma_2', \tau_1\tau_2').$$
    П.~(b) следует из п.~(а) и леммы \ref{t:2cyc-bij}.b.
\end{proof}

Назовем \emph{$(t\times t)$-симметризованным тором} для $1$-циклов $Q$ и $R$ в $\t{K_n}$ клеточный $2$-цикл $Q\times R+tQ\times tR$ в $\t{K_n}^2$. 

\begin{lemma}\label{t:2cyc-sym-lemma} 
(a) Любой $(t\times t)$-симметричный клеточный $2$-цикл в $\t{K_n}^2$ является суммой нескольких из следующих слагаемых:

    $\bullet$ $(t \times t)$-симметризованные торы $Q\times R+tQ\times tR$ для циклов $Q$ и $R$ длины $4$ в $\t{K_n}$;
         
    $\bullet$ $\t{K_3}^2$.
    
    (b) Количество $(t \times t)$-симметричных клеточных $2$-циклов в $\t{K_n}^2$ равно $2^{ \frac{q^2 + 1}{2}}$, где $q = n^2 - 3n  + 1$.
\end{lemma}

\begin{proof}[Набросок доказательства]
    Обозначим через $B$ базис в $H_1(\t{K_n})$ из леммы \ref{p:symgra}.b.
    По теореме Кюннета \ref{t:kunneth}.b множество $B \times B := \{Q \times R : Q, R \in B\}$ является базисом в $H_2(\t{K_n}^2)$.
    Поскольку $B$ состоит из $\t{K_3}$ и пар простых циклов, взаимно симметричных относительно $t$, то $B \times B$ состоит из $\t{K_3}^2$ и следующих пар торов, взаимно симметричных относительно $t \times t$:
    $$Q \times R \text{ и } tQ \times tR, \quad \t{K_3} \times R \text{ и } \t{K_3} \times tR, \quad Q \times \t{K_3} \text{ и } tQ \times \t{K_3},$$
    где $Q,R$ --- всевозможные циклы из $B$, не равные $\t{K_3}$.
    
    Тогда следующие клеточные $2$-циклы составляют базис $B_{t \times t}^2$ множества всех $(t\times t)$-симметричных клеточных $2$-циклов в $\t{K_n}^2$:
    $$Q\times R+tQ\times tR, \quad\t{K_3}\times(R+tR), \quad(Q+tQ)\times\t{K_3} \quad\text{и}\quad \t{K_3}^2.$$

    По утверждению \ref{pr:orinumber2}.a имеем $|B| = q$.
    Теперь п. (b) следует из того, что $|B_{t \times t}^2| = \frac{(q-1)^2}{2} + 2\left(\frac{q-1}{2}\right) + 1 = \frac{q^2 + 1}{2}$.

    (a) Для $n < 4$ доказательство тривиально.
    Достаточно доказать п. (a) только для клеточных 2-циклов из базиса $B_{t \times t}^2$.
    
    Для клеточного 2-цикла $\t{K_3}^2$ п. (a) очевиден.
    
    Для $(t \times t)$-симметризованного тора $Q \times R + tQ \times tR$ п. (a) доказывается так.
    Поскольку $Q$ и $R$ --- циклы из $B$, то по лемме \ref{p:symgra}.b они имеют длину 4 или 6.
    Поскольку $n \geq 4$, то каждый из циклов $Q$ и $R$ является суммой трех циклов длины 4: $Q = Q_1 + Q_2 + Q_3$, $R = R_1 + R_2 + R_3$ (рис. \ref{f:cycles64}; если сам цикл $C$ имеет длину 4, то $C = C + C+ C$). 
    Тогда 
    $$Q \times R + tQ \times tR = (Q_1 + Q_2 + Q_3) \times (R_1 + R_2 + R_3) + t(Q_1 + Q_2 + Q_3) \times t(R_1 + R_2 + R_3) =$$
    $$ = (Q_1 + Q_2 + Q_3) \times (R_1 + R_2 + R_3) + (tQ_1 + tQ_2 + tQ_3) \times (tR_1 + tR_2 + tR_3) = \sum_{i=1}^3 \sum_{j  = 1}^3 (Q_i \times R_j + tQ_i \times tR_j).$$

    Доказательство п. (a) для клеточных 2-циклов  $\t{K_3}\times(R+tR)$ и $(Q+tQ)\times\t{K_3}$ аналогично.
\end{proof}

        
         

\begin{theorem}\label{t:2cyc-sym-bases}
    (a) Любой симметричный клеточный $2$-цикл в $K_{n,n}^{\underline2}$ является суммой нескольких из следующих слагаемых: 
     
    $\bullet$ симметризованные торы $Q\times R+R\times Q$ для вершинно-непересекающихся циклов $Q,R$ длины~$4$ в $K_{n,n}$;
    
    $\bullet$ комбинаторный взрезанный квадрат $K_{3,3}^{\underline2}$ подграфа $K_{3,3}$, графа $K_{n,n}$. 

    (b) Любой клеточный $2$-цикл в $K^{\underline2}$ является суммой нескольких произведений вершинно-не\-пе\-ре\-секающихся циклов и нескольких комбинаторных взрезанных квадратов подграфов, гомеоморфных графу $K_5$ или $K_{3,3}$.
\end{theorem} 

\begin{proof}[Набросок доказательства п.~(а)] 
    Возьмем взаимно однозначное соответствие из утверждения \ref{l:h2sym}.b.
    По лемме \ref{t:2cyc-sym-lemma}.a, так как $(t \times t)$-симметризованные торы и $\t{K_3}^2$  порождают все $(t \times t)$-симметричные клеточные $2$-циклы в $\t{K_n}^2$, то их прообразы порождают все симметричные клеточные $2$-циклы в $K_{n,n}^{\underline2}$. Осталось проверить, что
    
    $\bullet$ прообразом $(t \times t)$-симметризованного тора для циклов длины 4 является симметризованный тор $Q\times R+R\times Q$ для (других) вершинно-непересекающиеся циклов $Q,R$ длины 4;  
    
    $\bullet$ прообразом тора $\t{K_3}^2$ является $K_{3,3}^{\underline2}$ (ср. с замечанием \ref{remark:2cycl-bij}.ab).
\end{proof}

Теорема \ref{t:2cyc-sym-bases}.a есть частный случай результата из \cite[теорема 3.1]{SS23}.
Доказательство теоремы \ref{t:2cyc-sym-bases}.b можно найти в \cite{Sa91} и \cite{SS23}.

\begin{pr}\label{p:sawrong} *
    (a) Следующее утверждение~\cite[3.4.2]{Sa91} неверно, даже для связных графов:
    \emph{существует клеточный $2$-цикл $C$ в $K^{\underline2}$, такой что любой клеточный $2$-цикл в $K^{\underline2}$ является суммой нескольких из следующих клеточных $2$-циклов: $C$ и произведения вершинно-непересекающихся циклов.}
    
    (b) (нерешенная задача) Верно ли утверждение из п.~(a) для $3$-связных графов? 
    
    (с) (нерешенная задача) Верен ли аналог теоремы~\ref{t:2cyc-sym-bases}.a с заменой графа $K_{n,n}$ на граф $K_n$ и подграфа $K_{3,3}$ графа $K_{n,n}$ на подграф (гомеоморфный) $K_5$ графа $K_n$?
    
    (d) (загадка) Сколько имеется клеточных $2$-циклов в $K^{\underline2}$ для связного графа $K$ с $V$ вершинами и $E$ ребрами (используйте дополнительные данные о графе, если нужно)?
    
    {\it Подсказки.} (a) Возьмите дизъюнктное объединение двух копий графа $K_5$.
    Сделайте это объединение связным, соединив две копии ребром.
\end{pr}



\begin{thebibliography}{MMM+}


\newcommand{\aate}{\bibitem[AA38]{AA38} \emph{A. Adrian Albert}. Symmetric and alternate matrices in an arbitrary field, I. Trans. Amer. Math. Soc., (1938) 43(3):386--436.}

\newcommand{\abc}{\bibitem[ABC+]{ABC+} * \emph{M. Atiyah, A. Borel, G. J. Chaitin, D. Friedan, J. Glimm, J. J. Gray, M. W. Hirsch, S. MacLane, B. B. Mandelbrot, D. Ruelle, A. Schwarz, K. Uhlenbeck, R. Thom, E. Witten, C.  Zeeman.} Responses to ``Theoretical Mathematics: Toward a cultural synthesis of mathematics and theoretical physics'', by A. Jaffe and F. Quinn. Bull. Am. Math. Soc. 30 (1994) 178--207. arXiv:math/9404229.}

\newcommand{\abgmns}{\bibitem[ABM+]{ABM+} * \emph{E. Alkin, E. Bordacheva, A. Miroshnikov, O. Nikitenko, A. Skopenkov,} Invariants of almost embeddings of graphs in the plane: results and problems, arXiv:2408.06392.}

\newcommand{\abms}{\bibitem[ABM+]{ABM+} * \emph{Э. Алкин, Е. Бордачева, А. Мирошников, А. Скопенков,} Инварианты почти вложений графов в плоскость, arXiv:2410.09860.}

\newcommand{\adnt}{\bibitem[Ad93]{Ad93} * \emph{M. Adachi}. Embeddings and Immersions. Amer. Math.
Soc., 1993. (Transl. of Math. Monographs; V.~124).}

\newcommand{\adoe}{\bibitem[Ad18]{Ad18} {\it K. Adiprasito,} Combinatorial Lefschetz theorems beyond positivity, arXiv:1812.10454v4.}

\newcommand{\adnsv}{\bibitem[ADN+]{ADN+} * \emph{E. Alkin, S. Dzhenzher, O. Nikitenko, A. Skopenkov, A. Voropaev.}
Cycles in graphs and in hypergraphs: resuts and problems, arXiv:2308.05175.}

\newcommand{\agles}{\bibitem[AGL]{AGL86} Mathematical Economics,  ed. by A. Ambrosetti, F. Gori, R. Lucchetti,
Lect. Notes Math. 1330, Springer, 1986.}


\newcommand{\akzz}{\bibitem[Ak00]{Ak00} * \emph{П. М. Ахметьев.} Вложения компактов, стабильные
гомотопические группы сфер и теория особенностей, Успехи Мат. Наук.  2000. 55:3. C.~3-62.}

\newcommand{\akoe}{\bibitem[AK19]{AK19} \emph{S. Avvakumov, R. Karasev.} Envy-free division using mapping degree,
Mathematika, 67:1 (2020), 36--53. arXiv:1907.11183.}

\newcommand{\akto}{\bibitem[AK21]{AK21} \emph{G. Arone and V. Krushkal.}
Embedding obstructions in $\R^d$ from the Goodwillie-Weiss calculus and Whitney disks. arXiv:2101.10995. }

\newcommand{\akm}{\bibitem[AKM]{AKM} \emph{M. Abrahamsen, L. Kleist and T. Miltzow.}
Geometric Embeddability of Complexes is $\exists\mathbb R$-complete, arXiv:2108.02585.}

\newcommand{\aksoe}{\bibitem[AKS]{AKS} \emph{S. Avvakumov, R. Karasev and A. Skopenkov.} Stronger counterexamples to the topological Tverberg conjecture, Combinatorica, 43 (2023), 717--727. arXiv:1908.08731.}


\newcommand{\akuoe}{\bibitem[AKu19]{AKu19} \emph{S. Avvakumov, S. Kudrya.}
Vanishing of all equivariant obstructions and the mapping degree.
Discr. Comp. Geom., 66:3 (2021) 1202--1216. arXiv:1910.12628.}

\newcommand{\alto}{\bibitem[Al22]{Al22} \emph{E. Alkin,}
Hardness of almost embedding simplicial complexes in $\R^d$, II. arXiv:2206.13486}

\newcommand{\amtf}{\bibitem[AM25]{AM25} \emph{E. Alkin, A. Miroshnikov,} On winding numbers of almost embeddings of $K_4$ in the plane, 	arXiv:2501.15642.}

\newcommand{\amsw}{\bibitem[AMS+]{AMSW} \emph{S. Avvakumov, I. Mabillard, A. Skopenkov and U. Wagner.}
Eliminating Higher-Multiplicity Intersections, III. Codimension 2, Israel J. Math. 245 (2021) 501--534.  arxiv:1511.03501.}


\newcommand{\anzt}{\bibitem[An03]{An03} * \emph{Д. В. Аносов.} Отображения окружности, векторные поля и их применения. М: МЦНМО, 2003.}

\newcommand{\arnf}{\bibitem[Ar95]{Ar95} * \emph{V. I. Arnold,}  Topological invariants of plane curves and caustics, University Lecture Series, Vol. 5, Amer. Math. Soc., Providence, RI, 1995.}

\newcommand{\arszo}{\bibitem[ARS01]{ARS01} \emph{P. Akhmetiev, D. Repov\v s and A. Skopenkov},
Embedding products of low-dimensional manifolds in $\R^m$, Topol. Appl. 113 (2001), 7--12.}

\newcommand{\arszt}{\bibitem[ARS02]{ARS02} \emph{P. Akhmetiev, D. Repovs and A. Skopenkov.} Obstructions to approximating maps of $n$-manifolds into $R^{2n}$ by embeddings, Topol. Appl., 123 (2002), 3--14.}

\newcommand{\asoed}{\bibitem[As]{As} \emph{A. Asanau,} \lowercase{A SIMPLE PROOF THAT CONNECTED SUM OF ORDERED
ORIENTED LINKS IS NOT WELL-DEFINED,} Math. Notes, to appear.}

\newcommand{\asoe}{\bibitem[As]{As} \emph{A. Asanau,} On the \lowercase{TRIPLE SELF-INTERSECTION NUMBER FOR GRAPHS IN THE PLANE,} unpublished, 2018.}

\newcommand{\avos}{\bibitem[Av14]{Av14} \emph{S. Avvakumov,} The classification of certain linked 3-manifolds in 6-space, Moscow Math. J., 16:1 (2016), 1--25. arXiv:1408.3918.}

\newcommand{\avose}{\bibitem[Av17]{Av17} \emph{S. Avvakumov,} The classification of linked 3-manifolds in 6-space, Algebraic \& Geometric Topology, 22:6 (2022) 2587--2630. arXiv:1704.06501.}



\newcommand{\bant}{\bibitem[Ba93]{Ba93} * \emph{T. Bartsch.} Topological methods for variational problems with
symmetries, Lecture Notes in Mathematics, 1560, Springer-Verlag, Berlin, 1993.}

\newcommand{\batt}{\bibitem[Ba23]{Ba23} * \emph{I. Barany.} Tverberg's theorem, a new proof. arXiv:2308.10105.}

\newcommand{\bbsn}{\bibitem[BB79]{BB} \emph{E.~G. Bajm{{\'o}}czy and I.~B{{\'a}}r{{\'a}}ny,}
\newblock On a common generalization of {B}orsuk's and {R}adon's theorem,
\newblock Acta Math.\ Acad.\ Sci.\ Hungar.\ 34:3 (1979), 347-350.}

\newcommand{\bbzos}{\bibitem[BBZ]{BBZ} * \emph{I.~B{{\'a}}r{{\'a}}ny, P.~V.~M. Blagojevi{{\'c}} and G.~M. Ziegler.} Tverberg's Theorem at 50: Extensions and Counterexamples, Notices of the Amer. Math. Soc., 63:7 (2016), 732--739.}


\newcommand{\bcm}{\bibitem[BCM]{BCM} * 13th Hilbert Problem on superpositions of functions, presented by A. Belov, A. Chilikov, I. Mitrofanov, S. Shaposhnikov and A. Skopenkov,
\url{http://www.turgor.ru/lktg/2016/5/index.htm}.}

\newcommand{\beet}{\bibitem[BE82]{BE82} * \emph{V.G. Boltyansky and V.A. Efremovich.} Intuitive Combinatorial Topology. Springer.}

\newcommand{\beetr}{\bibitem[BE82]{BE82} * \emph{В. Г. Болтянский и В. А. Ефремович.} Наглядная топология. М.:  Наука, 1982.}


\newcommand{\bfzn}{\bibitem[BF09]{BF09} \emph{K. Barnett, M. Farber}. Topology of Configuration Space of Two Particles on a Graph, I.  Algebr. Geom. Topol. 9 (2009) 593--624.	arXiv:0903.2180.}

\newcommand{\bfzof}{\bibitem[BFZ14]{BFZ14} \emph{P. V. M. Blagojevi{\'c}, F. Frick, and G. M. Ziegler,}
Tverberg plus constraints, Bull. Lond. Math. Soc. 46:5 (2014), 953-967, arXiv:1401.0690.}


\newcommand{\bfzos}{\bibitem[BFZ]{BFZ} \emph{P. V. M. Blagojevi{\'c}, F. Frick and G. M. Ziegler,}
Barycenters of Polytope Skeleta and Counterexamples to the Topological Tverberg Conjecture, via Constraints,
J. Eur. Math. Soc., 21:7 (2019) 2107-2116. arXiv:1510.07984.}


\newcommand{\bgso}{\bibitem[BG71]{BG71} J.C. Becker and H. H. Glover, {\it Note on the Embedding of Manifolds in Euclidean Space,} Proc. of the Amer. Math. Soc., 27:2 (1971) 405-410.}


\newcommand{\bgos}{\bibitem[BG16]{BG16} \emph{A. Bj\"orner and A. Goodarzi}, On Codimension one Embedding of Simplicial Complexes, in book: A Journey Through Discrete Mathematics, arXiv:1605.01240.}

\newcommand{\biet}{\bibitem[Bi83]{Bi83} * \emph{R. H. Bing.} The Geometric Topology of 3-Manifolds. Providence, R.~I. 1983. (Amer. Math. Soc. Colloq. Publ., 40).}

\newcommand{\bitz}{\bibitem[Bi20]{Bi20} * \emph{A. Bikeev.} Realizability of discs with ribbons on the M\"obius strip. Mat. Prosveschenie, 28 (2021), 150-158;
erratum to appear. arXiv:2010.15833.}

\newcommand{\bitzr}{\bibitem[Bi20]{Bi20} * \emph{А. Бикеев.} Реализуемость дисков с ленточками на ленте Мебиуса.
Мат. просвещение. Сер. 3. 28 (2021), 150--158.}

\newcommand{\bito}{\bibitem[Bi21]{Bi21} {\it A. I. Bikeev,}
Criteria for integer and modulo 2 embeddability of graphs to surfaces, arXiv:2012.12070v2.}


\newcommand{\bagos}{\bibitem[BG17]{BG17} \emph{S. Basu and S. Ghosh.} Equivariant maps related to the topological Tverberg conjecture, Homology, Homotopy and Applications 19:1 (2017) 155--170.}

\newcommand{\bkkmzof}{\bibitem[BKK]{BKK} \emph{M. Bestvina, M. Kapovich and B. Kleiner,}
Van Kampen's embedding obstruction for discrete groups, Invent. Math. 150 (2002) 219--235. arXiv:math/0010141.}

\newcommand{\bl}{\bibitem[BL]{BL} \url{https://en.wikipedia.org/wiki/Brunnian_link}}

\newcommand{\blf}{\bibitem[BL4]{BL4} Students form a 4-component Brunnian link,  \url{http://www.mccme.ru/circles/oim/foto2014/brunn4.png} (5Mb)}

\newcommand{\bmzf}{\bibitem[BM04]{BM04} \emph{Boyer, J. M. and Myrvold, W. J.} On the cutting edge: simplified $O(n)$ planarity by edge addition,  Journal of Graph Algorithms and Applications, 8:3 (2004) 241--273.}

\newcommand{\bm}{\bibitem[BM15]{BM15} \emph{I. Bogdanov and A. Matushkin.} Algebraic proofs of linear versions of the Conway--Gordon--Sachs theorem and the van Kampen--Flores theorem, arXiv:1508.03185.}


\newcommand{\bmzzn}{\bibitem[BMZ09]{BMZ09} \emph{P. V. M. Blagojevi{\'c}, B. Matschke, G. M. Ziegler,}
Optimal bounds for a colorful Tverberg-Vre\'cica type problem, Advances in Math., 226 (2011), 5198-5215, arXiv:0911.2692.}

\newcommand{\bmzof}{\bibitem[BMZ15]{BMZ15} \emph{P. V. M. Blagojevi{\'c}, B. Matschke, G. M. Ziegler,}
Optimal bounds for the colored Tverberg problem, J. Eur. Math. Soc.,  17:4 (2015) 739--754,
arXiv:0910.4987.}

\newcommand{\bpns}{\bibitem[BP97]{BP97} * \emph{R. Benedetti and C. Petronio.} Branched standard spines of 3-manifolds, Lecture Notes in Math. 1653, Springer-Verlag, Berlin-Heidelberg-New York, 1997.}

\newcommand{\brst}{\bibitem[Br72]{Br72} \emph{J. L. Bryant.} Approximating embeddings of polyhedra in codimension 3, Trans. Amer. Math. Soc., 170 (1972) 85--95.}

\newcommand{\brts}{\bibitem[Br68]{Br68} \emph{P. Bruegel,} 1568,
\url{https://en.wikipedia.org/wiki/The_Magpie_on_the_Gallows}.}


\newcommand{\bren}{\bibitem[Br82]{brown1982} * \emph{K.~S. Brown.} \newblock Cohomology of Groups. \newblock Springer-Verlag New York, 1982.}


\newcommand{\bssos}{\bibitem[BS17]{BS17} * \emph{I.~B\'{a}r\'{a}ny and P. Sober\'{o}n,} Tverberg's theorem is 50 years old: a survey, Bull. Amer. Math. Soc. (N.S.) 55:4 (2018), 459--492. arXiv:1712.06119.}

\newcommand{\bsto}{\bibitem[BS21]{BS21} * \emph{A. Buchaev and A. Skopenkov,} Simple proofs of estimations of Ramsey numbers and of discrepancy, Mat. Prosveschenie, to appear, arXiv:2107.13831.}

\newcommand{\brsnn}{\bibitem[BRS99]{BRS99} \emph{D. Repov\v s, N. Brodsky and A. B. Skopenkov.}
A classification of 3-thickenings of 2-polyhedra, Topol. Appl. 1999. 94. P.~307-314.}

\newcommand{\bsseo}{\bibitem[BSS]{BSS} \emph{I.~B\'{a}r\'{a}ny, S.~B. Shlosman, and A.~Sz{\H{u}}cs,}
\newblock On a topological generalization of a theorem of {T}verberg,
\newblock J.\ London Math.\ Soc.\ (II. Ser.) 23 (1981), 158--164.}

\newcommand{\btzs}{\bibitem[BT07]{BT07} \emph{A. Bj\"orner, M. Tancer}, Combinatorial Alexander Duality --- a Short and Elementary Proof, Discr. and Comp. Geom., 42 (2009) 586. arXiv:0710.1172.}

\newcommand{\buse}{\bibitem[Bu68]{Bu68} \emph{A. R. Butz,} Space filling curves and mathematical programming, Information and Control, 12:4 (1968) 314--330.}


\newcommand{\bz}{\bibitem[BZ16]{BZ16} * \emph{P. V. M. Blagojevi\'c and G. M. Ziegler,} Beyond the Borsuk-Ulam theorem: The topological Tverberg story, in: A Journey Through Discrete Mathematics, Eds. M. Loebl,
J. Ne\v set\v ril, R. Thomas, Springer, 2017, 273--341. arXiv:1605.07321v3.}



\newcommand{\cano}{\bibitem[Ca91]{Ca91} * \emph{D. de Caen}, The ranks of tournament matrices, Amer. Math. Monthly, 98:9 (1991) 829--831.}

\newcommand{\ca}{\bibitem[Ca]{Ca} \emph{J. Carmesin.} Embedding simply connected 2-complexes in 3-space, I-V, arXiv:1709.04642, arXiv:1709.04643, arXiv:1709.04645, arXiv:1709.04652, arXiv:1709.04659.}

\newcommand{\cfsz}{\bibitem[CF60]{CF60} \emph{P. E. Conner and E. E. Floyd}, Fixed points free involutions and equivariant maps, Bull. Amer. Math. Soc., 66 (1960) 416--441.}

\newcommand{\cfs}{\bibitem[CFS]{CFS} \emph{D. Crowley, S.C. Ferry, M. Skopenkov,} The rational classification of links of codimension $>2$, Forum Math. 26 (2014), 239--269. arXiv:1106.1455.}

\newcommand{\cget}{\bibitem[CG83]{CG83} \emph{J. H. Conway and C. M. A. Gordon},
Knots and links in spatial graphs, J. Graph Theory  7 (1983), 445--453.}

\newcommand{\cten}{\bibitem[Ch]{Ch} \emph{Chuang Tzu,} translated by H. A. Giles, Bernard Quaritch, London, 1889.}

\newcommand{\ctruku}{\bibitem[Ch]{Ch} \emph{Chuang Tzu,} translated to Russian by S. Kuchera, in: Ancient Chinese Philosophy, v. I, Mysl, Moscow, 1972.}


\newcommand{\chnn}{\bibitem[Ch99]{Ch99} * \emph{А. В. Чернавский,} Теорема Жордана.  Мат. Просвещение, 3 (1999), 142--157.}

\newcommand{\hcon}{\bibitem[HC19]{HC19} * \emph{C. Herbert Clemens.} Two-Dimensional Geometries. A Problem-Solving Approach, Amer. Math. Soc., 2019.}

\newcommand{\ckmoo}{\bibitem[CKMS]{CKMS} \emph{M. \v Cadek, M. Kr\v c\'al. J. Matou\v sek, F. Sergeraert,
L. Vok\v r\'inek, U. Wagner.} Computing all maps into a sphere, J. of the ACM, 61:3 (2014). arXiv:1105.6257.}


\newcommand{\ckmvwot}{\bibitem[CKM12+]{CKM12+} \emph{M. \v Cadek, M. Kr\v c\'al. J. Matou\v sek, L. Vok\v r\'inek, U. Wagner.} Polynomial-time computation of homotopy groups and Postnikov systems in fixed dimension, SIAM J. Comput., 43:5 (2014), 1728--1780. arXiv:1211.3093.}

\newcommand{\ckmvw}{\bibitem[CKM+]{CKM+} \emph{M. \v Cadek, M. Kr\v c\'al. J. Matou\v sek, L. Vok\v r\'inek, U. Wagner.} Extendability of continuous maps is undecidable, Discr. and Comp. Geom. 51 (2014) 24--66.
arXiv:1302.2370.}

\newcommand{\ckppt}{\bibitem[CKP+]{CKP+} \emph{E. Colin de Verdi\'ere, V. Kalu\v za, P. Pat\'ak, Z. Pat\'akov\'a and M. Tancer.} A direct proof of the strong Hanani-Tutte theorem on the projective plane. Journal of Graph Algorithms and Applications, 21:5 (2017) 939--981.}

\newcommand{\cksof}{\bibitem[CKS+]{CKS+} * New ways of weaving baskets, presented by G. Chelnokov, Yu. Kudryashov, A.Skopenkov and A. Sossinsky, \url{http://www.turgor.ru/lktg/2004/lines.en/index.htm}.}

\newcommand{\ckv}{\bibitem[CKV]{CKV} \emph{M.~{\v{C}}adek, M.~Kr\v{c}\'{a}l, and L.~Vok\v{r}\'{\i}nek.}
Algorithmic solvability of the lifting-extension problem, Discr. Comp. Geom. 57 (2017), 915--965. arXiv:1307.6444.}


\newcommand{\clr}{\bibitem[CLR]{CLR} * \emph{Т. Кормен, Ч. Лейзерсон, Р. Ривест.} Алгоритмы:
построение и анализ, МЦНМО, Москва, 1999.}

\newcommand{\clreng}{\bibitem[CLR]{CLR} * \emph{T. H. Cormen, C. E.Leiserson, R. L.Rivest, C. Stein.} Introduction to Algorithms, MIT Press, 2009.}

\newcommand{\crzfru}{\bibitem[CR]{CR} * \emph{Р. Курант, Дж. Роббинс,} Что такое математика. М.: МЦНМО, 2004.}

\newcommand{\crzfen}{\bibitem[CR]{CR} * \emph{R. Courant and H. Robbins,} What is Mathematics, Oxford Univ. Press.}

\newcommand{\crsne}{\bibitem[CRS98]{CRS98} * \emph{A. Cavicchioli, D. Repov\v s and A. B. Skopenkov.}
Open problems on graphs, arising from geometric topology, Topol. Appl. 1998. 84. P.~207-226.}

\newcommand{\crsot}{\bibitem[CRS]{CRS} \emph{M. Cencelj, D. Repov\v s and M. Skopenkov,}
Classification of knotted tori in the 2-metastable dimension, Mat. Sbornik, 203:11 (2012), 1654--1681.
arxiv:math/0811.2745.}

\newcommand{\csoo}{\bibitem[CS08]{CS08} \emph{D. Crowley and A. Skopenkov.} A classification of smooth embeddings of 4-manifolds in 7-space, II, Intern. J. Math., 22:6 (2011) 731-757, arxiv:math/0808.1795.}

\newcommand{\csos}{\bibitem[CS16]{CS16} \emph{D. Crowley and A. Skopenkov,} Embeddings of non-simply-connected 4-manifolds in 7-space. I. Classification modulo knots, Moscow Math. J., 21 (2021), 43--98. arXiv:1611.04738.}


\newcommand{\csoso}{\bibitem[CS16o]{CS16o} \emph{D. Crowley and A. Skopenkov,} Embeddings of non-simply-connected 4-manifolds in 7-space. II. On the smooth classification, Proc. A of the Royal Soc. of Edinburgh 152:1 (2022), 163--181. arXiv:1612.04776.}


\newcommand{\crsk}{\bibitem[CS]{CS} \emph{D. Crowley and A. Skopenkov,} Embeddings of non-simply-connected 4-manifolds in 7-space. III. Piecewise-linear classification. draft.}

\newcommand{\cutz}{\bibitem[Cu20]{Cu20} \emph{C. Culter,} Cantor sets are not tangent homogeneous,
Topol. Appl. 271 (2020) 1--9.}


\newcommand{\dies}{\bibitem[Di87]{Di} * \emph{T. tom Dieck,} Transformation groups, Studies in Mathematics, vol. 8, Walter de Gruyter, Berlin, 1987.}

\newcommand{\dize}{\bibitem[Di08]{Di08} * \emph{T. tom Dieck,} Algebraic topology, EMS Textbooks in Mathematics, 
EMS, Z\"urich, 2008.}

\newcommand{\dent}{\bibitem[De93]{De93}  \emph{T.K. Dey.} On counting triangulations in $d$-dimensions. Comput. Geom.  3:6 (1993) 315--325.}

\newcommand{\denf}{\bibitem[DE94]{DE94}  \emph{T.K. Dey and H. Edelsbrunner.} Counting triangle crossings and halving planes, Discrete Comput. Geom. 12 (1994), 281--289.}

\newcommand{\dgn}{\bibitem[DGN+]{DGN+} * S. Dzhenzher, T. Garaev, O. Nikitenko, A. Petukhov, A. Skopenkov, A. Voropaev, Low rank matrix completion and realization of graphs: results and problems, arXiv:2501.13935.}

\newcommand{\dgnr}{\bibitem[DGN+]{DGN+} * Минимизация ранга восполнением матриц, представляли А. Воропаев, Т. Гараев, С. Дженжер, О. Никитенко, А. Петухов и А. Скопенков, \url{https://www.mccme.ru/circles/oim/netflix_rus.pdf}.}
 
\newcommand{\dstt}{\bibitem[DS22]{DS22}  \emph{S. Dzhenzher and A. Skopenkov,} A quadratic estimation for the K\"uhnel conjecture on embeddings, arXiv:2208.04188.}

\newcommand{\botf}{\bibitem[Dz25]{Dz25} \emph{E. Dzhenzher,} Symmetric 1-cycles in the deleted product of a graph, Topol. Appl. (2025) 109277.}


\newcommand{\embo}{\bibitem[Eb]{Eb} * \url{http://www.map.mpim-bonn.mpg.de/Embeddings_of_manifolds_with_boundary:_classification}}

\newcommand{\embe}{\bibitem[Em]{Em} * \url{http://www.map.mpim-bonn.mpg.de/Embedding_(simple_definition)}}

\newcommand{\ers}{\bibitem[ERS]{ERS} * Invariants of graph drawings in the plane, presented by A. Enne, A. Ryabichev, A. Skopenkov and T. Zaitsev, \url{http://www.turgor.ru/lktg/2017/6/index.htm}}



\newcommand{\feto}{\bibitem[Fe21]{Fe21} \emph{M. Fedorov.} A description of values of Seifert form for punctured $n$-manifolds in $(2n-1)$-space, arXiv:2107.02541.}

\newcommand{\ffen}{\bibitem[FF89]{FF89} * \emph{А. Т. Фоменко и Д. Б. Фукс.} Курс гомотопической топологии. М.: Наука, 1989.}

\newcommand{\ffene}{\bibitem[FF89]{FF89} * \emph{A.T. Fomenko and D.B. Fuchs.} Homotopical Topology, Springer, 2016.}


\newcommand{\fhzo}{\bibitem[FH10]{FH10}  \emph{M. Farber, E. Hanbury}. Topology of Configuration Space of Two Particles on a Graph, II. Algebr. Geom. Topol. 10 (2010) 2203--2227. arXiv:1005.2300.}


\newcommand{\fkosc}{\bibitem[FK17]{FK17} \emph{R. Fulek, J. Kyn{\v{c}}l,} Counterexample to an Extension of the Hanani-Tutte Theorem on the Surface of Genus 4, Combinatorica, 39 (2019) 1267--1279, arXiv:1709.00508.}

\newcommand{\fkos}{\bibitem[FK17]{FK17} \emph{R. Fulek, J. Kyn{\v{c}}l,} Hanani-Tutte for approximating maps of graphs, arXiv:1705.05243.}

\newcommand{\fkon}{\bibitem[FK19]{FK19} \emph{R. Fulek, J. Kyn{\v{c}}l,}
$\Z_2$-genus of graphs and minimum rank of partial symmetric matrices,
35th Intern. Symp. on Comp. Geom. (SoCG 2019), Article No. 39; pp. 39:1--39:16, \linebreak
\url{https://drops.dagstuhl.de/opus/volltexte/2019/10443/pdf/LIPIcs-SoCG-2019-39.pdf}.
We refer to numbering in arXiv version: arXiv:1903.08637.}

\newcommand{\fktnf}{\bibitem[FKT]{FKT} \emph{M. H. Freedman, V. S. Krushkal and P. Teichner.} Van Kampen's
embedding obstruction is incomplete for 2-complexes in~$\R^4$, Math. Res. Letters. 1994. 1. P.~167-176.}

\newcommand{\fltf}{\bibitem[Fl34]{Fl34} \emph{A. Flores}, \"Uber $n$-dimensionale Komplexe die im $E^{2n+1}$ absolut selbstverschlungen sind, Ergeb. Math. Koll. 6 (1934) 4--7.}

\newcommand{\fo}{\bibitem[Fo]{Fo} * \emph{L. Fortnow.} Time for Computer Science to Grow Up,  \url{https://people.cs.uchicago.edu/~fortnow/papers/growup.pdf}.}

\newcommand{\fozf}{\bibitem[Fo04]{Fo04} * \emph{R. Fokkink.} A forgotten mathematician, Eur. Math. Soc. Newsletter 52 (2004) 9--14.}


\newcommand{\fpstz}{\bibitem[FPS]{FPS} \emph{R. Fulek, M.J. Pelsmajer and M. Schaefer.}
Strong Hanani-Tutte for the Torus, arXiv:2009.01683.}

\newcommand{\frse}{\bibitem[Fr78]{Fr78} \emph{M. Freedman,} Quadruple points of 3-manifolds in $S^4$, Comment. Math. Helv. 53 (1978), 385-394.}

\newcommand{\fres}{\bibitem[FR86]{FR86} \emph{R. Fenn, D. Rolfsen.}
Spheres may link homotopically in 4-space, J. London Math. Soc. 34 (1986) 177-184.}

\newcommand{\frofea}{\bibitem[Fr15']{Fr15'} \emph{F. Frick}, Counterexamples to the topological Tverberg conjecture, arXiv:1502.00947v1.}


\newcommand{\frof}{\bibitem[Fr15]{Fr15} \emph{F. Frick}, Counterexamples to the topological Tverberg conjecture,
Oberwolfach reports, 12:1 (2015), 318--321. arXiv:1502.00947.}

\newcommand{\fros}{\bibitem[Fr17]{Fr17} \emph{F. Frick}, O\lowercase{N AFFINE TVERBERG-TYPE RESULTS WITHOUT CONTINUOUS GENERALIZATION}, arXiv:1702.05466}


\newcommand{\fstz}{\bibitem[FS20]{FS20} \emph{F. Frick and P. Sober\'on}, The topological Tverberg problem beyond prime powers, arXiv:2005.05251.}

\newcommand{\ftss}{\bibitem[FT77]{FT77} \emph{R. Fenn, P. Taylor,} Introducing doodles, pp. 37-43
in: Topology of Low-Dimensional Manifolds, Proceedings of the Second Sussex Conference, 1977,
Ed. R. Fenn, V. 722 of Lecture Notes in Math.}

\newcommand{\fvto}{\bibitem[FV21]{FV21} \emph{M. Filakovsk\'y, L. Vok\v r\'inek.} Computing homotopy classes for diagrams, Discr. Comp. Geom. 70 (2023), 866--920. arXiv:2104.10152.}

\newcommand{\fwz}{\bibitem[FWZ]{FWZ} \emph{M. Filakovsk\'y, U. Wagner, S. Zhechev.} Embeddability of simplicial complexes is undecidable. Oberwolfach reports, to appear.}

\newcommand{\fwztz}{\bibitem[FWZ]{FWZ} \emph{M. Filakovsk\'y, U. Wagner, S. Zhechev.} Embeddability of simplicial complexes is undecidable. Proceedings of the 2020 ACM-SIAM Symposium on Discrete Algorithms.}



\newcommand{\ga}{\bibitem[GA]{GA} * \url{https://en.wikipedia.org/wiki/Galactic_algorithm}}

\newcommand{\gatt}{\bibitem[Ga23]{Ga23} \emph{T. Garaev}, On drawing $K_5$ minus an edge in the plane, arXiv:2303.14503.}

\newcommand{\gdikrse}{\bibitem[GDI]{GDI} * {\it A. Chernov, A. Daynyak, A. Glibichuk, M. Ilyinskiy, A. Kupavskiy, A. Raigorodskiy and A. Skopenkov,} Elements of Discrete Mathematics As a Sequence of Problems (in Russian),
MCCME, Moscow, 2016. Update of a part: \url{http://www.mccme.ru/circles/oim/discrbook.pdf}}

\newcommand{\gdikrs}{\bibitem[GDI]{GDI} * {\it А.А. Глибичук, А.Б. Дайняк, Д.Г. Ильинский, А.Б. Купавский, А.М. Райгородский, А.Б. Скопенков, А.А. Чернов,} Элементы дискретной математики в задачах, М, МЦНМО, 2016. Обновляемая версия части книги: 
\url{http://www.mccme.ru/circles/oim/discrbook.pdf}}

\newcommand{\giso}{\bibitem[Gi71]{Gi71} * {\it S. Gitler,} Immersion and Embedding of Manifolds, Proc. Symp. Pure Math. 22, 87-96 (1971).}

\newcommand{\gkp}{\bibitem[GKP]{GKP} * {\it R. Graham, D. Knuth, and O. Patashnik,} Concrete Mathematics: A Foundation for Computer Science, Addison–Wesley, first published in 1989, \url{https://www.csie.ntu.edu.tw/~r97002/temp/Concrete\%20Mathematics\%202e.pdf}.}

\newcommand{\gmpptw}{\bibitem[GMP+]{GMP+} \emph{X. Goaoc, I. Mabillard, P. Pat\'ak, Z. Pat\'akov\'a, M. Tancer, U. Wagner}, On Generalized Heawood Inequalities for Manifolds: a van Kampen--Flores-type Nonembeddability Result,
Israel J. Math., 222(2) (2017) 841-866. arXiv:1610.09063.}


\newcommand{\gppot}{\bibitem[GPP+]{GPP+} \emph{X. Goaoc, P. Pat\'ak, Z. Pat\'akov\'a, M. Tancer, and U. Wagner.} Bounding Helly numbers via Betti numbers. In 31st International Symposium on Computational Geometry, volume 34
of LIPIcs. Leibniz Int. Proc. Inform., pp. 507-521. Schloss Dagstuhl. Leibniz-Zent. Inform., Wadern, 2015. Full version: arXiv:1310.4613.}

\newcommand{\group}{\bibitem[Gr]{Gr} * \url{https://en.wikipedia.org/wiki/Groupthink}}

\newcommand{\grsz}{\bibitem[Gr69]{Gr69} \emph{B. Gr\"unbaum.} Imbeddings of simplicial complexes. Comment. Math. Helv., 44:1, 502--513, 1969.}


\newcommand{\gres}{\bibitem[Gr86]{Gr86} * \emph{M. Gromov}, Partial Differential Relations,
Ergebnisse der Mathematik und ihrer Grenzgebiete (3), Springer Verlag, Berlin-New York, 1986.}

\newcommand{\groz}{\bibitem[Gr10]{Gr10} \emph{M. Gromov,}
\newblock Singularities, expanders and topology of maps. Part 2: From combinatorics to topology via algebraic isoperimetry, \newblock Geometric and Functional Analysis 20 (2010), no.~2, 416--526.}

\newcommand{\grsn}{\bibitem[GR79]{GR79} \emph{J. L. Gross	and R. H. Rosen}, A linear time planarity algorithm for 2-complexes, Journal of the ACM, 26:4 (1979), 611--617.}

\newcommand{\gs}{\bibitem[GS]{GS} \emph{М. Гортинский и О. Скрябин.} Критерий вложимости графов в плоскость вдоль прямой, препринт.}

\newcommand{\gssn}{\bibitem[GS79]{GS} \emph{P.~M. Gruber and R.~Schneider,} Problems in geometric convexity. In {\em Contributions to geometry (Proc. Geom. Sympos., Siegen, 1978)}, 255--278. Birkh{\"a}user, Basel-Boston, Mass., 1979.}

\newcommand{\gsnn}{\bibitem[GS99]{GS99} \emph{R. Gompf and A. Stipsicz,}
4-manifolds and Kirby calculus, GSM20, AMS, Providence, RI, 1999.}


\newcommand{\gszs}{\bibitem[GS06]{GS06} \emph{D. Goncalves and A. Skopenkov,} Embeddings of homology equivalent manifolds with boundary, Topol. Appl., 153:12 (2006) 2026-2034. arxiv:1207.1326.}

\newcommand{\gssoe}{\bibitem[GSS+]{GSS+} * Projections of skew lines, presented by A. Gaifullin, A. Shapovalov, A. Skopenkov and M. Skopenkov, \url{http://www.turgor.ru/lktg/2001/index.php}.}

\newcommand{\gtes}{\bibitem[GT87]{GT87} * \emph{J. L. Gross and T. W. Tucker.}
Topological graph theory. New York: Wiley-Interscience, 1987.}

\newcommand{\guzn}{\bibitem[Gu09]{Gu09} \emph{A. Gundert.} On the complexity of embeddable simplicial complexes. Diplomarbeit, Freie Universit\"at Berlin, 2009. 	arXiv:1812.08447.}


\newcommand{\ha}{\bibitem[Ha]{Ha} * \emph{F. Harary.} Graph theory.
Рус. пер.: Ф. Харари. Теория графов. М., Мир, 1973.}

\newcommand{\hats}{\bibitem[Ha37]{Ha37} \emph{W. Hantzsche,} Einlagerung von Mannigfaltigkeiten in euklidische R\" aume, Math. Zeitschrift, 43:1 (1937) 38--58.}

\newcommand{\hastk}{\bibitem[Ha62k]{Ha62k} {\em A.~Haefliger,}  Knotted $(4k-1)$-spheres in $6k$-space, Ann. of Math. 75 (1962) 452--466.}

\newcommand{\hastl}{\bibitem[Ha62l]{Ha62l} \emph{A. Haefliger,} Differentiable links, Topology, 1 (1962) 241--244.}

\newcommand{\hast}{\bibitem[Ha63]{Ha63} \emph{A.~Haefliger,} Plongements differentiables dans le domain stable, Comment. Math. Helv. 36 (1962-63) 155--176.}

\newcommand{\hassa}{\bibitem[Ha66A]{Ha66A} \textit{A. Haefliger}. Differential embeddings of~$S^n$ in $S^{n+q}$ for $q>2$. Ann. Math. (2), 83 (1966), 402--~436.}

\newcommand{\hass}{\bibitem[Ha66C]{Ha66C} \emph{A.~Haefliger,}  Enlacements de spheres en codimension superiure \`a 2, Comment. Math. Helv. 41 (1966-67) 51--72.}

\newcommand{\hase}{\bibitem[Ha68]{Ha68} \emph{A. Haefliger,} Knotted Spheres and Related Geometric Topic,
in Proc. Int. Congr. Math., Moscow, 1966 (Mir, Moscow, 1968), 437--445.}

\newcommand{\hasn}{\bibitem[Ha69]{Ha69} \emph{L.~S.~Harris,} Intersections and embeddings of polyhedra, Topology 8 (1969) 1--26.}

\newcommand{\hasf}{\bibitem[Ha74]{Ha74} * \emph{P. Halmos,} How to talk mathematics. Notices of the Amer. Math. Soc., 21 (1974) 155--158.}

\newcommand{\haef}{\bibitem[Ha84]{Ha84} \emph{N. Habegger,} Obstruction to embedding disks II: a proof of a conjecture by Hudson, Topol. Appl. 17 (1984).}

\newcommand{\haes}{\bibitem[Ha86]{Ha86} \emph{N. Habegger,} Knots and links in codimension greater than 2, Topology, 25:3 (1986) 253--260.}

\newcommand{\hogr}{\bibitem[HG]{HG} * \url{http://www.map.mpim-bonn.mpg.de/Homology_groups_(simplicial;_simple_definition)}}

\newcommand{\hifn}{\bibitem[Hi59]{Hi59} \emph{M. W. Hirsch.} Immersions of manifolds, Trans. Amer. Math. Soc. 93 (1959) 242--276.}

\newcommand{\hjsf}{\bibitem[HJ64]{HJ64} \emph{R. Halin and H. A. Jung.}
Karakterisierung der Komplexe der Ebene und der 2-Sph\"are, Arch. Math. 1964. 15. P.~466-469.}

\newcommand{\hkne}{\bibitem[HK98]{HK98} \emph{N. Habegger and U. Kaiser,} Link homotopy in 2--metastable range, Topology 37:1 (1998) 75--94.}

\newcommand{\hmsnt}{\bibitem[HMS]{HMS93} * \emph{C. Hog-Angeloni, W. Metzler and A. J. Sieradski.}
Two-dimensional homotopy and combinatorial group theory. Cambridge: Cambridge Univ. Press, 1993. (London Math. Soc. Lecture Notes, 197).}

\newcommand{\ho}{\bibitem[Ho]{Ho} * The Hopf fibration, \url{https://www.youtube.com/watch?v=AKotMPGFJYk}}

\newcommand{\hozs}{\bibitem[Ho06]{Ho06} \emph{H. van der Holst,} Graphs and obstructions in four dimensions, J. Combin. Theory Ser. B 96:3 (2006), 388--404.}


\newcommand{\hpzn}{\bibitem[HP09]{HP09} \emph{H. van der Holst and R. Pendavingh,} On a graph property generalizing planarity and flatness, Combinatorica, 29 (2009) 337--361.}

\newcommand{\hssf}{\bibitem[HS64]{HS64} \emph{A. Haefliger and B. Steer,} Symmetry of linking coefficients, Comment. Math. Helv. 39 (1964) 259-270.}

\newcommand{\htsf}{\bibitem[HT74]{HT74} \emph{J. Hopcroft and R. E. Tarjan,} Efficient planarity testing, J. of the Association for Computing Machinery, 21:4 (1974) 549--568.}

\newcommand{\hufn}{\bibitem[Hu59]{hu59} * \emph{S. T. Hu,} Homotopy Theory, Academic Press, New York, 1959.}

\newcommand{\husn}{\bibitem[Hu69]{Hu69} * \emph{J. F. P. Hudson.} Piecewise linear topology, W. A. Benjamin, Inc., New York-Amsterdam, 1969.}


\newcommand{\io}{\bibitem[Io]{Io} * \url{https://en.wikipedia.org/wiki/Category:Impossible_objects}}

\newcommand{\info}{\bibitem[IF]{IF} * \url{http://www.map.mpim-bonn.mpg.de/Intersection_form}}

\newcommand{\irsf}{\bibitem[Ir65]{Ir65} \emph{M.~C.~Irwin,} Embeddings of polyhedral manifolds, Ann. of Math. (2)
82 (1965) 1--14.}

\newcommand{\isot}{\bibitem[Is]{Is} * \url{http://www.map.mpim-bonn.mpg.de/Isotopy}}


\newcommand{\jqnt}{\bibitem[JQ93]{JQ93} * \emph{A. Jaffe, F. Quinn,} ``Theoretical mathematics'': Toward a cultural synthesis of mathematics and theoretical physics. Bull.Am.Math.Soc. 29 (1993) 1-13. arXiv:math/9307227.}

\newcommand{\jozt}{\bibitem[Jo02]{Jo02} \emph{C. M. Johnson.} An obstruction to embedding a simplicial $n$-complex into a $2n$-manifold, Topology Appl. 122:3 (2002) 581--591.}

\newcommand{\jvz}{\bibitem[JVZ]{JVZ} D. Joji\'c, S. T. Vre\'cica, R. T. \v Zivaljevi\' c,
Topology and combinatorics of 'unavoidable complexes', arXiv:1603.08472v1.}



\newcommand{\kalai}{\bibitem[Ka]{Ka} G. Kalai, From Oberwolfach: The Topological Tverberg Conjecture is False, `Combinatorics and more' blog post, February 6, 2015, \url{gilkalai.wordpress.com}}

\newcommand{\kh}{\bibitem[Kh]{Kh} \emph{А.И. Храбров.} Руководство по чтению лекций
\url{http://vm.tstu.tver.ru/topics/pdf_tests/lection.pdf}}

\newcommand{\kho}{\bibitem[Kho]{Kho} \emph{N. Khoroshavkina.} A simple characterization of graphs of cutwidth 2, arXiv:1811.06716.}

\newcommand{\kkrot}{\bibitem[KKR]{KKR} \emph{K. Kawarabayashi, Y. Kobayashi and B. Reed.} The disjoint paths problem in quadratic time, J. of Comb. Theory, Ser. B, 102:2 (2012), 424--435.}

\newcommand{\kmsth}{\bibitem[KM63]{KM63} \emph{M. A. Kervaire and J. W. Milnor,} Groups of homotopy spheres. I,  Ann. of Math. (2) 77 (1963), 504-537.}

\newcommand{\kozeru}{\bibitem[Ko18]{Ko18} * \emph{Е. Колпаков.}
Доказательство теоремы Радона при помощи понижения размерности, Мат. Просвещение, 23 (2018), arXiv:1903.11055.}

\newcommand{\koze}{\bibitem[Ko18]{Ko18} * \emph{E. Kolpakov.}
A proof of Radon Theorem via lowering of dimension, Mat. Prosveschenie, 23 (2018), arXiv:1903.11055.}

\newcommand{\ko}{\bibitem[Ko]{Ko} \emph{E. Kolpakov.} A `converse' to the Constraint Lemma, arXiv:1903.08910.}

\newcommand{\koon}{\bibitem[Ko19]{Ko19} \emph{E. Kogan.} Linking of three triangles in 3-space, arXiv:1908.03865.}

\newcommand{\koto}{\bibitem[Ko21]{Ko21} \emph{E. Kogan.} On the rank of $\Z_2$-matrices with free entries on the diagonal, arXiv:2104.10668.}

\newcommand{\koee}{\bibitem[Ko88]{Ko88} \emph{U. Koschorke.} Link maps and the geometry of their invariants,
Manuscripta Math. 61:4 (1988) 383--415.}

\newcommand{\kono}{\bibitem[Ko91]{Ko91} \emph{U. Koschorke.} Link homotopy with many components,
Topology 30:2 (1991) 267--281.}

\newcommand{\kons}{\bibitem[Ko97]{Ko97} \emph{U. Koschorke.} A generalization of Milnor's $\mu$-invariants to higher-dimensional link maps, Topology 36:2 (1997) 301--324.}


\newcommand{\kps}{\bibitem[KPS]{KPS} * \emph{A. Kaibkhanov, D. Permyakov and A. Skopenkov.}
Realization of graphs with rotation, \url{http://www.turgor.ru/lktg/2005/3/index.htm}.}

\newcommand{\krzz}{\bibitem[Kr00]{Kr00} \emph{V. S. Krushkal.} Embedding obstructions and 4-dimensional thickenings of 2-complexes, Proc. Amer. Math. Soc. 128:12 (2000) 3683--3691. arXiv:math/0004058. }

\newcommand{\ksnn}{\bibitem[KS99]{KS99} * \emph{П. Кожевников и А. Скопенков.} Узкие деревья на плоскости, Мат. Образование. 1999. 2-3. С.~126-131.}

\newcommand{\kstz}{\bibitem[KS20]{KS20} \emph{R. Karasev and A. Skopenkov.}
Some `converses' to intrinsic linking theorems, Discr. Comp. Geom., 70:3 (2023), 921--930, arXiv:2008.02523.}


\newcommand{\ksto}{\bibitem[KS21]{KS21} * \emph{E. Kogan and A. Skopenkov.} A short exposition of the Patak-Tancer theorem on non-embeddability of $k$-complexes in $2k$-manifolds,  arXiv:2106.14010.}

\newcommand{\kstoe}{\bibitem[KS21e]{KS21e} \emph{E. Kogan and A. Skopenkov.}
Embeddings of $k$-complexes in $2k$-manifolds and minimum rank of partial symmetric matrices, arXiv:2112.06636v2.}

\newcommand{\kutt}{\bibitem[Ku23]{Ku23} \emph{W. K\"uhnel.} Generalized Heawood Numbers, The Electronic Journal of Combinatorics, 30:4 (2023) \#P4.17.}



\newcommand{\kuse}{\bibitem[Ku68]{Ku68} * \emph{К. Куратовский.} Топология. Т.~1,~2. М.: Мир, 1969.}

\newcommand{\kunfo}{\bibitem[Ku94]{Ku94} \emph{W. K\"uhnel.} Manifolds in the skeletons of convex polytopes, tightness, and generalized Heawood inequalities. In Polytopes: abstract, convex and computational (Scarborough, ON, 1993), volume 440 of NATO Adv. Sci. Inst. Ser. C Math. Phys. Sci., pp. 241--247. Kluwer
Acad. Publ., Dordrecht, 1994.}


\newcommand{\kunf}{\bibitem[Ku95]{Ku95} * \emph{W. K\"uhnel}, Tight Polyhedral Submanifolds and Tight Triangulations, Lecture Notes in Math. 1612, Springer, 1995.}


\newcommand{\lazz}{\bibitem[La00]{La00} \emph{F. Lasheras.} An obstruction to 3-dimensional thickening,
Proc. Amer. Math. Soc. 2000. 128. P.~893-902.}

\newcommand{\lfma}{\bibitem[LF]{LF} \url{http://www.map.mpim-bonn.mpg.de/Linking_form}}

\newcommand{\lloe}{\bibitem[LL18]{LL18} \emph{A.S. Levine and T. Lidman.} Simply connected, spineless 4-manifolds, Forum of Math., Sigma, 7 (2019) e14, 1--11, arxiv:1803.01765.}

\newcommand{\lo}{\bibitem[Lo]{Lo} M.~de~Longueville. Notes on the topological Tverberg theorem.
Discrete Math.  247 (2002), no.~1--3, 271--297.
(The paper first appeared in
Discrete Math. 241 (2001) 207--233, but the original version suffered from serious publisher's typesetting errors.)}

\newcommand{\loot}{\bibitem[Lo13]{Lo13} \emph{M. de Longueville.} A course in topological combinatorics. Universitext. Springer, New York (2013).}

\newcommand{\lssn}{\bibitem[LS69]{LS69} \emph{W. B. R. Lickorish and L. C. Siebenmann.}
Regular neighborhoods and the stable range,  Trans. Amer. Math. Soc.. 1969. 139. P.~207-230.}

\newcommand{\lsne}{\bibitem[LS98]{LS98} \emph{L. Lovasz and A. Schrijver,}
A Borsuk theorem for antipodal links and a spectral characterization of linklessly embeddable graphs, Proc. Amer. Math. Soc. 126:5 (1998), 1275-1285.}

\newcommand{\ltof}{\bibitem[LT14]{LT14} \emph{E. Lindenstrauss and M. Tsukamoto,} Mean dimension and an embedding problem: an example, Israel J. Math. 199 (2014).}


\newcommand{\lyzf}{\bibitem[LY04]{LY04} * \emph{Y. Lin and A. Yang,} On 3-cutwidth critical graphs, Discrete Mathematics, 275 (2004), 339--346.}

\newcommand{\lz}{\bibitem[LZ]{LZ} * \emph{S. Lando and A. Zvonkin.} Embedded Graphs. Springer.}


\newcommand{\maez}{\bibitem[Ma80]{Ma80} * R. Mandelbaum, {\em Four-Dimensional Topology: An introduction},
Bull. Amer. Math. Soc. (N.S.) 2 (1980) 1-159.}

\newcommand{\mast}{\bibitem[Ma73]{Ma73} \emph{С. В. Матвеев.} Специальные остовы кусочно-линейных многообразий, Мат. Сборник. 1973. 92. С.~282-293.}

\newcommand{\maste}{\bibitem[Ma73]{Ma73} \emph{S. V. Matveev.} Special skeletons of PL manifolds (in Russian), Mat. Sbornik. 1973. 92. P.~282-293.}

\newcommand{\manz}{\bibitem[Ma90]{Ma90} \emph{W. S.  Massey.} Homotopy classification of 3-component links of codimension greater than 2, Topol.  Appl. 34 (1990) 269--300.}

\newcommand{\mans}{\bibitem[Ma97]{Ma97} \emph{Yu. Makarychev.} A short proof of Kuratowski's graph planarity criterion, J. of Graph Theory, 25 (1997), 129--131.}

\newcommand{\matns}{\bibitem[Mat97]{Mat97} \emph{J. Matou\v sek.} A Helly-type theorem for unions of convex sets. Discr. Comp. Geom., 18:1 (1997) 1-12.}

\newcommand{\mazt}{\bibitem[Ma03]{Ma03} * \emph{J.~Matou{\v{s}}ek.} Using the {B}orsuk-{U}lam theorem:
Lectures on topological methods in combinatorics and geometry. Springer Verlag, 2008.}


\newcommand{\mazf}{\bibitem[Ma05]{Ma05} \emph{V. Manturov.} A proof of the Vasiliev conjecture on the planarity of singular links, Izv. RAN 2005.}

\newcommand{\metn}{\bibitem[Me29]{Me29} \emph{K. Menger.} \"Uber pl\"attbare Dreiergraphen und Potenzen nicht pl\"attbarer Graphen, Ergebnisse Math. Kolloq., 2 (1929) 30--31.}

\newcommand{\mezf}{\bibitem[Me04]{Me04} \emph{S. Melikhov.} Sphere eversions and realization of mappings, Trudy MIAN 247 (2004) 159-181 (in Russian) arXiv:math.GT/0305158.}

\newcommand{\mezs}{\bibitem[Me06]{Me06} \emph{S. A. Melikhov}, The van Kampen obstruction and its relatives, 	
Proc. Steklov Inst. Math 266 (2009), 142-176 (= Trudy MIAN 266 (2009), 149-183), arXiv:math/0612082.}

\newcommand{\meoo}{\bibitem[Me11]{Me11} \emph{S. A. Melikhov}, Combinatorics of embeddings, arXiv:1103.5457.}

\newcommand{\meos}{\bibitem[Me17]{Me17} \emph{S. Melikhov,} Gauss type formulas for link map invariants, arXiv:1711.03530.}

\newcommand{\meoe}{\bibitem[Me18]{Me18} \emph{S. A. Melikhov,} A triple-point Whitney trick, J. Topol. Anal., 2018, 1--6. arXiv:2210.04016.}


\newcommand{\metz}{\bibitem[Me20]{Me20} \emph{S. A. Melikhov,} Topological isotopy and Cochran's derived invariants, in `Topology, Geometry, and Dynamics: Rokhlin Memorial', Contemporary Mathematics, 772, AMS, Providence, RI, 2021. arXiv:2011.01409.}

\newcommand{\mett}{\bibitem[Me22]{Me22} \emph{S. A. Melikhov,} Embeddability of joins and products of polyhedra, Topol. Methods in Nonlinear Analysis, 60:1 (2022), 185-201. arXiv:2210.04015.}

\newcommand{\miff}{\bibitem[Mi54]{Mi54} \emph{J. Milnor,} Link groups, Ann. of Math. 59 (1954), 177--195.}

\newcommand{\miso}{\bibitem[Mi61]{Mi61} \emph{J. Milnor,} A procedure for killing homotopy groups of differentiable manifolds, Proc. Sympos. Pure Math, Vol. III (1961), 39--55.}

\newcommand{\mins}{\bibitem[Mi97]{Mi97} \emph{P. Minc.} Embedding simplicial arcs into the plane, Topol. Proc. 1997. 22. 305--340.}


\newcommand{\adnsvr}{\bibitem[MNS]{MNS} * \emph{А. Мирошников, О. Никитенко и А. Скопенков.} Циклы в графах и в гиперграфах: в направлении теории гомологий, arXiv:2406.16705.}
 
\newcommand{\dmnse}{\bibitem[MNS]{MNS} * \emph{A. Miroshnikov, O. Nikitenko, A. Skopenkov.}
Cycles in graphs and in hypergraphs: towards homology theory (in Russian), arXiv:2406.16705.}

\newcommand{\moss}{\bibitem[Mo77]{Mo77} * \emph{E. E. Moise.} Geometric Topology in Dimensions 2 and 3 (GTM), Springer-Verlag, 1977.}

\newcommand{\moen}{\bibitem[Mo89]{Mo89} \textit{B. Mohar}. An obstruction to embedding graphs in
surfaces. Discrete Math. 78 (1989) 135--142.}

\newcommand{\moze}{\bibitem[Mo08]{Mo08} \textit{T. Moriyama}. An invariant of embeddings of 3–manifolds in 6–manifolds and Milnor's triple linking number, J. Math. Sci. Univ. Tokyo, 18 (2011), 193--237. arXiv:0806.3733.}


\newcommand{\mrst}{\bibitem[MRS+]{MRS+} \emph{A. de Mesmay, Y. Rieck, E. Sedgwick, M. Tancer,}
Embeddability in $\R^3$ is NP-hard. arXiv:1708.07734.}

\newcommand{\mesczs}{\bibitem[MS06]{MS06} \emph{S.A. Melikhov, E.V. Shchepin,} The telescope approach to embeddability of compacta. arXiv:math.GT/0612085.}

\newcommand{\msos}{\bibitem[MS17]{MS17}  \emph{T. Maciazek, A. Sawicki.} Homology groups for particles on one-connected graphs
J. Math. Phys. 58, 062103 (2017). arXiv:1606.03414.}

\newcommand{\mstwof}{\bibitem[MST+]{MST+} \emph{J. Matou\v sek, E. Sedgwick, M. Tancer, U. Wagner}, Embeddability in the 3-sphere is decidable, Journal of the ACM 65:1 (2018) 1--49, arXiv:1402.0815.}


\newcommand{\mtzo}{\bibitem[MT01]{MT01} * \emph{B. Mohar and C. Thomassen.} Graphs on Surfaces.
The John Hopkins University Press, 2001.}

\newcommand{\mtwoz}{\bibitem[MTW10]{MTW10} \emph{J. Matou\v sek, M. Tancer, U. Wagner.} A geometric proof of
the colored Tverberg theorem, Discr. and Comp. Geometry, 47:2 (2012), 245--265. arXiv:1008.5275.}


\newcommand{\mtwoo}{\bibitem[MTW]{MTW} \emph{J. Matou\v sek, M. Tancer, U. Wagner.}
Hardness of embedding simplicial complexes in $\R^d$, J. Eur. Math. Soc. 13:2 (2011), 259--295. arXiv:0807.0336.}



\newcommand{\mwoe}{\bibitem[MW18]{MW18} * \emph{F. Manin, S. Weinberger.} Algorithmic aspects of immersibility and embeddability, Intern. Math. Res. Notices, rnae170. arXiv:1812.09413.}


\newcommand{\mwsn}{\bibitem[MW69]{MW69} * \emph{J. MacWilliams}. Orthogonal matrices over finite fields. Amer. Math. Monthly, 76 (1969) 152--164.}

\newcommand{\mwofo}{\bibitem[MW14]{MW14} \emph{I. Mabillard and U. Wagner.} Eliminating Tverberg Points, I. An Analogue of the Whitney Trick, Proc. of the 30th Annual Symp. on Comp. Geom. (SoCG'14), ACM, New York, 2014, pp. 171--180.}

\newcommand{\mwof}{\bibitem[MW15]{MW15} \emph{I. Mabillard and U. Wagner.}
Eliminating Higher-Multiplicity Intersections, I. A Whitney Trick for Tverberg-Type Problems. arXiv:1508.02349.}


\newcommand{\mwos}{\bibitem[MW16]{MW16} \emph{I. Mabillard and U. Wagner.} Eliminating Higher-Multiplicity Intersections, II. The Deleted Product Criterion in the $r$-Metastable Range. arXiv:1601.00876v2.}

\newcommand{\mwosd}{\bibitem[MW16']{MW16'} \emph{I. Mabillard and U. Wagner.} Eliminating Higher-Multiplicity Intersections, II. The Deleted Product Criterion in the r-Metastable Range,
Proceedings of the 32nd Annual Symposium on Computational Geometry (SoCG'16).}


\newcommand{\neno}{\bibitem[Ne91]{Ne91} \emph{S. Negami.} Ramsey theorems for knots, links and spatial graphs,
Trans. Amer. Math. Soc., 324 (1991), 527--541.}



\newcommand{\nkon}{\bibitem[NKS]{NKS} * \emph{L. T. Nguyen, J. Kim, B. Shim.}
Low-Rank Matrix Completion: A Contemporary Survey. arXiv:1907.11705.}

\newcommand{\noss}{\bibitem[No76]{No76} * \emph{С. П. Новиков.} Топология-1. М.: Наука, 1976. (Итоги науки и техники. ВИНИТИ. Современные проблемы математики. Основные направления, 12).}

\newcommand{\nszn}{\bibitem[NS09]{NS09} \emph{I. Novik and E. Swartz,} Socles of Buchsbaum modules, complexes and posets, Adv. Math. 222 (2009), 2059-2084. arXiv:0711.0783.}

\newcommand{\nwns}{\bibitem[NW97]{NW97} \emph{A. Nabutovsky, S. Weinberger}. Algorithmic aspects of homeomorphism problems. arXiv:math/9707232.}


\newcommand{\omoe}{\bibitem[Om18]{Om18} * \emph{А. Омельченко,} Теория графов. М.: МЦНМО, 2018.}

\newcommand{\orszo}{\bibitem[ORS]{ORS} \emph{A. Onischenko, D. Repov\v s and A. Skopenkov.}
Resolutions of 2-polyhedra by fake surfaces and embeddings into $\R^4$, Contemp. Math.  288 (2001) 396--400.}

\newcommand{\ossf}{\bibitem[OS74]{OS74} \emph{R. P. Osborne and R. S. Stevens.} Group presentations
corresponding to spines of 3-manifolds, I, Amer. J.~Math. 1974. 96. P.~454-471; II, Amer. J.~Math. 1977. 234.
P.~213-243; III, Amer. J.~Math. 1977. 234 P.~245-251.}


\newcommand{\oz}{\bibitem[Oz]{Oz} \emph{M. \"Ozaydin,} Equivariant maps for the symmetric group, unpublished,
\url{http://minds.wisconsin.edu/handle/1793/63829}.}

\newcommand{\panof}{\bibitem[Pan15]{Pan15} \emph{K. Panagiotis.} A note on the topology of irreducible $SO(3)$-manifolds, 	arXiv:1508.06150.}

\newcommand{\paof}{\bibitem[Pa15]{Pa15} \emph{S. Parsa,} On links of vertices in simplicial $d$-complexes embeddable in the Euclidean $2d$-space, Discrete Comput. Geom. 59:3 (2018), 663--679.
This is arXiv:1512.05164v4 up to numbering of sections, theorems etc.; we refer to numbering in arxiv version.
Correction: Discrete Comput. Geom. 64:3 (2020) 227--228.}

\newcommand{\paoe}{\bibitem[Pa18]{Pa18} \emph{S. Parsa,} On links of vertices in simplicial $d$-complexes
embeddable in the euclidean $2d$-space, arXiv:1512.05164v6.}

\newcommand{\patz}{\bibitem[Pa20]{Pa20} \emph{S. Parsa,} On links of vertices in simplicial $d$-complexes
embeddable in the euclidean $2d$-space, arXiv:1512.05164v8.}


\newcommand{\patzl}{\bibitem[Pa20]{Pa20} \emph{S. Parsa,}
Correction to: On the Links of Vertices in Simplicial $d$-Complexes Embeddable in the Euclidean $2d$-Space,
Discrete Comput. Geom. 64:3 (2020) 227--228.}

\newcommand{\patza}{\bibitem[Pa20]{Pa20} \emph{S. Parsa,} On the Smith classes, the van Kampen obstruction and embeddability of $[3]*K$, arXiv:2001.06478.}

\newcommand{\patzb}{\bibitem[Pa20b]{Pa20b} \emph{S. Parsa,} On the embeddability of $[3]*K$, arXiv:2001.06506.}

\newcommand{\pato}{\bibitem[Pa21]{Pa21} \emph{S. Parsa,} Instability of the Smith index under joins and applications to embeddability, Trans. Amer. Math. Soc. 375 (2022), 7149--7185, arXiv:2103.02563.}

\newcommand{\pak}{\bibitem[Pa]{Pa} * \emph{I. Pak}, Lectures on Discrete and Polyhedral Geometry, \url{http://www.math.ucla.edu/~pak/geompol8.pdf}.}

\newcommand{\peze}{\bibitem[Pe08]{Pe08} \emph{Д. Пермяков.} Классификация погружений графов в плоскость,
Вестник МГУ, сер.1, 2008, N5, 55-56.}

\newcommand{\peos}{\bibitem[Pe16]{Pe16} \emph{Д. Пермяков.} Матем. сб., 207:6 (2016),  93--112.}

\newcommand{\pest}{\bibitem[Pe72]{Pe72} * \emph{B. B. Peterson.} The Geometry of Radon's Theorem, Amer. Math. Monthly 79 (1972), 949-963.}


\newcommand{\prnf}{\bibitem[Pr95]{Pr95} * \emph{V. V. Prasolov.} Intuitive topology. Amer. Math. Soc., Providence, R.I., 1995.}

\newcommand{\prnfr}{\bibitem[Pr95]{Pr95} * \emph{В. В. Прасолов.} Наглядная топология. М.: МЦНМО, 1995.}


\newcommand{\przs}{\bibitem[Pr06]{Pr06} * \emph{V. V. Prasolov.}
Elements of Combinatorial and Differential Topology, 2006, GSM 74, Amer. Math. Soc., Providence, RI.}

\newcommand{\przsru}{\bibitem[Pr04]{Pr04} * \emph{В. В. Прасолов.}
Элементы комбинаторной и дифференциальной топологии. М.: МЦНМО, 2004. \url{http://www.mccme.ru/prasolov}.}

\newcommand{\przse}{\bibitem[Pr07]{Pr07} * \emph{V. V. Prasolov.} Elements of homology theory. 2007, GSM 74, Amer. Math. Soc., Providence, RI.}


\newcommand{\przseru}{\bibitem[Pr06]{Pr06} * \emph{В. В. Прасолов.} Элементы теории гомологий. М.: МЦНМО, 2006.}


\newcommand{\psns}{\bibitem[PS96]{PS96} * \emph{V. V. Prasolov, A. B. Sossinsky } Knots, Links, Braids, and 3-manifolds. Amer. Math. Soc. Publ., Providence, R.I., 1996.}


\newcommand{\pszf}{\bibitem[PS05]{PS05} * \emph{В. В. Прасолов и М. Б. Скопенков.}
Рамсеевская теория зацеплений, Мат. Просвещение. 2005. 9. С.~108--115.}

\newcommand{\pszfen}{\bibitem[PS05]{PS05} * \emph{V. V. Prasolov and M.B. Skopenkov.}
Ramsey link theory, Mat, Prosvescheniye, 9 (2005), 108--115.}

\newcommand{\psoo}{\bibitem[PS11]{PS11} \emph{Y. Ponty and C. Saule.} A combinatorial framework for designing (pseudoknotted) RNA algorithms, Proc. of the 11th Intern. Workshop on Algorithms in Bioinformatics, WABI'11, 250--269.}


\newcommand{\pstz}{\bibitem[PS20]{PS20} \emph{S. Parsa and A. Skopenkov.} On embeddability of joins and their `factors', Topol. Appl., 326 (2023) 108409, arXiv:2003.12285.}



\newcommand{\psszn}{\bibitem[PSS]{PSS} \emph{M. J. Pelsmajer, M. Schaefer and D. Stasi.} Strong Hanani-Tutte on the projective plane. SIAM J. Discrete Math., 23:3 (2009) 1317--1323.}

\newcommand{\psszs}{\bibitem[PSS]{PSS} \emph{M. J. Pelsmajer, M. Schaefer, and D. \v Stefankovi\v c.}
Removing even crossings. J. Combin. Theory Ser. B, 97(4):489–500, 2007.}

\newcommand{\pton}{\bibitem[PT19]{PT19} \emph{P. Pat\'ak and M. Tancer.} Embeddings of $k$-complexes into $2k$-manifolds. Discrete Comput. Geom. 71 (2024), 960--991. arXiv:1904.02404.}

\newcommand{\pw}{\bibitem[PW]{PW} \emph{I. Pak, S. Wilson}, G\lowercase{EOMETRIC REALIZATIONS OF POLYHEDRAL COMPLEXES}, \linebreak \url{http://www.math.ucla.edu/~pak/papers/Fary-full31.pdf}.}


\newcommand{\razf}{\bibitem[RA05]{RA05} * \emph{J. L. Ram\'irez Alfons\'in.} Knots and links in spatial graphs: a survey. Discrete Math., 302 (2005), 225--242.}

\newcommand{\rep}{\bibitem[Rep]{Rep} Referee's report on the paper ``Some `converses' to intrinsic linking theorems', \url{https://www.mccme.ru/circles/oim/materials/ksreport.pdf}}

\newcommand{\rnoo}{\bibitem[RN11]{RN11} * \emph{R. L. Ricca, B. Nipoti.} Gauss' linking number revisited.
J. of Knot Theory and Its Ramif. 20:10 (2011) 1325--1343. \url{https://www.maths.ed.ac.uk/~v1ranick/papers/ricca.pdf} .}

\newcommand{\rrstz}{\bibitem[RRS]{RRS} * \emph{V. Retinskiy, A. Ryabichev and A. Skopenkov.}
Motivated exposition of the proof of the Tverberg Theorem (in Russian).
Mat. Prosveschenie, 27 (2021), 166--169. arXiv:2008.08361.}


\newcommand{\rssec}{\bibitem[RS68]{RS68} \emph{C. P. Rourke and B. J. Sanderson,} Block bundles II, Ann. of Math. (2), 87 (1968) 431--483.}

\newcommand{\rsst}{\bibitem[RS72]{RS72} * \emph{C. P. Rourke and B. J. Sanderson,}
\newblock Introduction to Piecewise-Linear Topology,
\newblock \emph{Ergebn.\ der Math.} 69, Springer-Verlag, Berlin, 1972.}

\newcommand{\rsstr}{\bibitem[RS72]{RS72} * \emph{К. П. Рурк и Б. Дж. Сандерсон.} Введение в кусочно-линейную топологию, Москва. Мир. 1974.}

\newcommand{\rsns}{\bibitem[RS96]{RS96} * \emph{D. Repov\v s and A. B. Skopenkov.}
Embeddability and isotopy of polyhedra in Euclidean spaces,
Proc. of the Steklov Inst. Math. 1996. 212. P.~173-188.}

\newcommand{\rsne}{\bibitem[RS98]{RS98} \emph{D. Repov\v s and A. B. Skopenkov.}
A deleted product criterion for approximability of a map by embeddings, Topol. Appl. 1998. 87 P.~1-19.}

\newcommand{\rsnn}{\bibitem[RS99]{RS99} * \emph{D. Repov\v s and A. B. Skopenkov.} New results on embeddings of polyhedra and manifolds into Euclidean spaces,
Russ. Math. Surv. 54:6 (1999), 1149--1196.}


\newcommand{\rsnnd}{\bibitem[RS99']{RS99'} * \emph{Д. Реповш и А. Скопенков.}
Кольца Борромео и препятствия к вложимости, Труды МИРАН. 1999. 225. С.~331-338.}

\newcommand{\rszz}{\bibitem[RS00]{RS00} \emph{D. Repov\v s and A. Skopenkov.} Cell-like resolutions of polyhedra by special ones,  Colloq. Math. 2000. 86:2. P. 231--237.}

\newcommand{\rszzd}{\bibitem[RS00']{RS00'} * \emph{Д. Реповш и А. Скопенков.} Характеристические классы для начинающих, Мат. Просвещение. 2000. 4. С.~151-176.}

\newcommand{\rszo}{\bibitem[RS01]{RS01} \emph{D. Repovs and A. Skopenkov.} On contractible $n$-dimensional compacta, non-embeddable into $\R^{2n}$, Proc. Amer. Math. Soc. 129 (2001) 627--628.}

\newcommand{\rszt}{\bibitem[RS02]{RS02} * \emph{Д. Реповш и А. Скопенков.} Теория препятствий для начинающих,
Мат. Просвещение. 2002. 6. C.~60-77.}

\newcommand{\rszf}{\bibitem[RS04]{RS04} \emph{N. Robertson and P. Seymour.} Graph Minors. XX. Wagner's conjecture, J. of Comb. Theory, B, 92:2 (2004) 325--357.}

\newcommand{\rssnf}{\bibitem[RSS]{RSS95} \emph{D. Repov\v s, A. B. Skopenkov  and E. V. \v S\v cepin.}
On uncountable collections of continua and their span, Colloq. Math. 1995. 69:2. P.~289-296.}

\newcommand{\rssnfd}{\bibitem[RSS']{RSS95'} \emph{D. Repov\v s, A. B. Skopenkov and E. V \v S\v cepin.}
On embeddability of $X\times I$ into Euclidean space, Houston J.~Math. 1995. 21. P.~199-204.}

\newcommand{\rssz}{\bibitem[RSS+]{RSSZ} * \emph{A. Rukhovich, A. Skopenkov, M. Skopenkov, A. Zimin},
Realizability of hypergraphs, \url{https://www.turgor.ru/lktg/2013/1/1-1en.pdf} .}


\newcommand{\rstnt}{\bibitem[RST']{RST93} \emph{N. Robertson, P. Seymour and R. Thomas}, Linkless embeddings of graphs in 3-space, Bull. of the Amer. Math. Soc., 21 (1993) 84--89.}

\newcommand{\rstno}{\bibitem[RST]{RST91} * \emph{N. Robertson, P. Seymour and R. Thomas}, A survey of
linkless embeddings, Graph Structure Theory (Seattle, WA, 1991), Contemp. Math. 147, (1993) 125--136.}


\newcommand{\rwzl}{\bibitem[RWZ+]{RWZ+} \emph{Y. Ren, C. Wen, S. Zhen, N. Lei, F. Luo, D.X. Gu},
Characteristic class of isotopy for surfaces, J. Syst. Sci. Complex. 33 (2020) 2139--2156.}


\newcommand{\saeo}{\bibitem[Sa81]{Sa81} \emph{H. Sachs.} On spatial representation of finite graphs,
in: Finite and infinite sets (Eger, 1981), 649--662, Colloq. Math. Soc. Janos Bolyai, 37, North-Holland, Amsterdam, 1984.}

\newcommand{\sano}{\bibitem[Sa91]{Sa91} \emph{K. S. Sarkaria.}
A one-dimensional Whitney trick and Kuratowski's graph planarity criterion, Israel J.~Math. 73 (1991), 79--89.} 


\newcommand{\sanov}{\bibitem[Sa91g]{Sa91g} \emph{K. S. Sarkaria.} A generalized Van Kampen-Flores theorem, Proc. Amer. Math. Soc. 111 (1991), 559--565.}

\newcommand{\sant}{\bibitem[Sa92]{Sa92} \emph{K. S. Sarkaria.} Tverberg’s theorem via number fields. Israel J. Math., 79:317–320, 1992.}

\newcommand{\sann}{\bibitem[Sa99]{Sa99} O. Saeki {\em On punctured 3-manifolds in 5-sphere}, Hiroshima Math. J. 29 (1999) 255--272.}


\newcommand{\sazz}{\bibitem[Sa00]{Sa00} \emph{K. S. Sarkaria.} Tverberg partitions and Borsuk-Ulam theorems. Pacific J. Math., 196:1 (2000) 231--241.}

\newcommand{\sczf}{\bibitem[Sc04]{Sc04} \emph{T. Sch\"oneborn.} On the Topological Tverberg Theorem, arXiv:math/0405393.}


\newcommand{\scot}{\bibitem[Sc13]{Sc13} * \emph{M. Schaefer.} Hanani-Tutte and related results. In Geometry --- intuitive, discrete, and convex, Bolyai Soc. Math. Stud., 24 (2013), 259--299.
\url{http://ovid.cs.depaul.edu/documents/htsurvey.pdf} }


\newcommand{\sctz}{\bibitem[Sc20]{Sc20} \emph{M. Schaefer.} The Graph Crossing Number and
its Variants: A Survey. The Electr. J. of Comb. (2020), DS21, \url{https://www.combinatorics.org/files/Surveys/ds21/ds21v5-2020.pdf}}


\newcommand{\scef}{\bibitem[Sc84]{Sc84} \emph{E.~V.~\v S\v cepin.} Soft mappings of manifolds, Russian Math. Surveys, 39:5 (1984).}

\newcommand{\shfs}{\bibitem[Sh57]{Sh57} \emph{A. Shapiro,} Obstructions to the embedding of a complex in a Euclidean space, I, The first obstruction, Ann. Math. 66 (1957), 256--269.}


\newcommand{\shen}{\bibitem[Sh89]{Sh89} * \emph{Ю. А. Шашкин,} Неподвижные точки, М., Наука, 1989.}

\newcommand{\shoe}{\bibitem[Sh18]{Sh18} * \emph{S. Shlosman},  Topological Tverberg Theorem: the proofs and the counterexamples, Russian Math. Surveys, 73:2 (2018), 175–182. arXiv:1804.03120.}

\newcommand{\sisn}{\bibitem[Si69]{Si69} \emph{K. Sieklucki.} Realization of mappings, Fund. Math. 1969. 65. P.~325-343.}

\newcommand{\sios}{\bibitem[Si16]{Si16} \emph{S. Simon,} Average-Value Tverberg Partitions via Finite Fourier Analysis, Israel J. Math., 216 (2016), 891-904, arXiv:1501.04612.}



\newcommand{\sknf}{\bibitem[Sk94]{Sk94} \emph{А. Скопенков.} Геометрическое доказательство теоремы
Нойвирта об утолщаемости 2-мерных полиэдров, Math. Notes. 1995. 58:5. P.~1244-1247.}


\newcommand{\skns}{\bibitem[Sk97]{Sk97} \emph{A. Skopenkov,} On the deleted product criterion for embeddability of manifolds in $\R^m$, Comment. Math. Helv. 72 (1997), 543--555.}

\newcommand{\skne}{\bibitem[Sk98]{Sk98} \emph{A. B. Skopenkov.} On the deleted product criterion for embeddability in $\R^m$, Proc. Amer. Math. Soc., 126:8 (1998), 2467-2476.}

\newcommand{\skzz}{\bibitem[Sk00]{Sk00} \emph{A. Skopenkov,} On the generalized Massey--Rolfsen invariant for link maps, Fund. Math. 165 (2000), 1--15.}

\newcommand{\skzt}{\bibitem[Sk02]{Sk02} \emph{A. Skopenkov,} On the Haefliger-Hirsch-Wu invariants for embeddings and immersions, Comment. Math. Helv. 77 (2002), 78--124.}

\newcommand{\skzth}{\bibitem[Sk03]{Sk03} \emph{M. Skopenkov,} Embedding products of graphs into Euclidean spaces,
Fund. Math. 179 (2003),~191--198, arXiv:0808.1199.}

\newcommand{\skzthd}{\bibitem[Sk03']{Sk03'} \emph{M. Skopenkov,} On approximability by embeddings of cycles in the plane, Topol. Appl. 134 (2003),~1--22, arXiv:0808.1187.}

\newcommand{\skzf}{\bibitem[Sk05]{Sk05} * \emph{A. Skopenkov,}
On the Kuratowski graph planarity criterion, Mat. Prosveschenie, 9 (2005), 116-128. arXiv:0802.3820.}


\newcommand{\skzs}{\bibitem[Sk05i]{Sk05i} \emph{A. Skopenkov,} A new invariant and parametric connected sum of embeddings, Fund. Math. 197 (2007) 253--269. arxiv:math/0509621.}

\newcommand{\skzei}{\bibitem[Sk05]{Sk05} \emph{A.  Skopenkov,} A classification of smooth embeddings of
4-manifolds in 7-space, I, Topol. Appl., 157 (2010) 2094--2110. arXiv:math/0512594.}

\newcommand{\skze}{\bibitem[Sk06]{Sk06} * \emph{A. Skopenkov,} Embedding and knotting of manifolds in Euclidean spaces, London Math. Soc. Lect. Notes, 347 (2008) 248--342. arXiv:math/0604045.}

\newcommand{\skzsi}{\bibitem[Sk06']{Sk06'} \emph{A. Skopenkov,} A classification of smooth embeddings of 3-manifolds in 6-space, Math. Zeitschrift, 260:3 (2008) 647--672. arxiv:math/0603429.}

\newcommand{\skozp}{\bibitem[Sk08]{Sk08} \emph{A.  Skopenkov,} Embeddings of $k$-connected $n$-manifolds into
$\R^{2n-k-1}$. arxiv:math/0812.0263; earlier version published in Proc. Amer. Math. Soc., 138 (2010) 3377--3389.}

\newcommand{\skoz}{\bibitem[Sk10]{Sk10} * \emph{А. Скопенков,} Вложения в плоскость графов с вершинами степени 4,
Мат. просвещение, 21 (2017), arXiv:1008.4940.}

\newcommand{\skoo}{\bibitem[Sk11]{Sk11} \emph{M. Skopenkov,} When is the set of embeddings finite up to isotopy? Intern. J. Math. 26:7 (2015), 28 pp. arXiv:1106.1878.}

\newcommand{\skofo}{\bibitem[Sk14]{Sk14} \emph{A. Skopenkov,} How do autodiffeomorphisms act on embeddings, Proc. A of the Royal Society of Edinburgh, 148:4 (2018), 835--848. arXiv:1402.1853.}

\newcommand{\sks}{\bibitem[Sk14]{Sk14} * \emph{A. Skopenkov,} Realizability of hypergraphs and intrinsic linking  theory, Mat. Prosveschenie, 32 (2024), 125--159, arXiv:1402.0658.}

\newcommand{\sksr}{\bibitem[Sk14]{Sk14} * \emph{А. Скопенков,} Реализуемость гиперграфов и неотъемлемая зацепленность, Мат. просвещение, 32 (2024), 125--159. arXiv:1402.0658.}


\newcommand{\skof}{\bibitem[Sk15]{Sk15} * \emph{А. Скопенков,} Алгебраическая топология с геометрической точки зрения, Москва, МЦНМО, 2015 (1е издание).}

\newcommand{\skofe}{\bibitem[Sk15]{Sk15} * \emph{A. Skopenkov,} Algebraic Topology From Geometric Viewpoint (in Russian), MCCME, Moscow, 2015 (1st edition). }

\newcommand{\skofel}{\bibitem[Sk15e]{Sk15e} * \emph{А. Скопенков,} Алгебраическая топология
с геометрической точки зрения, эл. версия, \url{http://www.mccme.ru/circles/oim/home/combtop13.htm\#photo}}


\newcommand{\skotzr}{\bibitem[Sk20]{Sk20} * \emph{А. Скопенков,} Алгебраическая топология с геометрической точки зрения, Москва, МЦНМО, 2020 (2е издание).
Обновляемая версия части книги: \url{http://www.mccme.ru/circles/oim/obstruct.pdf}}

\newcommand{\skotz}{\bibitem[Sk20]{Sk20} * \emph{A. Skopenkov,} Algebraic Topology From Geometric Standpoint (in Russian), MCCME, Moscow, 2020 (2nd edition).
Update of a part: \url{http://www.mccme.ru/circles/oim/obstruct.pdf} .
Part of the English translation: \url{https://www.mccme.ru/circles/oim/obstructeng.pdf}.}



\newcommand{\skofp}{\bibitem[Sk15]{Sk15} \emph{A. Skopenkov,} Classification of knotted tori,
Proc. A of the Royal Soc. of Edinburgh, 150:2 (2020), 549-567. Full version: arXiv:1502.04470.}


\newcommand{\skos}{\bibitem[Sk16]{Sk16} * \emph{A. Skopenkov,} A user's guide to the topological Tverberg Conjecture, arXiv:1605.05141v5. Abridged earlier published version: Russian Math. Surveys, 73:2 (2018), 323--353.}



\newcommand{\skosd}{\bibitem[Sk16']{Sk16'} * \emph{A. Skopenkov,} Stability of intersections of graphs in the plane and the van Kampen obstruction, Topol. Appl. 240(2018) 259--269, arXiv:1609.03727.}


\newcommand{\skosc}{\bibitem[Sk16c]{Sk16c} * \emph{A. Skopenkov,}  Embeddings in Euclidean space: an introduction to their classification, to appear in Boll. Man. Atl. 
\url{http://www.map.mpim-bonn.mpg.de/Embeddings_in_Euclidean_space:_an_introduction_to_their_classification}}

\newcommand{\skosie}{\bibitem[Sk16e]{Sk16e} * \emph{A. Skopenkov,} Embeddings just below the stable range: classification, to appear in Boll. Man. Atl.
\url{http://www.map.mpim-bonn.mpg.de/Embeddings_just_below_the_stable_range:_classification}}

\newcommand{\skost}{\bibitem[Sk16t]{Sk16t} * \emph{A. Skopenkov,} 3-manifolds in 6-space, to appear in Boll. Man. Atl. \url{http://www.map.mpim-bonn.mpg.de/3-manifolds_in_6-space}.}

\newcommand{\skosf}{\bibitem[Sk16f]{Sk16f} * \emph{A. Skopenkov,} 4-manifolds in 7-space, to appear in Boll. Man. Atl. \url{http://www.map.mpim-bonn.mpg.de/4-manifolds_in_7-space}.}

\newcommand{\skosh}{\bibitem[Sk16h]{Sk16h} * \emph{A. Skopenkov,} High codimension links, to appear in Boll. Man. Atl.
\linebreak
\url{http://www.map.mpim-bonn.mpg.de/High_codimension_links}.}

\newcommand{\skosi}{\bibitem[Sk16i]{Sk16i} * \emph{A. Skopenkov,} Isotopy, submitted to Boll. Man. Atl.
\url{http://www.map.mpim-bonn.mpg.de/Isotopy}.}

\newcommand{\skosk}{\bibitem[Sk16k]{Sk16k} * \emph{A. Skopenkov,} Knotted tori,
\url{http://www.map.mpim-bonn.mpg.de/Knotted_tori}.}

\newcommand{\skoss}{\bibitem[Sk16s]{Sk16s} * \emph{A. Skopenkov,} Knots, i.e. embeddings of spheres,
\linebreak
\url{http://www.map.mpim-bonn.mpg.de/Knots,_i.e._embeddings_of_spheres}.}

\newcommand{\skose}{\bibitem[Sk17]{Sk17} \emph{A. Skopenkov,}
Eliminating higher-multiplicity intersections in the metastable dimension range. arXiv:1704.00143.}

\newcommand{\skosed}{\bibitem[Sk17v]{Sk17v} * \emph{A. Skopenkov,}
On van Kampen-Flores, Conway-Gordon-Sachs and Radon theorems,  arXiv:1704.00300.}

\newcommand{\sk}{\bibitem[Sk17o]{Sk17o} \emph{A. Skopenkov,} On the metastable Mabillard-Wagner conjecture.  arXiv:1702.04259.}

\newcommand{\skmos}{\bibitem[Sk17d]{Sk17d} \emph{M. Skopenkov}. Discrete field theory: symmetries and conservation laws, arXiv:1709.04788.}

\newcommand{\skoe}{\bibitem[Sk18]{Sk18} * \emph{A. Skopenkov.} Invariants of graph drawings in the plane.
Arnold Math. J., 6 (2020) 21--55; full version: arXiv:1805.10237.}


\newcommand{\skoer}{\bibitem[Sk18]{Sk18} * \emph{А. Скопенков,} Инварианты изображений графов на плоскости,
Мат. просвещение, 31 (2023), 74-127. arXiv:1805.10237.}


\newcommand{\sktthd}{\bibitem[Sk23']{Sk23'} * \emph{A. Skopenkov.} Invariants of graph drawings in the plane (in Russian). Mat. Prosveschenie, 31 (2023), 74-127. arXiv:1805.10237.}

\newcommand{\skoeo}{\bibitem[Sk18o]{Sk18o} * \emph{A. Skopenkov.} A short exposition of S. Parsa's theorems on intrinsic linking and non-realizability. Discr. Comp. Geom. 65:2 (2021), 584--585; full version:  arXiv:1808.08363.}


\newcommand{\skona}{\bibitem[Sk19]{Sk19} * \emph{A. Skopenkov,} A short exposition of the Levine-Lidman example of spineless 4-manifolds, arXiv:1911.07330.}

\newcommand{\sktze}{\bibitem[Sk21m]{Sk21m} * \emph{A. Skopenkov.} Mathematics via Problems. Part 1: Algebra. Amer. Math. Soc., Providence, 2021. Preliminary version: \url{https://www.mccme.ru/circles/oim/algebra_eng.pdf}}

\newcommand{\sktz}{\bibitem[Sk20u]{Sk20u} * \emph{A. Skopenkov.} A user's guide to basic knot and link theory,
in: Topology, Geometry, and Dynamics, Contemporary Mathematics, vol. 772, Amer. Math. Soc., Providence, RI, 2021, pp. 281--309.
Russian version: Mat. Prosveschenie 27 (2021), 128--165. arXiv:2001.01472.}

\newcommand{\sktzru}{\bibitem[Sk20u]{Sk20u} * \emph{А. Скопенков.} Основы теории узлов и зацеплений для пользователя, Мат. просвещение, 27 (2021), 128--165. arXiv:2001.01472.}

\newcommand{\sktzo}{\bibitem[Sk20o]{Sk20o} \emph{A. Skopenkov.} On some results of S. Abramyan and T. Panov, arXiv:2005.11152.}

\newcommand{\sktzr}{\bibitem[Sk20e]{Sk20e} * \emph{A. Skopenkov.}
Extendability of simplicial maps is undecidable, Discr. Comp. Geom., 69:1 (2023), 250--259, arXiv:2008.00492.}


\newcommand{\sktzd}{\bibitem[Sk21d]{Sk21d} * \emph{A. Skopenkov.}
On different reliability standards in current mathematical research, arXiv:2101.03745.
More often updated version: \url{https://www.mccme.ru/circles/oim/rese_inte.pdf}.}

\newcommand{\sktt}{\bibitem[Sk22]{Sk22} * \emph{A. Skopenkov.} Invariants of embeddings of 2-surfaces in 3-space,
arXiv:2201.10944.}

\newcommand{\skttn}{\bibitem[Sk22n]{Sk22n} * \emph{A. Skopenkov}, Netflix problem and realization of (hyper)graphs, \url{https://www.mccme.ru/circles/oim/home/netflix20sep.pdf}}

\newcommand{\sktth}{\bibitem[Sk23]{Sk23} \emph{A. Skopenkov.}  To S. Parsa's theorem on embeddability of joins, arXiv:2302.11537.}

\newcommand{\sktf}{\bibitem[Sk24]{Sk24} * \emph{A. Skopenkov.} Double and triple linking numbers in space (in Russian). Mat. Prosveschenie, 33 (2024), 87--132.}

\newcommand{\sktfr}{\bibitem[Sk24]{Sk24} * \emph{А. Скопенков.} Двойные и тройные коэффициенты зацепления в пространстве. Мат. просвещение, 33 (2024), 87--132.}

\newcommand{\sktfb}{\bibitem[Sk24]{Sk24} \emph{A. Skopenkov.} The band connected sum and the second Kirby move for higher-dimensional links (full version), arXiv:2406.15367.}


\newcommand{\sktfe}{\bibitem[Sk24]{Sk24} \emph{A. Skopenkov.}
Embeddings of $k$-complexes in $2k$-manifolds and minimum rank of partial symmetric matrices, arXiv:2112.06636v4.}

\newcommand{\skd}{\bibitem[Sk]{Sk} * \emph{А. Скопенков.} Алгебраическая топология с алгоритмической точки зрения, 
\url{http://www.mccme.ru/circles/oim/algor.pdf}.}

\newcommand{\skde}{\bibitem[Sk]{Sk} * \emph{A. Skopenkov.} Algebraic Topology From Algorithmic Standpoint, draft of a book, mostly in Russian,
\url{http://www.mccme.ru/circles/oim/algor.pdf}.}


\newcommand{\skon}{\bibitem[Skw]{Skw} * \emph{A. Skopenkov.} Whitney trick for eliminating multiple intersections, slides for talks at St Petersburg, Brno, Kiev, Moscow,  \url{https://www.mccme.ru/circles/oim/eliminat_talk.pdf}.}

\newcommand{\skl}{\bibitem[EEF]{EEF} * {\it Proposed by D. Eliseev, A. Enne, M. Fedorov, A. Glebov, N. Khoroshavkina, E. Morozov, A. Skopenkov, R. \v Zivaljevi\'c.}
A user's guide to knot and link theory, \url{https://www.turgor.ru/lktg/2019/3} .}

\newcommand{\skr}{\bibitem[Skr]{Skr} * \emph{A. Skopenkov.} Realizability of hypergraphs, slides for talks,  \url{https://www.mccme.ru/circles/oim/algor1_beamer.pdf}.}


\newcommand{\skt}{\bibitem[Skt]{Skt} * \emph{A. Skopenkov.} Transparent anonymous peer review,
\linebreak
\url{https://www.mccme.ru/circles/oim/home/transp_peer_review.htm} .}

\newcommand{\rslktg}{\bibitem[KRR+]{RRSl} * Towards higher-dimensional combinatorial geometry, presented by
E. Kogan, V. Retinskiy, E. Riabov and A. Skopenkov, \url{https://www.mccme.ru/circles/oim/multicomb.pdf} .}



\newcommand{\sm}{\bibitem[Sm]{Sm} S. Smirnov.}

\newcommand{\sper}{\bibitem[Sp]{Sp} * Sperner's lemma defeats the rental harmony problem, \url{https://www.youtube.com/watch?v=7s-YM-kcKME}.}

\newcommand{\sset}{\bibitem[SS83]{SS83} \emph{Е. В. Щепин, М. А. Штанько.} Спектральный критерий вложимости компактов в евклидовы пространства, Труды Ленинградской Международной Топологической конференции. Л.: Наука, 1983. С.~135-142.}

\newcommand{\ssnt}{\bibitem[SS92]{SS92} \emph{J.~Segal and S.~Spie\.z.} Quasi embeddings and embeddings of polyhedra in $\R^m$,  Topol. Appl., 45 (1992) 275--282.}

\newcommand{\sszt}{\bibitem[SS03]{SS03} \emph{F. W. Simmons and F. E. Su.}
Consensus-halving via theorems of Borsuk-Ulam and Tucker, Math. Social Sciences 45 (2003) 15–25. \url{https://www.math.hmc.edu/~su/papers.dir/tucker.pdf}.}

\newcommand{\ssot}{\bibitem[SS13]{SS13} \emph{M. Schaefer and D. \v Stefankovi\v c.} Block additivity of $\Z_2$-embeddings. In Graph drawing, volume 8242 of Lecture Notes in Comput. Sci., 185--195.
Springer, Cham, 2013. \url{http://ovid.cs.depaul.edu/documents/genus.pdf}}

\newcommand{\sstt}{\bibitem[SS23]{SS23} \emph{A. Skopenkov and O. Styrt,} Embeddability of joinpowers, and minimal rank of partial matrices, arXiv:2305.06339.}

\newcommand{\sssne}{\bibitem[SSS]{SSS} \emph{J. Segal, A. Skopenkov and S. Spie\. z.}
Embeddings of polyhedra in $\R^m$ and the deleted product obstruction, Topol. Appl., 85 (1998), 225-234.}

\newcommand{\sstnf}{\bibitem[SST95]{SST95} \emph{R. S. Simon, S. Spie\. z and H. Toru\'nczyk.}
T\lowercase{HE EXISTENCE OF EQUILIBRIA IN CERTAIN GAMES, SEPARATION FOR FAMILIES OF CONVEX FUNCTIONS
AND A THEOREM OF BORSUK-ULAM TYPE}, Israel J. Math 92 (1995) 1--21.}

\newcommand{\sstzt}{\bibitem[SST02]{SST02} \emph{R. S. Simon, S. Spie\. z and H. Toru\'nczyk.}
E\lowercase{QUILIBRIUM EXISTENCE AND TOPOLOGY IN SOME REPEATED GAMES WITH INCOMPLETE INFORMATION},
Trans. Amer. Math. Soc. 354:12 (2002) 5005-5026.}

\newcommand{\stez}{\bibitem[ST80]{ST80} * {\it H.~Seifert and W.~Threlfall.}
A textbook of topology, v~89 of {\em Pure and Applied Mathematics}.
Academic Press, New York-London, 1980.}


\newcommand{\stzs}{\bibitem[ST07]{ST07} * \emph{А. Скопенков и А. Телишев.}
И вновь о критерии Куратовского планарности графов, Мат. Просвещение, 11 (2007), 159--160.}

\newcommand{\stzse}{\bibitem[ST07]{ST07} * \emph{A. Skopenkov and A. Telishev}, Once again on the Kuratowski graph planarity criterion, Mat. Prosveschenie, 11 (2007), 159-160. arXiv:0802.3820.}

\newcommand{\stos}{\bibitem[ST17]{ST17} \emph{A. Skopenkov  and M. Tancer,}
Hardness of almost embedding simplicial complexes in $\R^d$, Discr. Comp. Geom., 61:2 (2019), 452--463. arXiv:1703.06305.}

\newcommand{\stno}{\bibitem[ST91]{ST91} \emph{S.~Spie\. z and H.~Toru\'nczyk}, Moving compacta in $\R^m$ apart,
Topol. Appl. 41 (1991), 193--204.}

\newcommand{\sttt}{\bibitem[St24]{St24} \emph{M. Starkov,} An example of an `unlinked' set of $2k+3$ points in $2k$-space, arXiv:2402.09002.}

\newcommand{\sunt}{\bibitem[Su]{Su} * \emph{Д. Судзуки.} Основы дзэн-буддизма. Наука дзэн --- ум дзэн. Киев: Преса Украiни. 1992.}

\newcommand{\stwh}{\bibitem[SW]{SW} * \url{http://www.map.mpim-bonn.mpg.de/Stiefel-Whitney_characteristic_classes}}

\newcommand{\sz}{\bibitem[SZ05]{SZ} \emph{T. Sch\"oneborn and G. Ziegler}, The Topological Tverberg Theorem and Winding Numbers, J. Comb. Theory, Ser. A, 112:1 (2005) 82--104, arXiv:math/0409081.}

\newcommand{\szno}{\bibitem[Sz91]{Sz91} \emph{A.~Sz\"ucs,} On the cobordism groups of immersions and embeddings,
Math. Proc. Camb. Phil. Soc., 109 (1991) 343--349.}


\newcommand{\ta}{\bibitem[Ta]{Ta} * Handbook of Graph Drawing and Visualization. ed. by R. Tamassia, CRC Press, 2016.}


\newcommand{\tanf}{\bibitem[Ta95]{Ta95} \emph{K. Taniyama,} Homology classification of spatial embeddings of a graph, Topol. Appl. 65 (1995) 205--228.}

\newcommand{\tazz}{\bibitem[Ta00]{Ta00} \emph{K. Taniyama,} Higher dimensional links in a simplicial complex embedded in a sphere, Pacific Jour. of Math. 194:2 (2000), 465-467.}

\newcommand{\theo}{\bibitem[Th81]{Th81} * \emph{C.~Thomassen,} Kuratowski's theorem, J.~Graph. Theory 5 (1981), 225--242.}

\newcommand{\tooo}{\bibitem[To11]{To11} \emph{Tonkonog D.} Embedding 3-manifolds with boundary into closed 3-manifolds, Topol. Appl. 158 (2011), 1157-1162. arXiv:1003.3029.}


\newcommand{\tsbzf}{\bibitem[TSB]{TSB} \emph{D. M. Thilikos, M. Serna and H. L. Bodlaender},
Cutwidth I: A linear time fixed parameter algorithm, J. of Algorithms, 56:1 (2005), 1--24.}


\newcommand{\tsbzfd}{\bibitem[TSB05']{TSB05'} \emph{D. M. Thilikos, M. Serna and H. L. Bodlaender},
Cutwidth II: , J. of Algorithms, 56:1 (2005), 25--49.}



\newcommand{\umse}{\bibitem[Um78]{Um78} \emph{B. Ummel.} The product of nonplanar complexes does not imbed in 4-space, Trans. Amer. Math. Soc., 242 (1978) 319--328.}




\newcommand{\vant}{\bibitem[Va92]{Va92} * \emph{V.~A.~Vassiliev.} Complements of discriminants of smooth maps: Topology and applications, Amer. Math. Soc., Providence, RI, 1992 (рус. перевод: В. А. Васильев, Топология дополнений к дискриминантам, Фазис, Москва, 1997).}

\newcommand{\val}{\bibitem[Val]{Val} * \url{https://en.wikipedia.org/wiki/Valknut}}


\newcommand{\vi}{\bibitem[Vi]{Vi} * \emph{O. Viro.}
Some integral calculus based on Euler characteristic, Lect. Notes in Math. 1346.}

\newcommand{\vizt}{\bibitem[Vi02]{Vi02} * \emph{Э. Б. Винберг.} Курс алгебры. Москва. Факториал Пресс. 2002.}

\newcommand{\vizteng}{\bibitem[Vi02]{Vi02} * \emph{E. B. Vinberg.} A Course in Algebra. Graduate Studies in Mathematics, vol. 56. 2003.}

\newcommand{\vinhzs}{\bibitem[VINH07]{VINH07} * \emph{О. Я. Виро, О. А. Иванов, Н. Ю. Нецветаев и В. М. Харламов.}
Элементарная топология, МЦНМО. 2007.}

\newcommand{\vktt}{\bibitem[vK32]{vK32} \emph{E.~R.~van~Kampen}, Komplexe in euklidischen R\"aumen, Abh. Math. Sem. Hamburg, 9 (1933) 72--78; Berichtigung dazu, 152--153.}

\newcommand{\kafo}{\bibitem[vK41]{vK41} \emph{E. R. van Kampen,} Remark on the address of S. S. Cairns,
in Lectures in Topology, 311--313, University of Michigan Press, Ann Arbor, MI, 1941.}

\newcommand{\vo}{\bibitem[Vo96]{vo96} \emph{A. Yu. Volovikov,} On a topological generalization of the Tverberg theorem. Math. Notes 59:3 (1996), 324--326.}

\newcommand{\vopns}{\bibitem[Vo96v]{Vo96v} \emph{A. Yu. Volovikov,} On the van Kampen-Flores Theorem.
Math. Notes 59:5 (1996), 477--481.}


\newcommand{\vznt}{\bibitem[VZ93]{VZ93} \emph{A. Vu\v ci\'c and R. T. \v Zivaljevi\'c}, Note on a conjecture of Sierksma, Discr. Comput. Geom. 9 (1993), 339-349.}

\newcommand{\vzzn}{\bibitem[VZ09]{VZ09} \emph{S. T. Vre\'cica and R. T. \v Zivaljevi\'c},  Chessboard complexes
indomitable, J. of Comb. Theory, Ser. A 118:7 (2011), 2157--2166. arXiv:0911.3512.}


\newcommand{\walst}{\bibitem[Wa62]{Wa62} \emph{C.~T.~C.~Wall}, Classification of $(n-1)$-connected $2n$-manifolds, Ann. of Math., 75 (1962) 163--189.}


\newcommand{\wallss}{\bibitem[Wa67]{Wa67} \emph{C.~T.~C.~Wall.} Classification problems in differential topology, IV, Thickenings, Topology 1966. 5. P. 73--94.}

\newcommand{\waldss}{\bibitem[Wa67m]{Wa67m} \emph{F. Waldhausen.} Eine Klasse von 3-dimensional Mannigfaltigkeiten, I. Invent. Math. 1967. 3. P.~308-333.}

\newcommand{\walsz}{\bibitem[Wa70]{Wa70} \emph{C. T. C. Wall,} Surgery on compact manifolds,
1970, Academic Press, London.}

\newcommand{\wess}{\bibitem[We67]{We67} \emph{C.~Weber.} Plongements de poly\`edres dans le domain metastable, Comment. Math. Helv. 42 (1967), 1--27.}

\newcommand{\whit}{\bibitem[Wl]{Wl} * \url{https://en.wikipedia.org/wiki/Whitehead_link}}

\newcommand{\winum}{\bibitem[Wn]{Wn} * \url{https://en.wikipedia.org/wiki/Winding_number}}

\newcommand{\wrss}{\bibitem[Wr77]{Wr77} \emph{P. Wright.} Covering 2-dimensional polyhedra by 3-manifolds spines.
Topology. 16 (1977), 435--439.}

\newcommand{\wufe}{\bibitem[Wu58]{Wu58} \emph{W. T. Wu.} On the realization of complexes in a euclidean space (in Chinese): I, Sci Sinica, 7 (1958) 251--297; II, Sci Sinica, 7 (1958) 365--387; III, Sci Sinica, 8 (1959) 133--150.}

\newcommand{\wufn}{\bibitem[Wu59]{Wu59} \emph{W.~T.~Wu.} On the isotopy of a finite complex in Euclidean space, I, II, Science Record, N.S. 3:8 (1959) 342--347, 348--351.}

\newcommand{\wusf}{\bibitem[Wu65]{Wu65} * \emph{W. T. Wu.} A Theory of Embedding, Immersion and Isotopy of Polytopes in an Euclidean Space. Peking: Science Press, 1965.}


\newcommand{\yann}{\bibitem[Ya99]{Ya99} \emph{Z. Yang.} Computing Equilibria and Fixed Points: The Solution of Nonlinear Inequalities, Kluwer, Springer Science + Business Media, 1990.}


\newcommand{\zesz}{\bibitem[Ze60]{Ze60} \emph{E. C. Zeeman}, Unknotting spheres in five dimensions, Bull. Amer. Math. Soc. 66 (1960) 198.
\linebreak
\url{https://www.ams.org/journals/bull/1960-66-03/S0002-9904-1960-10431-4/S0002-9904-1960-10431-4.pdf}}

\newcommand{\z}{\bibitem[Ze]{Z} * \emph{E. C. Zeeman}, A Brief History of Topology, UC Berkeley, October 27, 1993, On the occasion of Moe Hirsch's 60th birthday, \url{http://zakuski.utsa.edu/~gokhman/ecz/hirsch60.pdf}.}

\newcommand{\zioz}{\bibitem[Zi10]{Zi10} * \emph{D. \v Zivaljevi\'c}, Borromean and Brunnian Rings, \url{http://www.rade-zivaljevic.appspot.com/borromean.html}.}

\newcommand{\zioo}{\bibitem[Zi11]{Zi11} * \emph{G. M. Ziegler}, 3N Colored Points in a Plane, Notices of the Amer. Math. Soc., 58:4 (2011), 550-557.}


\newcommand{\zot}{\bibitem[Zi13]{Z13} \emph{A. Zimin.} Alternative proofs of the Conway-Gordon-Sachs Theorems, arXiv:1311.2882.}

\newcommand{\zss}{\bibitem[ZSS]{ZSS} * Элементы математики в задачах: через олимпиады и кружки к профессии. Сборник под редакцией А. Заславского, А. Скопенкова и М. Скопенкова. М.: МЦНМО, 2018.
Обновляемая версия части книги: \url{http://www.mccme.ru/circles/oim/materials/sturm.pdf}.}

\newcommand{\zu}{\bibitem[Zu]{Zu} \emph{J. Zung.} A non-general-position Parity Lemma,
\url{http://www.turgor.ru/lktg/2013/1/parity.pdf}.}

 

\abms  
\adnsv
 

\biet

\bibitem[CPG]{CPG} \href{https://ru.wikipedia.org/wiki/Прямое_произведение_графов}{https://ru.wikipedia.org/wiki/Прямое\_произведение\_графов}

\botf
\ffen

\bibitem[Fl]{Fl} \href{https://en.wikipedia.org/wiki/Flow_network}{https://en.wikipedia.org/wiki/Flow\_network}

\fhzo
\gdikrs
\hasf
\hcon
 


\msos


\sano
\shen
\skze
\skoer
\skotzr
\skd
\sstt
\vant
\zss

\end{thebibliography}
\end{document}